\documentclass[a4paper,11pt]{article}
\usepackage{amssymb}
\usepackage{amsfonts}
\usepackage{amsmath}
\usepackage{amsthm}
\usepackage{cite}

\newcommand{\s}{s}

\newcommand{\R}{\mathbb{R}}

\newcommand{\N}{\mathbb{N}}

\newcommand{\cuad}{{\sqcap\kern-.68em\sqcup}}

\newcommand{\norm}[1]{\|#1\|}

\newtheorem{definition}{Definition}[section]
\newtheorem{theorem}[definition]{Theorem}
\newtheorem{proposition}[definition]{Proposition}

\newtheorem{lemma}[definition]{Lemma}

\newtheorem{remark}[definition]{Remark}
\newcommand{\bremark}{\begin{remark} \em}
\newcommand{\eremark}{\end{remark} }

\newcommand{\cC}{{\mathcal C}}

\newcommand{\cE}{{\mathcal E}}
\newcommand{\cF}{{\mathcal F}}
\newcommand{\cG}{{\mathcal G}}
\newcommand{\cH}{{\mathcal H}}

\newcommand{\cL}{{\mathcal L}}

\newcommand{\cW}{{\mathcal W}}

\newcommand{\weak}{\rightharpoonup}

\newcommand{\eps}{\varepsilon}

\newcommand{\sOmega}{{\mbox{\tiny $\Omega$}}}

\begin{document}
\begin{center}{\bf  \large The Poisson problem for the fractional Hardy operator: Distributional identities and singular solutions}\medskip

\bigskip\medskip

 {\small  Huyuan Chen\footnote{chenhuyuan@yeah.net} \qquad Tobias Weth\footnote{weth@math.uni-frankurt.de}
}
 \bigskip\medskip

{\small  $^1$Department of Mathematics, Jiangxi Normal University,\\
Nanchang, Jiangxi 330022, PR China\\[3mm]

$^2$Goethe-Universit\"{a}t Frankfurt, Institut f\"{u}r Mathematik, Robert-Mayer-Str. 10,\\
D-60629 Frankfurt, Germany
 } \\[6mm]

\begin{abstract}
The purpose of this paper is to study and classify singular solutions of the Poisson problem 
$$
\left \{
\begin{aligned}
\cL^\s_\mu u  = f  \quad\  {\rm in}\ \, \Omega\setminus \{0\},\\
u  =0   \quad\  {\rm in}\ \,  \R^N \setminus \Omega\ 
\end{aligned}
\right.
$$
for the fractional Hardy operator $\cL_\mu^\s u= (-\Delta)^\s   u +\frac{\mu}{|x|^{2\s}}u$ in a bounded domain $\Omega \subset \R^N$ ($N \ge 2$) containing the origin. Here $(-\Delta)^\s$, $\s\in(0,1)$, is the fractional Laplacian of order $2\s$, and $\mu \ge \mu_0$, where $\mu_0 =      -2^{2\s}\frac{\Gamma^2(\frac{N+2\s}4)}{\Gamma^2(\frac{N-2\s}{4})}<0$ is the best constant in the fractional Hardy inequality. The analysis requires a thorough study of fundamental solutions and associated distributional identities. Special attention will be given to the critical case $\mu= \mu_0$ which requires more subtle estimates than the case $\mu>\mu_0$.   
\end{abstract}

\end{center}
   \smallskip

   \noindent {\small {\bf AMS Subject Classifications}:     35R11;  35J75; 35B40.
  \smallskip

\tableofcontents

\vspace{2mm}

\setcounter{equation}{0}
\section{Introduction}

Within recent years,  there has been an increasing interest in boundary value problems for nonlocal equations involving the fractional Laplacian and associated function spaces. This interest is motivated by various applications and relationships to the theory of partial differential equations, see
e.g. \cite{BV16,CS1,DPV,Ls07,V} and the references therein.  The fractional Laplacian is defined as
$$
(-\Delta)^\s  u(x)= C_{N,\s}\lim_{\epsilon\to0^+} \int_{\R^N\setminus B_\epsilon(0) }\frac{ u(x)-
u(z)}{|x-z|^{N+2\s}}  dz,
$$
where $\s\in(0,1)$, $B_\epsilon(0) \subset \R^N$ is the ball of radius $\eps>0$ centered at the origin and $C_{N,\s}=2^{2\s}\pi^{-\frac N2}\s\frac{\Gamma(\frac{N+2\s}2)}{\Gamma(1-\s)}$. Here $\Gamma$ denotes the Gamma function.
For basic properties of the fractional Laplacian, we refer e.g. to \cite{DPV}. In particular, it is known  that $(-\Delta)^\s  u(x)$ is well defined if $u$ is twice continuously differentiable in a neighborhood of $x$ and contained in the space
$L^1(\R^N,\frac{dx}{1+|x|^{N+2\s}})$. Here and in the following, if a (Lebesgue-)measurable subset $\Omega \subset \R^N$, a measurable nonnegative function $V$ on $\Omega$ and $1 \le p \le \infty$ is given, we let
$L^p(\Omega,V(x)dx)$ denote the space of all measurable functions $w: \Omega \to \R$ with $\int_{\Omega}  |w(x)|^p V(x)\,dx < \infty$. We also note that, for $u \in L^1(\R^N,\frac{dx}{1+|x|^{N+2\s}})$, the fractional Laplacian $(-\Delta)^\s  u$ can also be defined as a distribution:
$$
 [(-\Delta)^\s u]( \varphi )= \int_{\R^N} u (-\Delta)^\s \varphi \,dx \qquad \text{for $\varphi \in C^\infty_c(\Omega).$}
$$
We then have
$$
\cF((-\Delta)^\s u) = |\cdot|^{2\s} \hat u \qquad \text{on\ \ $\R^N \setminus \{0\}$}
$$
in the sense of distributions, where, here and in the following, both $\cF(w)$ and $\hat w$ denote the Fourier transform of a tempered distribution $w$. The main aim of this paper is to study singular solutions of linear inhomogeneous
equations involving the fractional Hardy operator
\begin{equation}\label{mu--00}
\cL_\mu^\s u  : = (-\Delta)^\s   u +\frac{\mu}{|x|^{2\s}}u,\qquad \quad \mu \ge \mu_0 :=
     -2^{2\s}\frac{\Gamma^2(\frac{N+2\s}4)}{\Gamma^2(\frac{N-2\s}{4})}
\end{equation}
in bounded domains $\Omega \subset \R^N$   containing the origin, where $N \ge 2$,
the (negative) constant $\mu_0$ is the optimal 
constant in the fractional Hardy inequality (see \cite{Ya}),  which we write in the form
\begin{equation}
  \label{eq:fractional-hardy-C-infty}
\int_{\R^N}\varphi  (-\Delta)^\s \varphi \,dx + \mu_0 \int_{\R^N} \frac{\varphi^2}{|x|^{2\s}}\,dx \ge 0 \qquad \text{for all $\varphi \in C^\infty_c(\R^N)$.}
\end{equation}
We note that, in the case $s=1$, the operator $\cL_\mu^\s$ is the classical Hardy operator which, due to the critical nature of the inverse-square potential, has been studied extensively in the last decades in the context of linear and nonlinear boundary value problems, see e.g. \cite{BDT,C,Du,FF,FM, CQZ, CZ,ChVe,V1}. More recently, equations involving $\cL_\mu^\s$ in the nonlocal case $0<s<1$ and related inequalities have been receiving growing attention and are addressed e.g. in \cite{ABP,BMP,Bhakta-et-al,Robert-Ghoussoub,FLS-2008,D,DMP,W,F,GV,V}. We will discuss some of these contributions in more detail below.

In the present paper, for an arbitrary bounded $C^2$-domain $\Omega \subset \R^N$ which contains the origin, we wish to classify solutions of the problem 
 \begin{equation}\label{eq 2.1fk}
\left \{
\begin{aligned}
\cL^\s_\mu u  = f  \quad  {\rm in}\ \  \Omega \setminus \{0\}, \\
 u  =0    \quad  {\rm in}\ \   \Omega^c\qquad\ 
\end{aligned}
\right.
\end{equation}
for a given function $f \in L^1_{loc}(\Omega \setminus \{0\})$.
Here and in the following, we set $\Omega^c=\R^N\setminus\Omega$. Moreover, 
by a solution we mean a function $u \in L^\infty_{loc}(\R^N \setminus \{0\}) \cap L^1(\R^N,\frac{dx}{1+|x|^{N+2\s}})$ which satisfies (\ref{eq 2.1fk}) in distributional sense, i.e., $u \equiv 0$ in $\Omega^c$ and 
$$
\int_{\R^N} u (-\Delta)^\s \varphi \,dx +\mu \int_{\R^N} \frac{u \varphi}{|x|^{2\s}}\,dx  = \int_{\Omega}f \phi\,dx   \qquad \text{for all $\varphi \in C^\infty_c(\Omega \setminus \{0\})$.}
$$
If $f \in C^\theta_{loc}(\Omega \setminus \{0\})$ for some $\theta>0$, then, by the regularity theory for the fractional Laplacian (see \cite{Ls07}),  any distributional solution $u$ of (\ref{eq 2.1fk}) is also a classical pointwise solution in $\Omega \setminus \{0\}$.

To motivate our study, let us first consider the case $\mu= 0$, in which $\cL^\s_0 = (-\Delta)^\s$. In this special case, our main result reads as follows.  
\begin{theorem}
  \label{sec:main-theorem-special-case}
Let $\Omega$ be a bounded $C^2$ domain containing the origin and  $f\in C^\theta_{loc}(\bar \Omega\setminus\{0\})$ for some $\theta \in (0,1)$.
 \begin{enumerate}
  \item[(i)] If $f\in L^1(\Omega)$, then for every $k \in \R$ there exists a unique solution 
$u_k \in L^\infty_{loc}(\R^N \setminus \{0\}) \cap L^1(\R^N,\frac{dx}{1+|x|^{N+2\s}})$ of the problem 
 \begin{equation}\label{eq 2.1fk-special-case}
\left \{
\begin{aligned}
(-\Delta)^\s u  = f  \quad  {\rm in}\ \  \Omega \setminus \{0\}, \\
 u  =0    \quad  {\rm in}\ \   \Omega^c,\qquad\,
\end{aligned}
\right.
\end{equation}
satisfying the distributional identity
\begin{equation}
  \label{eq:distributional-k-special-case}
\int_{\Omega}u_k   (-\Delta)^\s \xi \,dx = \int_{\Omega}f \xi\, dx+ k \xi(0)\qquad  \text{for all $\xi \in \cC^2_0(\Omega).$}
\end{equation}
If moreover $f\in L^\infty(\Omega, |x|^\rho dx)$ for some $\rho<2\s$, then $u_k$ has the asymptotics 
 \begin{equation}\label{beh-1-special-case}
 \lim_{x \to 0} u(x) |x|^{N-2\s} = \frac{k}{c_{\s,0}} \qquad \text{with}\quad c_{\s,0}:=C_{N,\s} \omega_{_{N-1}} \int^1_0 \int_{B_{t}(0)} \frac{|z|^{2\s-N}-1}{|e_1-z|^{N+2\s}}  dzdt.
 \end{equation}
\item[(ii)] Assume that $f$ is nonnegative and $\int_{\Omega} f\, dx  =+\infty$. Then problem (\ref{eq 2.1fk-special-case}) has no nonnegative distributional solution $u \in L^\infty_{loc}(\R^N \setminus \{0\}) \cap L^1(\R^N,\frac{dx}{1+|x|^{N+2\s}})$.
\end{enumerate}
\end{theorem}
Here and in the following, $\omega_{_{N-1}}$ is the measure of the unit sphere $\mathbb{S}^{N-1}$ in $\R^N$, and $e_1=(1,0,\cdots,0)\in\R^N$ is the first coordinate vector.
 A key role in our analysis will be played by the fundamental solution associated with the operator $\cL^\s_\mu$, which is a radially symmetric classical solution to the problem 
\begin{equation}\label{eq 1.1}
\cL^\s_\mu u=0\quad \  {\rm in}\ \, \mathbb{ R}^N\setminus \{0\}.
 \end{equation}
 In the case $\mu=0$, this fundamental solution is, up to a constant, given by $x \mapsto \Phi_{\s,0}(x):= |x|^{2\s-N}$, and it is uniquely determined by the distributional identity 
 \begin{equation}
   \label{eq:distributional-identity-mu-0}
\int_{\R^N}\Phi_{\s,0} (-\Delta)^\s \xi\,dx =  c_{\s,0} \xi(0) \qquad \text{for all $\xi \in C^2_c(\R^N)$.}   
 \end{equation}
A key step in analyzing the solution set of (\ref{eq 2.1fk}) for $\mu \not = 0$ is to extend the distributional identity~(\ref{eq:distributional-identity-mu-0}). For this we need to recall some properties of the set of radially symmetric classical solutions of (\ref{eq 1.1}). The following proposition summarizes and partly extends results contained in \cite[Section 3.1]{FLS-2008} and \cite[Section 2]{Robert-Ghoussoub}, see also \cite[Lemma 3.1]{F}.

\begin{proposition}
\label{proposition-A}
Assume that  $\s\in (0,1)$, $N\ge2$ and $\mu_0$ is given by (\ref{mu--00}).
Then for $\mu\ge \mu_0$, there exist a unique $\tau_-(\s,\mu)\in(-N,\, \frac{2\s-N}{2}]$ and a unique
$\tau_+(\s,\mu)\in[\frac{2\s-N}{2},\, 2\s)$ such that the functions
$\Phi_{\s,\mu},\  \Gamma_{\s,\mu} \in L^1(\R^N,\frac{dx}{1+|x|^{N+2\s}}) \cap C^\infty(\R^N \setminus \{0\})$ given by
\begin{equation}\label{fu}
  \Phi_{\s, \mu}(x)=\left\{\arraycolsep=1pt
\begin{array}{lll}
  |x|^{\tau_-(\s,\mu)}\quad
   &{\rm if}\ \, \mu>\mu_0\\[1mm]
 \phantom{   }
- |x|^{\tau_-(\s,\mu)}\ln|x| \quad &{\rm   if}\ \, \mu=\mu_0
 \end{array}
 \right. \qquad{\rm and}\quad \Gamma_{\s, \mu}(x)=|x|^{\tau_+(\s,\mu)}
\end{equation}
are classical solutions of (\ref{eq 1.1}). Furthermore, the mapping $\mu \mapsto \tau_-(\s,\mu)$ is continuous and strictly decreasing in $[\mu_0,+\infty)$,
and the mapping $\mu \mapsto \tau_+(\s,\mu)$ is continuous and strictly increasing in $[\mu_0,+\infty)$. In addition,
 \begin{align}
&\tau_-(\s,\mu)+\tau_+(\s,\mu)=2\s-N \qquad \text{for all $\mu \ge \mu_0$,}\nonumber\\
&\tau_-(\s,\mu_0)=\tau_+(\s,\mu_0)=\frac{2\s-N}2,\quad\ 
\tau_-(\s,0)=2\s-N, \quad\ \tau_+(\s,0)=0, \nonumber\\
&\lim_{\mu\to+\infty} \tau_-(\s,\mu)=-N\quad {\rm and}\quad \lim_{\mu\to+\infty} \tau_+(\s,\mu)=2\s.
 \nonumber
 \end{align}
\end{proposition}

For the range $\mu_0 < \mu \le 0$, these properties have been established in \cite[Lemma 3.1 and 3.2]{FLS-2008}, see also \cite[Proposition 2.1]{Robert-Ghoussoub}. The extension to the cases $\mu = \mu_0$ and $\mu>0$ is not difficult, and we shall give a self-contained proof in Section~\ref{sec:radi-symm-solut} below. 

At first glance it seems natural to guess that, for $\mu \not = 0$, the function $\Phi_{\s,\mu}$ also satisfies the distributional identity (\ref{eq:distributional-identity-mu-0}) in place of $\Phi_{\s,0}$ with $(-\Delta)^\s$ replaced by $\cL^\s_\mu$ and with $c_{\s,0}$ replaced by a suitable constant. However, this fails to be true already in the case of the classical Hardy operator $\cL^1_\mu:= -\Delta +\frac{\mu}{|x|^2}$ in dimensions $N \ge 3$. To see this, we note that, for $\s=1$, the radially symmetric solutions $\Phi_{1,\mu}$ and $\Gamma_{1,\mu}$  have the exponents
\begin{equation}\label{1.1}
 \tau_\pm(1,\mu)=- (N-2)/2\,\pm\sqrt{\mu+(N-2)^2/4},
\end{equation}
which are two roots of $\mu-\tau(\tau+N-2)=0$, see e.g. \cite{C, BDT}. 
Via an integration by parts argument, it is easy to observe that 
\begin{equation}
  \label{eq:new-LHS-int}
\int_{\R^N}|x|^{\tau_-(1,\mu)}\Big(-\Delta +\frac{\mu}{|x|^2}\Big) \xi\,dx = 0 \qquad \text{for all $\xi \in C^2_c(\R^N)\qquad$ \,if \,$\mu \in \Big(-\frac{(N-2)^2}{4},\, 0\Big)$,}
\end{equation}
since $\tau_-(1,\mu)> 2-N$ in this case. Moreover, in the case $\mu>0$ we have 
$\tau_-(1,\mu)< 2-N$, and therefore the integral in (\ref{eq:new-LHS-int}) is not even well-defined if $\xi(0)\not = 0$. In the recent paper \cite{CQZ}, this problem has been solved by establishing the new distributional identity
\begin{equation}\label{eq0030}
\int_{\R^N}\Phi_\mu \mathcal{L}_\mu^*\xi\, \Gamma_{1,\mu}dx =b_\mu\xi(0) \qquad \text{for all $\xi\in C_c^\infty(\R^N)$}
\end{equation}
with a formally adjoint operator $\mathcal{L}^*_\mu=-\Delta -2\frac{\tau_+(1,\mu) }{|x|^2}\,x\cdot\nabla$
and a suitable (explicit) constant $b_\mu$. 

In the present paper, we have to overcome a similar problem since $\tau_-(\s,\mu)<2\s-N$ for $\mu>0$ and $\tau_-(\s,\mu)>2\s-N$ for $\mu<0$. Therefore we need a new distributional identity related to the operators $\cL^\s_\mu$ which extends (\ref{eq:distributional-identity-mu-0}) to the case $\mu \not = 0$. {\it To simplify the notations, we write $B_r=B_r(0)$ for $r>0$, $\Phi_{\mu}$ in place of $\Phi_{\s,\mu}$ and $\Gamma_{\mu}$ in place of $\Gamma_{\s,\mu}$ in the following.}  We shall see that the dual of the operator $\cL^\s_\mu$ is the weighted fractional Laplacian $(-\Delta)^\s_{\Gamma_\mu}$ given by 
\begin{equation}\label{L}
(-\Delta)^\s_{\Gamma_\mu} v(x):=
C_{N,\s}\lim_{\epsilon\to0^+} \int_{\R^N\setminus B_\epsilon }\frac{v(x)-
v(z)}{|x-z|^{N+2\s}} \, \Gamma_\mu(z) dz.
\end{equation}
This expression is well defined for $x \in \R^N \setminus \{0\}$ if $v \in L^1(\R^N, \frac{\Gamma_\mu(x)}{1+|x|^{N+2\s}}dx)$ and if $v$ is twice continuously differentiable in a neighborhood of $x$. 
The weight $\Gamma_\mu$ blows up at origin for $\mu\in[\mu_0,0)$ and decays at the origin for $\mu>0$; therefore the operator $(-\Delta)^\s_{\Gamma_\mu}$ is in general not uniformly elliptic. The distributional identity for  the fundamental solution $\Phi_\mu$ of $\cL_\mu^\s$ at the origin now reads as follows.

 \begin{theorem}
\label{Theorem B}
 For any $\xi\in C^2_c(\R^N)$, we have
\begin{equation}\label{1.2}
 \int_{\R^N}\Phi_\mu   (-\Delta)^\s_{\Gamma_\mu}\xi \, dx  =c_{\s,\mu}\xi(0),
 \end{equation}
where the normalization constant $c_{\s,\mu}>0$ is given by
 \begin{equation}\label{cmu}
 c_{\s,\mu}:= \left\{ \displaystyle \arraycolsep=1pt
\begin{array}{lll}
  \displaystyle C_{N,\s} \omega_{_{N-1}} \int^1_0 \int_{B_{t}} \frac{|z|^{\tau_-(\s,\mu)}-|z|^{\tau_+(\s,\mu)}}{|e_1-z|^{N+2\s}}  dzdt\quad\ 
   &{\rm if}\ \, \mu>\mu_0,\\[4mm]
 \phantom{   }
 \displaystyle C_{N,\s} \omega_{_{N-1}} \int^1_0 \int_{B_{t}} \frac{|z|^{\frac{2\s-N}{2}}(-\ln|z|)}{|e_1-z|^{N+2\s}}  dzdt \quad\  &{\rm  if}\ \,  \mu=\mu_0.
 \end{array}
 \right.
 \end{equation}
\end{theorem}
The integral on the LHS of (\ref{1.2}) is indeed well defined for $\xi \in C^2_c(\R^N)$, since we have the estimate
\begin{equation}\label{1.3}
   | (-\Delta)^\s_{\Gamma_\mu}\xi(x)| \le  c_0 \min \{\Lambda _\mu(x),|x|^{-N-2\s}\}\qquad \text{for $\xi\in C^2_c(\R^N)$ and $x \in \R^N \setminus \{0\}$,}
\end{equation}
where $c_0=c_0(\s,\mu,\xi)>0$ is a constant and
\begin{equation}\label{def-Lambda}
 \Lambda _\mu(x)= \left\{\arraycolsep=1pt
\begin{array}{lll}
  \displaystyle 1\quad
   &{\rm if}\ \,  \tau_+(\s,\mu)>2\s-1,\\[1mm]
   |x|^{1-2\s+\tau_+(\s,\mu)} \quad
   &{\rm if}\ \,  \tau_+(\s,\mu)<2\s-1,\\[1mm]
 \phantom{   }
 \displaystyle  1+(-\ln|x|)_+ \quad &{\rm   if } \ \,  \tau_+(\s,\mu) =2\s-1.
 \end{array}
 \right.
\end{equation}
We shall prove this estimate in Proposition~\ref{lambda-estimate} below. Moreover, the integrals in the definition of the normalization constant $c_{\s,\mu}$ also exist in Lebesgue sense, see Lemma~\ref{h-estimate-lemma} below.

Although the operator $(-\Delta)^\s_{\Gamma_\mu}$ plays a similar role for $\cL^\s_\mu$ as the operator $\mathcal{L}^*_\mu$ plays for $\cL^1_\mu$, the proofs of (\ref{eq0030}) and (\ref{1.2}) are completely different. While (\ref{eq0030}) essentially follows by a suitable integration by parts, the proof of (\ref{1.2}) is based on a lengthy combination of integral transformations and estimates of remainder terms. With the distributional identity (\ref{1.2}) at hand, we may now study the singular problem (\ref{eq 2.1fk}) for all $\mu \ge \mu_0$. Our main result is the following.

\begin{theorem}
\label{theorem-C}
Let $\mu\ge\mu_0$ and $f\in C^\theta_{loc}(\bar\Omega\setminus \{0\})$ for some $\theta\in(0,1)$. 
\begin{enumerate}
\item[(i)] (Existence) If $f  \in L^1(\Omega,\Gamma_\mu(x) dx)$, then for every $k\in\R$ there exists a solution $u_k \in L^1(\Omega,\Lambda_\mu  dx)$ of   problem (\ref{eq 2.1fk})
satisfying the distributional identity
\begin{equation}
  \label{eq:distributional-k}
\int_{\Omega}u_k   (-\Delta)^\s_{\Gamma_\mu} \xi \,dx = \int_{\Omega}f \xi\, \Gamma_\mu dx +c_{\s,\mu} k\xi(0) \qquad \text{for all $\xi \in \cC^2_0(\Omega).$}
\end{equation}
\item[(ii)]  (Existence and Uniqueness) If $f \in L^\infty(\Omega, |x|^{\rho}dx)$ for some $\rho < 2\s - \tau_+(\s,\mu)$, then for every $k\in\R$ there exists a unique solution $u_k \in L^1(\Omega,\Lambda_\mu  dx)$ of  problem (\ref{eq 2.1fk}) with the asymptotics
 \begin{equation}\label{beh 1}
 \lim_{x \to 0}\:\frac{u_k(x)}{\Phi_\mu(x)} = k.
 \end{equation}
Moreover, $u_k$ satisfies the distributional identity (\ref{eq:distributional-k}). 
\item[(iii)] (Nonexistence) If $f$ is nonnegative and 
\begin{equation}\label{f2}
 \int_{\Omega} f\, \Gamma_\mu dx   =+\infty,
\end{equation}
then the problem
\begin{equation}\label{eq 1.1f}
 \arraycolsep=1pt\left\{
\begin{array}{lll}
 \displaystyle  \mathcal{L}_\mu^\s  u= f\quad
   {\rm in}\ \, {\Omega}\setminus \{0\},\\[1.5mm]
 \phantom{  L_\mu \, }
 \displaystyle  u\ge 0\quad  {\rm   in}\ \,   \Omega^c
 \end{array}\right.
\end{equation}
has no nonnegative distributional solution $u \in L^\infty_{loc}(\R^N \setminus \{0\}) \cap L^1(\R^N,\frac{dx}{1+|x|^{N+2\s}})$.
\end{enumerate}
\end{theorem}
We note here that the assumption $f \in L^\infty(\Omega, |x|^{\rho}dx)$ for some $\rho < 2\s - \tau_+(\s,\mu)$ also implies that $f  \in L^1(\Omega,\Gamma_\mu(x) dx)$ since $2\s<2\le N$ by assumption.
 
\begin{remark}{\rm 
  \begin{itemize}
\item[(i)] In both Theorem~\ref{Theorem B} and Theorem~\ref{theorem-C}, the case $\mu= \mu_0$ is more difficult to treat than the case $\mu> \mu_0$ and requires separate estimates. 
  \item[(ii)] We shall see that the solution $u_k$ of problem (\ref{eq 2.1fk}) writes in the form $u_k= u_{min} + k\Phi_\Omega$, where $\Phi_\Omega \in L^\infty(\R^N \setminus \{0\})\cap L^1(\Omega)$ solves 
    \begin{equation}
      \label{eq:problem-fundamental-solution}
\mathcal{L}_\mu^\s \Phi_\Omega = 0 \quad \text{in $\,\Omega \setminus \{0\}$},\qquad \Phi_\Omega \equiv 0 \quad \text{in $\,\Omega^c$},\qquad  \lim_{x \to 0}\:\frac{\Phi_\Omega(x)}{\Phi_\mu(x)} = 1.
    \end{equation}
In the case $f \ge 0$, $u_{\min}$ is approached by a sequence $(u_n)_n$ of solutions of (\ref{eq 2.1fk}) corresponding to bounded and monotone approximations $f_n$ of 
$f$. When $\mu>\mu_0$, $f$ is bounded and $k=0$, an application of the Hardy inequality and Riesz representation theorem 
gives rise to a unique weak solution of (\ref{eq 2.1fk}) in the standard Sobolev space $\cH^\s_0(\Omega)$, see Section~\ref{sec:preliminary-notation} below. In the case $\mu=\mu_0$ the situation is more delicate, and we cannot expect to have solutions in $\cH^\s_0(\Omega)$. For this reason, we develop an approach for 
singular weak solutions contained in $\cH^\s_{loc}(\Omega \setminus \{0\})$.
The key step in this approach is Theorem~\ref{sec:dirichl-probl-bound-comp-corol} below.
\item[(iii)] Theorem~\ref{sec:main-theorem-special-case} is essentially a special case of Theorem~\ref{theorem-C} for $\mu=0$. The only additional information we have in Theorem~\ref{sec:main-theorem-special-case} is a stronger uniqueness statement. Here, solutions are already uniquely determined by the distributional identity~(\ref{eq:distributional-k-special-case}). This is merely a consequence of the well-known fact that every $u \in L^\infty_{loc}(\R^N \setminus \{0\}) \cap L^1(\R^N,\frac{dx}{1+|x|^{N+2\s}})$ satisfying 
\begin{equation}
  \label{eq:distributional-special-case-0}
\int_{\Omega}u   (-\Delta)^\s \xi \,dx = 0 \qquad \text{for all $\xi \in \cC^2_0(\Omega)$}
\end{equation}
and $u  =0$ in $\Omega^c$ must vanish identically. The corresponding statement for $\mu \not = 0$ remains an open problem. In this case, we have to replace (\ref{eq:distributional-special-case-0}) by the condition 
\begin{equation}
  \label{introduction-remark}
\int_{\Omega}u   (-\Delta)^\s_{\Gamma_\mu} \xi \,dx = 0 \qquad \text{for all $\xi \in \cC^2_0(\Omega).$}
\end{equation}
  \end{itemize}
}
\end{remark}

In our final main result, we note that the nonexistence result given in Theorem~\ref{theorem-C}(iii) extends, even without condition~(\ref{f2}), to the case $\mu<\mu_0$. 

\begin{theorem}
\label{theorem-D}
Let $\mu <\mu_0$ and $f \in L^\infty_{loc}(\overline \Omega \setminus \{0\})$ be a nonnegative function. 
Then the problem \eqref{eq 1.1f} has no nonnegative distributional solution $u \in L^\infty_{loc}(\R^N \setminus \{0\}) \cap L^1(\R^N,\frac{dx}{1+|x|^{N+2\s}})$.
\end{theorem}

To put our results into perspective, we wish to briefly discuss recent related work. In fact, our study complements recent seminal contributions in the literature on weak solutions of the problem
\begin{equation}
  \label{eq:weak-nonsing-problem}
  \left\{
    \begin{aligned}
      \cL^\s_\mu u  &= f  &&\qquad  \text{in $\Omega$},\\
      u  &=0    &&\qquad \text{in $\Omega^c$}
    \end{aligned}
  \right.
  \end{equation}
  and its nonlinear generalizations, see e.g. \cite{ABP,BMP,Bhakta-et-al,Robert-Ghoussoub,FLS-2008} and the references therein. Different definitions of weak solutions of (\ref{eq:weak-nonsing-problem}) are used in the literature, and they usually depend on $\mu$ and the integrability properties of $f$, see e.g. \cite[Definitions 2.5. and 2.7]{ABP}. Nevertheless, these definitions always involve test function spaces larger than $C^\infty_c(\Omega \setminus \{0\})$ and are therefore less general than the definition of a distributional solution of (\ref{eq 2.1fk}). In particular, no distributional identities related to point measures supported at the origin have been considered in previous papers. Related to this aspect, we point out that, in the case $\mu \in (\mu_0,0)$, source terms with measures supported {\em away from the origin} on the RHS of (\ref{eq:weak-nonsing-problem}) can be treated with the help of the Green function constructed in \cite{Bhakta-et-al}, but the construction in \cite{Bhakta-et-al} does not yield a solution of (\ref{eq:problem-fundamental-solution}). Moreover, as we have shown in Theorems~\ref{Theorem B} and \ref{theorem-C}, the associated distributional identities are different in the more subtle case of a point measure supported at the origin.
  
In the local case $\s=1$,  semilinear Dirichlet problems with singular Hardy terms inside the domain or on the boundary and measures have
been  studied in \cite{CQZ,CZ,ChVe} by considering related distributional identities and solutions with respect to a suitably chosen dual test functions space. On the contrary, in the fractional case $\s\in(0,1)$, due to the nonlocality and singular nature of the operator $(-\Delta)^\s_{\Gamma_\mu}$, it remains a challenging open problem how to enlarge the test functions space $\cC^2_0(\Omega)$ into a suitable one that gives rise to a version of Kato's inequality. In the particular case $\mu=0$, this has been done in \cite{CZ}.

The paper is organized as follows. In Section \ref{sec:radi-symm-solut}, we include, for the readers convenience, a self-contained proof of Proposition~\ref{proposition-A} which follows arguments in \cite{FLS-2008} and extends the range of parameters $\mu$ considered there. Section~\ref{sec:fund-solut-distr} is devoted to the proof of the distributional identity given in Theorem~\ref{Theorem B} and related estimates. In Section~\ref{sec:dirichl-probl-bound} we develop the functional analytic framework to study singular solutions of (\ref{eq 2.1fk}), and we prove the existence and uniqueness parts of Theorem~\ref{theorem-C}. In particular, this requires to address weak maximum principles and to study variational weak solutions in Section~\ref{sec:preliminary-notation}. 
In Section~\ref{sec:nonexistence} we prove our nonexistence results given in Theorem~\ref{theorem-C}(iii) and Theorem~\ref{theorem-D}. Finally, in the appendix of this paper, we annex proofs of auxiliary results which are adaptations of known facts to the singular context of the Hardy operator. In particular, in Section~\ref{sec:appendix-a.-an} we prove local $L^\infty$-bounds for solutions of a linear inhomogeneous fractional problem, while in Section~\ref{sec:appendix-b.-density} we prove a density property in $\cH^\s_0(\Omega)$.\\

{\bf Notation.} Throughout the paper, we write $B_r:= B_r(0)$ for the open ball with radius $r>0$ centered at the origin. As noted above, since $s \in (0,1)$ is fixed, we write $\Phi_{\mu}$ in place of $\Phi_{\s,\mu}$ and $\Gamma_{\mu}$ in place of $\Gamma_{\s,\mu}$ for the functions introduced in Proposition~\ref{proposition-A}. Moreover, to further abbreviate the notation, we set $d \gamma_\mu = \Gamma_\mu dx$. 

\setcounter{equation}{0}

\section{Radially symmetric solutions of the fractional Hardy equation}
\label{sec:radi-symm-solut}
This section is devoted to the proof of Proposition~\ref{proposition-A}. We start with some preliminary lemmas. In the following, for $\tau>-N-2$, we regard the function $|\cdot|^\tau$ as a tempered distribution in $\mathcal{S}'(\R^N)$. It is a regular distribution if $\tau>-N$, and it is understood as a principal value integral if $-N-2<\tau\le N$, i.e.
$$
|\cdot|^{\tau}(\psi) = \lim_{\eps \to 0} \int_{\R^N \setminus B_\eps}[\psi(x)-\psi(0)]|x|^\tau dx \qquad \text{for a Schwartz function $\psi \in \mathcal{S}(\R^N)$.}
$$

\begin{lemma}
\label{fractional-power}
Let $\tau \in (-N,2\s)$, so that $|\cdot |^\tau \in L^1(\R^N,\frac{dx}{1+|x|^{N+2\s}})$. For $\tau \not \in \{0, 2 \s -N\}$, we then have
\begin{equation}
  \label{eq:fractional-power}
 (-\Delta)^\s |\cdot|^\tau = c_\s(\tau)|\cdot|^{\tau-2\s} \qquad \text{in $\mathcal{S}'(\R^N)$ with $c_\s(\tau) = 2^{2\s} \frac{\Gamma(\frac{N+\tau}{2})\Gamma(\frac{2\s-\tau}{2})}{\Gamma(-\frac{\tau}{2})\Gamma(\frac{N-2\s+\tau}{2})}.$}
\end{equation}
Moreover, $(-\Delta)^\s 1 = 0$ and
$(-\Delta)^\s |\cdot|^{2\s-N} = 2^{2 \s}\pi^{N/2}\frac{\Gamma(\s)}{\Gamma(\frac{N-2\s}{2})} \delta_0$ in $\mathcal{S}'(\R^N)$.
\end{lemma}
\begin{proof} The claim for $\tau = 0$ is clear, so we may assume $\tau \not = 0$. We recall (see e.g. \cite[Chapter II]{GS}) that
$$
\mathcal{F}((-\Delta)^\s |\cdot|^\tau)(\zeta)=|\zeta|^{2\s}\mathcal{F}(|\cdot|\tau)(\zeta)=
\sigma(\tau)|\cdot|^{2\s-N-\tau}(\zeta) \quad {\rm in}\ \ \mathcal{S}'(\R^N)
$$
with
$\sigma(\tau):=2^{\tau+N}\pi^{N/2}\frac{\Gamma(\frac{\tau+N}2)}{\Gamma(-\frac\tau2)}.$ Consequently, if $\tau \not = 2\s-N$, we have that
$$
 (-\Delta)^\s |\cdot|^\tau  =  \sigma(\tau)  \mathcal{F}^{-1}\bigl(|\cdot|^{2\s-N-\tau}\bigr)
=  \frac{\sigma(\tau)}{\sigma(\tau-2\s)} |\cdot|^{\tau-2\s} = c_\s(\tau)|\cdot|^{\tau-2\s} \quad {\rm in}\ \ \mathcal{S}'(\R^N)
$$
and if $\tau  = 2\s-N$,
$$
(-\Delta)^\s |\cdot|^{2\s-N} = \sigma(2\s-N) \mathcal{F}^{-1}(1) = \sigma(2\s-N)\delta_0= 2^{2 \s}\pi^{N/2}\frac{\Gamma(\s)}{\Gamma(\frac{N-2\s}{2})}\delta_0\quad {\rm in}\ \ \mathcal{S}'(\R^N).
$$
We complete the proof.\end{proof}

\begin{remark}
\label{fractional-power-remark}
By the regularity theory for the fractional Laplacian (see \cite{Ls07}), the identities (\ref{eq:fractional-power}) hold in classical sense in $\R^N \setminus \{0\}$. 

Moreover, for $\tau \in \{0, 2\s- N\}$, we have $(-\Delta)^\s |\cdot|^{\tau}= 0$ in $\R^N \setminus \{0\}$ in classical sense.  
\end{remark}

\begin{lemma}
\label{c-function}
The function
$$
c_\s: (-N, 2 \s) \to \R, \qquad c_\s(\tau) =2^{2\s} \frac{\Gamma(\frac{N+\tau}{2})\Gamma(\frac{2\s-\tau}{2})}{\Gamma(-\frac{\tau}{2})\Gamma(\frac{N-2\s+\tau}{2})}
$$
is strictly concave and uniquely maximized at the point $\frac{2\s-N}{2}$ with the maximal value $2^{2\s}  \frac{\Gamma^2(\frac{N+2\s}4)}{\Gamma^2(\frac{N-2\s}{4})}.$

Moreover,
\begin{equation}
  \label{eq:c-symmetry}
c_\s(\tau)= c_\s(2\s-N-\tau) \qquad \text{for $\tau \in (-N,2\s)$}
\end{equation}
and
\begin{equation}
  \label{eq:c-asymptotics}
\lim_{\tau \to -N}c_\s(\tau) = \lim_{\tau \to 2 \s}c_\s(\tau)=-\infty.
\end{equation}
\end{lemma}

\begin{proof}
By definition, we have $c_\s(2\s-N-\tau)=2^{2\s} \frac{\Gamma(\frac{2\s-\tau}{2})\Gamma(\frac{N+\tau}{2})}{\Gamma(\frac{N-2\s+\tau}{2})\Gamma(\frac{-\tau}{2})}=c_\s(\tau)$,
hence (\ref{eq:c-symmetry}) holds. Moreover, $\lim \limits_{\tau\to -N}\Gamma(\frac{N+\tau}{2})= +\infty$, whereas the other terms in the definition of $c_\s$ remain bounded. Hence $c_\s(\tau) \to-\infty$ as $\tau\to -N$, and by (\ref{eq:c-symmetry}) we get (\ref{eq:c-asymptotics}).

To prove the concavity of $c_\s$, we use derive a different representation. By Lemma~\ref{fractional-power} and Remark~\ref{fractional-power-remark}, we have
\begin{align*}
c_\s(\tau) |x|^{\tau-2\s} &=  [(-\Delta)^\s |\cdot|^\tau](x) = -\frac{C_{N,\s}}2 \int_{\R^N}\frac{|x+y|^{\tau}+|x-y|^\tau-2|x|^\tau}{|y|^{N+2\s}}\,dy\\
   &= -\frac{C_{N,\s}}2|x|^{\tau-2\s}\int_{\R^N}\frac{|e_1+z|^{\tau}+|e_1-z|^\tau-2}{|z|^{N+2\s}}\,dz \quad \text{for $x \in \R^N \setminus \{0\}$,}
\end{align*}
where $e_1=(1,0,\cdots,0)\in\R^N$, and thus $c_\s (\tau)=-\frac{C_{N,\s}}2\int_{\R^N}\frac{|x-e_1|^\tau+|x+e_1|^\tau-2}{|x|^{N+2\s}}\,dx.$

Consequently, for $\tau \in (-N,2\s)$, we have 
$$
c_\s'(\tau)=-\frac{C_{N,\s}}{2}\int_{\R^N}\frac{|e_1-x|^{\tau} \log|e_1-x|+|e_1+x|^{\tau}\log|e_1+x|}{|x|^{N+2\s}}dx
$$
and
$$
c_\s''(\tau)=-\frac{C_{N,\s}}{2}\int_{\R^N}\frac{|e_1-x|^{\tau} (\log|e_1-x|)^2+|e_1+x|^{\tau}(\log|e_1+x|)^2}{|x|^{N+2\s}}dx<0,
$$
which yields the strict concavity of $c_\s$. Combining this property with (\ref{eq:c-symmetry}) and (\ref{eq:c-asymptotics}), 
we obtain that $c_\s(\cdot)$ is uniquely maximized at $\tau= \frac{2\s-N}{2}$.
\end{proof}

\begin{proof}[Proof of Proposition~\ref{proposition-A}.]
 It follows from Lemma~\ref{c-function} that, for $\mu>\mu_0 = -c_\s(\frac{2\s-N}{2})$, the equation
$$
c_\s(\tau )=-\mu
$$
has a unique solution $\tau_-(\s,\mu) \in (-N,\frac{2\s-N}{2})$ and a unique solution $\tau_+(\s,\mu) \in (\frac{2\s-N}{2},2\s)$. Moreover, $\tau_-(\s,\mu)+\tau_+(\s,\mu)=2\s-N$ by (\ref{eq:c-symmetry}), and
$$
\lim_{\mu\to+\infty} \tau_-(\s,\mu)=-N, \qquad \lim_{\mu\to+\infty} \tau_+(\s,\mu)=2\s
$$
by (\ref{eq:c-asymptotics}). Defining $\Phi_\mu$ and $\Gamma_\mu$ by~(\ref{fu}), we thus deduce from Lemma~\ref{fractional-power} and Remark~\ref{fractional-power-remark} that the claim of Proposition~\ref{proposition-A} holds for $\mu > \mu_0$. For $\mu= \mu_0$ we now define $\tau_+(\s,\mu_0)= \tau_-(\s,\mu_0)= \frac{2\s-N}{2}$. Then $\Gamma_{\mu_0}$ -- as defined in ~(\ref{fu}) -- is a solution of (\ref{eq 1.1}) by Lemma~\ref{fractional-power}. Moreover, differentiating the identity
$$
c_\s(\tau) |\cdot|^{\tau-2\s} =   (-\Delta)^\s |\cdot|^\tau \quad\ \text{in $\mathcal{S}'(\R^N)$}
$$
from Lemma~\ref{fractional-power} with respect to $\tau$ at $\tau = \frac{2\s-N}{2}$, we obtain that
$$
c_\s'(\frac{2\s-N}{2})|\cdot|^{-\frac{2\s+N}{2}}+ c_\s(\frac{2\s-N}{2})|x|^{-\frac{2\s+N}{2}}\ln |\cdot|
= \bigl[(-\Delta)^\s (|\cdot|^{\frac{2\s-N}{2}} \ln |\cdot|)\bigr] \quad\ \text{in $\mathcal{S}'(\R^N)$.}
$$
Since $c_\s'(\frac{2\s-N}{2})= 0$ by Lemma~\ref{c-function}, we conclude that
$$
  (-\Delta)^\s \Phi_{\mu_0}= - [(-\Delta)^\s (|\cdot|^{\frac{2\s-N}{2}} \ln |\cdot|)]
= - c_\s(\frac{2\s-N}{2})|\cdot|^{-\frac{2\s+N}{2}}\ln |\cdot| =-\frac{\mu_0}{|x|^{2\s}}\Phi_{\mu_0} \quad\ \text{in $\mathcal{S}'(\R^N)$.}
$$
Moreover, the equation holds in classical sense in $\R^N \setminus \{0\}$. Hence $\Gamma_{\mu_0}$ solves (\ref{eq 1.1}) as well.
The proof is complete.
\end{proof}

\section{Fundamental solution and distributional identity}
\label{sec:fund-solut-distr}

In this section we give the proof of Theorem~\ref{Theorem B}. Recall that the fractional Laplacian with weight $\Gamma_\mu$ given by (\ref{L}) is the dual operator of $\cL^\s_\mu$ and its properties are
 important in our analysis of the fundamental solution associated to $\cL^\s_\mu$. Related to this, we first provide estimate (\ref{1.3}).

 \begin{proposition}
\label{lambda-estimate}
Let $\s \in (0,1)$ and $\mu \ge \mu_0$. Then we have    
\begin{equation}\label{1.3-proof}
   | (-\Delta)^\s_{\Gamma_\mu}\xi(x)| \le  c_0 \min \{\Lambda _\mu(x),|x|^{-N-2\s}\}\qquad \text{for $\xi\in C^2_c(\R^N)$ and $x \in \R^N \setminus \{0\}$,}
\end{equation}
where $c_0=c_0(\s,\mu,\xi)>0$ is a constant and $\Lambda_\mu$ is given by \eqref{def-Lambda}.
 \end{proposition}

\begin{proof}
In the following, we put  $S_r:= \partial B_r$ and $B_r^c:= \R^N \setminus B_r$ for $r>0$.
We first note the following facts. If $\eta, \tau>-N$, then we have
\begin{equation}
  \label{eq:eta-gamma-est-1}
A_{\eta,\tau}(r):= \int_{B_r \setminus B_2} |z|^{\eta} |z+e_1|^{\tau}\,dz \le \kappa_1(\eta,\tau)
\left \{
\begin{aligned}
&(1+r^{\eta+\tau+N})&&\ \text{if \,$\eta+\tau \not = -N$},\\
&(1+\log r) &&\ \text{if \,$\eta+\tau = -N$}\\
\end{aligned}
\right.
\end{equation}
for $r \ge 2$ with a constant $\kappa_1(\eta,\tau)$, since $\frac12|z|\leq |z+e_1|\leq 2|z|$ for $|z|\geq 2$. Moreover, if $\tau, \eta \in \R$ satisfy $\tau + \eta <-N$, we have that
\begin{equation}
  \label{eq:eta-gamma-est-2}
B_{\eta,\tau}(r):= \int_{B_r^c} |z|^{\eta} |z+e_1|^{\tau}\,dz \le \kappa_2(\eta,\tau)r^{\eta +\tau+N}<\infty    \quad \text{for \,$r \ge 2$}
\end{equation}
with a constant $\kappa_2(\eta,\tau)$.

We now turn to the claim of the proposition. In the following, the constant $c>0$ depends only on $\xi$ and may change from line to line. We first note the following scaling property of the operator $(-\Delta)^\s_{\Gamma_\mu}$. If $\xi \in C^2_c(\R^N)$, $R>0$ and $\xi_R \in C^2_c(\R^N)$ is given by
$\xi(x)= \xi(Rx)$, then $[(-\Delta)^\s_{\Gamma_\mu} \xi_R](x)= R^{2\s-\tau_+(\s,\mu)} = [(-\Delta)^\s_{\Gamma_\mu} \xi](Rx)$ for $x \in \R^N \setminus \{0\}$.
As a consequence, we may assume in the following that $\xi \in C^2_c(\R^N)$ is supported in $B_{\frac{1}{4}}(0)$.

For $x \in \R^N$ with $|x| \ge \frac{1}{2}$, then we have
\begin{equation}
  \label{eq:outside-estimate}
|(-\Delta)^\s_{\Gamma_\mu} \xi(x)| \le C_{N,\s} \int_{B_\frac{1}{4}} \frac{|\xi(y)|\Gamma_{\Gamma_\mu}(y) }{|x-y|^{N+2\s}}  dy
\le  \frac{c}{|x|^{N+2\s}} \int_{B_\frac{1}{4}} |\xi(y)| \Gamma_\mu(y) dy \le  \frac{c}{|x|^{N+2\s}}.
\end{equation}
We now consider $x \in \R^N$ with $|x| < \frac{1}{2}$.  By translation and suitable addition and subtraction of terms, we see that
$$
2 \frac{(-\Delta)^\s_{\Gamma_\mu} \xi(x)}{C_{N,\s}} = \Delta_1(x)+\Delta_2(x)
$$
with
$$
\Delta_1(x):= \int_{\R^N} \frac{2\xi(x)-\xi(x+z)-\xi(x-z)}{|z|^{N+2\s}}\Gamma_\mu(x+z)\,dz
$$
and
$$
\Delta_2(x):= \int_{\R^N} \frac{(\xi(x)-\xi(x+z))(\Gamma(x+z)-\Gamma(x-z))}{|z|^{N+2\s}}\,dz.
$$
Since $|2\xi(x)-\xi(x+z)-\xi(x-z)| \le c \min \{1, |z|^2\}$ for $x,z \in \R^N$, we have that, using (\ref{eq:eta-gamma-est-1}) and (\ref{eq:eta-gamma-est-2}),
\begin{align}
|\Delta_1(x)|&\le c \int_{|z|<1} |z|^{-N-2\s+2}|x+z|^{\tau_+(\s,\mu)} dz +  c \int_{|z| \ge 1} |z|^{-N-2\s}|x+z|^{\tau_+(\s,\mu)} dz \nonumber\\
&= c |x|^{2-2\s+\tau_+(\s,\mu)}\int_{|z|<\frac{1}{|x|}} |z|^{-N-2\s+2}|z + e_1|^{\tau_+(\s,\mu)} dz\nonumber \\
&\quad+ c |x|^{\tau_+(\s,\mu)-2\s}\int_{|z|>\frac{1}{|x|}} |z|^{-N-2\s}|z + e_1|^{\tau_+(\s,\mu)} dz \nonumber\\
&= c\Bigl[c_1(\s,\mu) +  |x|^{2-2\s+\tau_+(\s,\mu)} A_{\eta_1,\tau}(\frac{1}{|x|}) +  |x|^{\tau_+(\s,\mu)-2\s}
B_{\eta_2,\tau}(\frac{1}{|x|})\Bigr] \label{inequality-Delta-1}
\end{align}
with $\eta_1 = 2-N-2\s$, $\eta_2 = -N -2\s$, $\tau= \tau_+(\s,\mu)$ and
$$
c_1(\s,\mu)= \int_{|z|< 2}|z|^{-N-2\s}|z + e_1|^{\tau_+(\s,\mu)} dz.
$$
By (\ref{eq:eta-gamma-est-1}) and (\ref{eq:eta-gamma-est-2}), we thus conclude that
\begin{equation}
  \label{eq:Delta-1-est}
|\Delta_1(x)| \le  c_2(\s,\mu)\left\{\arraycolsep=1pt
\begin{array}{lll}
1\quad
   &{\rm if}\ \, \tau_+(\s,\mu)>2\s-2,\\[1mm]
    |x|^{2-2\s+\tau_+(\s,\mu)}  \quad
   &{\rm if}\ \, \tau_+(\s,\mu)<2\s-2,\\[1mm]
 \phantom{   }
 (-\ln|x|) \quad &{\rm   if } \ \,  \tau_+(\s,\mu) =2\s-2.
 \end{array}
 \right.
\end{equation}
To estimate $|\Delta_2(x)|$, we note that
\begin{align*}
|\Gamma(x+z)&-\Gamma(x-z))|=\bigl| |x+z|^{\tau_+(\s,\mu)} - |x-z|^{\tau_+(\s,\mu)}\bigr|\\
&\le c \min \Bigl \{|x+z|^{\tau_+(\s,\mu)} + |x-z|^{\tau_+(\s,\mu)}, |z| \Bigl(|x+z|^{\tau_+(\s,\mu)-1} + |x-z|^{\tau_+(\s,\mu)-1}\Bigr)\Bigr\}
\end{align*}
and $|\xi(x)-\xi(x+z)| \le c \min \{1,|z|\}$, implying that
\begin{align*}
|\Delta_2(x)| &\le c \int_{|z|\le 1} |z|^{-N-2\s+2} \Bigl(|x+z|^{\tau_+(\s,\mu)-1} + |x-z|^{\tau_+(\s,\mu)-1}\Bigr)dz\\
&\quad+ c\int_{|z| > 1}|z|^{-N-2\s} \Bigl(|x+z|^{\tau_+(\s,\mu)} + |x-z|^{\tau_+(\s,\mu)}\Bigr)dz.
\end{align*}
Following the inequalities in \eqref{inequality-Delta-1} with $\tau_+(\s,\mu)-1$ in place of $\tau_+(\s,\mu)$ and using
(\ref{eq:eta-gamma-est-1}) and (\ref{eq:eta-gamma-est-2}) again, we find that
\begin{equation}
  \label{eq:Delta-2-est}
|\Delta_2(x)| \le  c(\s,\mu)\left\{\arraycolsep=1pt
\begin{array}{lll}
1\quad
   &{\rm if}\ \, \tau_+(\s,\mu)>2\s-1,\\[1mm]
    |x|^{1-2\s+\tau_+(\s,\mu)}  \quad
   &{\rm if}\ \, \tau_+(\s,\mu)<2\s-1,\\[1mm]
 \phantom{   }
 (1-\ln|x|) \quad &{\rm   if } \ \, \tau_+(\s,\mu) =2\s-1
 \end{array}
 \right.
\end{equation}
with a constant $c(\s,\mu)>0$. Combining (\ref{eq:outside-estimate}),~(\ref{eq:Delta-1-est}) and (\ref{eq:Delta-2-est}), we thus obtain (\ref{1.3-proof}).
\end{proof}

We remark that  an inspection of the estimates in above proof shows that for fixed $\s \in (0,1)$, $\xi \in C^2_c(\R^N)$ and $\tau < 2 \s-1$, the constant $c(\s,\mu,\xi)$ can by chosen uniformly for any $\mu \ge \mu_0$ with $\tau_+(\s,\mu) \le \tau$.

\begin{lemma}
\label{relationship-operators-0}
We have 
\begin{equation}
  \label{eq:adjoint-op-pointwise}
[\cL^\s_\mu (\Gamma_\mu v)](x) = (-\Delta)^\s_{\Gamma_\mu} v(x) \qquad \text{for $v \in C^2_c(\R^N)$, $x \in \R^N \setminus \{0\}$}
\end{equation}
and
\begin{equation}
  \label{eq:adjoint-op}
\int_{\R^N} v \cL^\s_\mu u\,d\gamma_\mu = \int_{\R^N} u (-\Delta)^\s_{\Gamma_\mu} v\,dx \qquad \text{for $v \in C^2_c(\R^N)$, $u \in C^2(\R^N) \cap L^1(\R^N, \frac{dx}{1+|x|^{N+2\s}})$.}
\end{equation}
Here we recall that we have set $d\gamma_\mu = \Gamma_\mu dx$. 
\end{lemma}

\begin{proof}
Let $\psi= \Gamma_\mu v$. For $x \in \R^N \setminus \{0\}$, we then find, recalling  the fact that $\cL^\s_\mu \Gamma_\mu = 0$ on $\R^N \setminus \{0\}$, that
\begin{equation}
[\cL^\s_\mu \psi](x)  = v(x)[\cL^\s_\mu \Gamma_\mu](x)+[(-\Delta)^\s (v \Gamma_\mu) - v(-\Delta)^\s \Gamma_\mu](x)=  (-\Delta)^\s_{\Gamma_\mu} v(x). \label{product-prelim}
\end{equation}
Consequently,
\begin{align*}
 \int_{\R^N} v \cL^\s_\mu u\,d\gamma_\mu = \int_{\R^N} \psi \cL^\s_\mu u \, dx&=\int_{\R^N}  \frac{\mu}{|x|^{2\s}} u \psi\,dx + \frac{C_{N,\s}}2 \int_{\R^N} \int_{\R^N}\frac{[u(x)-u(y)][\psi(x)-\psi(y)]}{|x-y|^{N+2\s}}dy dx\\
&=  \int_{\R^N} u \cL^\s_\mu \psi\,dx = \int_{\R^N} u [(-\Delta)^\s_{\Gamma_\mu} v]\,dx,
\end{align*}
where all integrals are well defined in Lebesgue sense as a consequence of Lemma~\ref{lambda-estimate}. The claim follows.
\end{proof}

Before we may complete the proof of Theorem~\ref{Theorem B}, we need another integral estimate. 

\begin{lemma}
\label{h-estimate-lemma}
Let $\tau > -N$, and consider the function 
$$
h: (0,1) \to \R, \qquad h(t) = \int_{B_{t} } \frac{1-|z|^{\tau}}{|e_1-z|^{N+2\s}}  dz 
$$
Then there exists a constant $c=c(\tau,\s,N)>0$ with 
\begin{equation}
  \label{eq:h-general-estimate-1}
|h(t)| \le c(t^N+ t^{N+\tau}) \qquad \text{for $t \in (0,\frac{1}{2}]$}
\end{equation}
and, 
\begin{equation}
  \label{eq:h-general-estimate-2}
|h(t)| \le \left\{
  \begin{aligned}
&c &&\quad\ \text{if $\s \in (0,\frac{1}{2})$,}\\ 
&|\ln (1-t)| &&\quad\ \text{if $\s = \frac{1}{2}$,}\\     
&(1-t)^{1-2\s} &&\quad\ \text{if $\s \not = \frac{1}{2}$}\\ 
  \end{aligned}
\right.
\qquad \text{for $t \in [\frac{1}{2},1)$.}
 \end{equation}
\end{lemma}

\begin{proof}
Estimate (\ref{eq:h-general-estimate-1}) follows easily from the definition of $h$. To see (\ref{eq:h-general-estimate-2}), we let $c>0$ denote constants which only depend on $N$, $\s$ and $\tau$. For $z \in B_1 \setminus B_{\frac{1}{2}}$, we then have 
$$
\bigl|1-|z|^\tau \bigr|\le c (1-|z|) \le c |e_1-z|,
$$
which implies that, for $t \in (\frac{1}{2},1)$, 
$$
|h(t)|\le \Bigl| \int_{B_{\frac{1}{2}}} \frac{1-|z|^{\tau}}{|e_1-z|^{N+2\s}}  dz 
\Bigr|+
\Bigl|\int_{B_{t}\setminus B_{\frac{1}{2}}} \frac{1-|z|^{\tau}}{|e_1-z|^{N+2\s}}  dz\Bigr|\\ 
\le |h(\frac{1}{2})| +  c \int_{B_{t}}|e_1-z|^{1-N-2\s}
dz,
$$
where 
$$
\int_{B_{t}}|e_1-z|^{1-N-2\s} \le 
\int_{B_2 \setminus B_{1-t}}|y|^{1-N-2\s}\,dy \le 
\left\{
  \begin{aligned}
&c &&\quad\ \text{if $\s \in (0,\frac{1}{2})$,}\\ 
&|\ln (1-t)| &&\quad\ \text{if $\s = \frac{1}{2}$,}\\     
&(1-t)^{1-2\s} &&\quad\ \text{if $\s \not = \frac{1}{2}$}.\\ 
  \end{aligned}
\right.
$$
The claim follows.
\end{proof}

We are now in a position to prove Theorem~\ref{Theorem B}.

\begin{proof}[Proof of Theorem~\ref{Theorem B}.] 
Let $u=\Phi_\mu$ and  $\xi\in C^2_c(\R^N)$ such that 
supp$\,\xi\subset B_R$ with $R>0$. Moreover, put $\psi=\xi\Gamma_\mu$, and let $\eps\in(0,\, \frac14)$. Then we have
\begin{align}
0=&\int_{\R^N\setminus B_\epsilon } [\cL^\s_\mu u]\psi\, dx =\int_{\R^N\setminus B_\epsilon } (-\Delta)^\s u \psi\,dx+\int_{\R^N\setminus B_\epsilon }  \frac{\mu}{|x|^{2\s}} u \psi\,dx\nonumber\\ =&\int_{\R^N\setminus B_\epsilon }  \frac{\mu}{|x|^{2\s}} u \psi\,dx+ \frac{C_{N,\s}}2 \int_{\R^N\setminus B_\epsilon } \int_{\R^N\setminus B_\epsilon }\frac{[u(x)-u(y)][\psi(x)-\psi(y)]}{|x-y|^{N+2\s}}dy dx\nonumber\\
&+  C_{N,\s} \int_{\R^N\setminus B_\epsilon } \int_{B_\epsilon }\frac{u(x)-u(y)}{|x-y|^{N+2\s}}\psi(x) dy dx\label{2.0} \\
   =&  \int_{\R^N\setminus B_\epsilon } u \cL^\s_\mu \psi\,dx +  C_{N,\s}  \int_{\R^N\setminus B_\epsilon } \int_{B_\epsilon }\frac{u(x)-u(y)}{|x-y|^{N+2\s}}\psi(x) dy dx -   C_{N,\s} \int_{\R^N\setminus B_\epsilon } \int_{B_\epsilon }\frac{\psi(x)-\psi(y)}{|x-y|^{N+2\s}}u(x) dy dx,\nonumber
\end{align}
which by (\ref{eq:adjoint-op-pointwise}) implies that
\begin{align}
&\int_{\R^N\setminus B_\epsilon } u (-\Delta)^\s_{\Gamma_\mu} \xi\,dx =
\int_{\R^N\setminus B_\epsilon } u \cL^\s_\mu \psi\,dx\nonumber \\
=\ \ &C_{N,\s} \int_{\R^N\setminus B_\epsilon } \int_{B_\epsilon }\frac{\psi(x)-\psi(y)}{|x-y|^{N+2\s}}u(x) dy dx-C_{N,\s} \int_{\R^N\setminus B_\epsilon } \int_{B_\epsilon }\frac{u(x)-u(y)}{|x-y|^{N+2\s}}\psi(x) dy dx. \label{2.3.1}
\end{align}
When $\mu>\mu_0$,
\begin{align*}
 &\int_{\R^N\setminus B_\epsilon } \int_{B_\epsilon }\frac{u(x)-u(y)}{|x-y|^{N+2\s}}\psi(x) dy dx =
\int_{\R^N\setminus B_\epsilon } |x|^{\tau_+(\s,\mu)}\xi(x) \int_{B_\epsilon }\frac{|x|^{\tau_-(\s,\mu)}-|y|^{\tau_-(\s,\mu)}}{|x-y|^{N+2\s}}  dy dx \\
   &= \int_{\R^N\setminus B_\epsilon }\frac{\xi(x)}{|x|^N} \int_{B_{\frac\epsilon{|x|}} }\frac{1-|z|^{\tau_-(\s,\mu)}}{|e_1-z|^{N+2\s}}  dz dx
   \\   &= \int_{\R^N\setminus B_{\sqrt{\epsilon}} }\frac{\xi(x)}{|x|^N} \int_{B_{\frac\epsilon{|x|}} }\frac{1-|z|^{\tau_-(\s,\mu)}}{|e_1-z|^{N+2\s}}  dz dx+ \int_{B_{\sqrt{\epsilon}} \setminus B_\epsilon } \frac{\xi(x)}{|x|^N}\int_{B_{\frac\epsilon{|x|}} }\frac{1-|z|^{\tau_-(\s,\mu)}}{|e_1-z|^{N+2\s}}  dz dx,
\end{align*}
where $e_1=(1,0,\cdots,0)\in\R^N$. Setting 
$h_1(t):= \int_{B_{t} } \frac{1-|z|^{\tau_-(\s,\mu)}}{|e_1-z|^{N+2\s}}  dz $, we now deduce from (\ref{eq:h-general-estimate-1}) that 
\begin{equation}
  \label{eq:h_1-pointwise-est}
|h_1(t)|\le c\,t^{N+\tau_-(\s,\mu)} \qquad \text{for any $t\le \sqrt{\epsilon}$  with a constant $c>0$ independent of $\epsilon \in (0,\frac{1}{4})$.}   
\end{equation}
Therefore,
\begin{align*}
 &\Big| \int_{\R^N\setminus B_{\sqrt\epsilon} } \frac{\xi(x)}{|x|^N} \int_{B_{\frac\epsilon{|x|}} }\frac{1-|z|^{\tau_-(\s,\mu)}}{|e_x-z|^{N+2\s}}  dz dx\Big|
\le \int_{\R^N\setminus B_{\sqrt\epsilon} }  \frac{|\xi(x)|}{|x|^{N}} \bigl|h_1(\frac\epsilon{|x|})\bigr| dx \\
    &\le c\,\omega_{N-1}\, \|\xi\|_{L^\infty(\R^N)}\epsilon^{ N+\tau_-(\s,\mu)}\, 
\int_{\sqrt{\epsilon}}^{R} t^{-N-1-\tau_-(\s,\mu)} dt \leq O(\epsilon^{ \frac{N+\tau_-(\s,\mu)}{2}}) \:\to\:  0\quad {\rm as}\ \ \epsilon\to0
\end{align*}
and
\begin{align}
&\int_{ B_{\sqrt\epsilon} \setminus B_{\epsilon} }\frac{\xi(x)}{|x|^N} \int_{B_{\frac\epsilon{|x|}} }\frac{1-|z|^{\tau_-(\s,\mu)}}{|e_x-z|^{N+2\s}}  dz dx
= [\xi(0) +O(\sqrt\epsilon)]\int_{B_{\sqrt\epsilon} \setminus B_{ \epsilon} } h_1(\frac\epsilon{|x|})|x|^{-N}dx \nonumber\\
&= \omega_{N-1}[\xi(0) +O(\sqrt\epsilon)] \int_\epsilon^{\sqrt\epsilon} h_1(\frac{\epsilon}{t}) t^{-1} dt = \omega_{N-1}[\xi(0) +O(\sqrt\epsilon)] 
\int^1_{\sqrt{\eps}} \frac{h_1(s)}{s} ds\\
& \to\: \omega_{N-1} \xi(0) \int^1_0 \frac{h_1(s)}{s} ds
    \quad {\rm as}\ \ \epsilon\to 0, \label{h-1-convergence}
\end{align}
where $N+\tau_-(\s,\mu)>0$. Hence 
\begin{equation}\label{2.1}
\lim_{\epsilon\to0}\int_{\R^N\setminus B_\epsilon } \int_{B_\epsilon }\frac{u(x)-u(y)}{|x-y|^{N+2\s}}\psi(x) dy dx=\omega_{N-1} \xi(0) \int^1_0 \frac{h_1(s)}{s} ds.
\end{equation}
Now we deal with the last term in  (\ref{2.0}). By direct computation, we have that
\begin{align}
&\int_{\R^N\setminus B_\epsilon } \int_{B_\epsilon }\frac{\psi(x)-\psi(y)}{|x-y|^{N+2\s}}u(x) dy dx  = \int_{\R^N\setminus B_\epsilon } \int_{B_\epsilon }\frac{|x|^{\tau_+(\s,\mu)}-|y|^{\tau_+(\s,\mu)}}{|x-y|^{N+2\s}}|x|^{\tau_-(\s,\mu)}\xi(x) dy dx \nonumber
\\
&+\int_{\R^N\setminus B_\epsilon } \int_{B_\epsilon }\frac{\xi(x)-\xi(y)}{|x-y|^{N+2\s}}|x|^{\tau_-(\s,\mu)}|y|^{\tau_+(\s,\mu)} dy dx \nonumber\\
&=\int_{\R^N\setminus B_\epsilon }\frac{\xi(x)}{|x|^N} \int_{B_\frac{\epsilon}{|x|} }\frac{1-|z|^{\tau_+(\s,\mu)}}{|e_1-z|^{N+2\s}} dy dx
+\int_{\R^N\setminus B_\epsilon } \int_{B_\epsilon }\frac{\xi(x)-\xi(y)}{|x-y|^{N+2\s}}|x|^{\tau_-(\s,\mu)}|y|^{\tau_+(\s,\mu)} dy dx , \label{extra-eq-distr-id-proof}
\end{align}
where
\begin{align*}
\Big|\int_{\R^N\setminus B_\epsilon } \int_{B_\epsilon }&\frac{\xi(x)-\xi(y)}{|x-y|^{N+2\s}}|x|^{\tau_-(\s,\mu)}|y|^{\tau_+(\s,\mu)} dy dx\Big|
  \le \norm{\xi}_{C^2}\int_{\R^N\setminus B_\epsilon } \int_{B_\epsilon } \frac{|x|^{\tau_-(\s,\mu)}|y|^{\tau_+(\s,\mu)}}{|x-y|^{N+2\s-2}}dydx \\
&= \norm{\xi}_{C^2}\Bigl(\int_{\R^N\setminus B_{2\epsilon }} \int_{B_\epsilon }+\int_{  B_{2\epsilon }\setminus  B_{\epsilon }} \int_{B_\epsilon } \Bigr)\frac{|x|^{\tau_-(\s,\mu)}|y|^{\tau_+(\s,\mu)}}{|x-y|^{N+2\s-2}} dydx =  O(\epsilon^2)
\end{align*}
since
\begin{align*}
\int_{\R^N\setminus B_{2\epsilon }} \int_{B_\epsilon }\frac{|x|^{\tau_-(\s,\mu)}|y|^{\tau_+(\s,\mu)}}{|x-y|^{N+2\s-2}}dydx
&=  \int_{\R^N\setminus B_{2\epsilon }}|x|^{2-N} \int_{B_{\frac{\epsilon}{|x|}} }\frac{|z|^{\tau_+(\s,\mu)}}{|e_1-z|^{N+2\s-2}} dz dx  \\
   &\le  2^{N+2\s-2} \int_{\R^N\setminus B_{2\epsilon }}(\frac{\epsilon}{|x|})^{N+\tau_+(\s,\mu)} |x|^{2-N} dx =  O(\epsilon^2)
\end{align*}
and
\begin{align*}
\int_{ B_{2\epsilon} \setminus B_\epsilon} \int_{B_\epsilon}
\frac{|x|^{\tau_-(\s,\mu)}|y|^{\tau_+(\s,\mu)}}{|x-y|^{N+2\s-2}}dydx
&\le  \int_{ B_{2\epsilon} \setminus B_\epsilon} |x|^{2-N} \left(\int_{B_1}
|z|^{\tau_+(\s,\mu)}dz+\int_{B_{1/2}}\frac{dz}{|e_1-z|^{N+2\s-2}}\right) dx  \\
   &\le O(1) \int_{ B_{2\epsilon} \setminus B_\epsilon}  |x|^{2-N} dx = O(\epsilon^2).
\end{align*}
To compute the first term in (\ref{extra-eq-distr-id-proof}), we let 
$h_2(t):= \int_{B_{t}} \frac{1-|z|^{\tau_+(\s,\mu)}}{|e_1-z|^{N+2\s}}  dz $, and we note that 
\begin{equation}
  \label{eq:h-2-pointwise-estimate}
|h_2(t)|\le c\bigl(t^{N}+t^{N+\tau_+(\s,\mu)}\bigr) \qquad \text{for any $t\le \sqrt{\epsilon}$ with $c>0$ independent of $t$ and $\epsilon \in (0,\frac{1}{4})$}   
\end{equation}
by (\ref{eq:h-general-estimate-1}). Therefore 
$$
 \Big| \int_{\R^N\setminus B_{\sqrt\epsilon}}\frac{\xi(x)}{|x|^N} \int_{B_{\frac\epsilon{|x|}}}\frac{1-|z|^{\tau_+(\s,\mu)}}{|e_1-z|^{N+2\s}}  dz dx\Big|
 \le   \int_{\R^N\setminus B_{\sqrt\epsilon}} \bigl|h_2(\frac\epsilon{|x|})\bigr|\frac{\xi(x)}{|x|^N} dx \le O(\epsilon^N+\epsilon^{N+\tau_+(\s,\mu)} ) \:\to\: 0 \qquad \text{as $\epsilon\to 0$}
$$
and, similarly as in (\ref{h-1-convergence}), 
$$
 \int_{B_{\sqrt\epsilon}\setminus B_{\epsilon}}
\frac{\xi(x)}{|x|^N} \int_{B_{\frac\epsilon{|x|}}}\frac{1-|z|^{\tau_+(\s,\mu)}}{|e_1-z|^{N+2\s}}  dz dx
 \:\to\: \omega_{N-1}\xi(0)  \int^1_0 \frac{h_2(s)}{s} ds \qquad \text{as $\epsilon\to0.$}
$$
Consequently, 
\begin{equation}\label{2.2}
\lim_{\epsilon\to0}\int_{\R^N\setminus B_\epsilon} \int_{B_\epsilon}\frac{\psi(x)-\psi(y)}{|x-y|^{N+2\s}}u(x) dy dx=\omega_{N-1} \int^1_0 h_2(s) s^{-1} dt\, \xi(0).
\end{equation}

Combining (\ref{2.3.1}), (\ref{2.1}) and (\ref{2.2}), we conclude that 
$$
\frac{1}{C_{N,\s}}\int_{\R^N}\psi_\mu   (-\Delta)^\s_{\Gamma_\mu} \xi \, dx =  \omega_{N-1} \xi(0) \int^1_0 \frac{h_2(s)-h_1(s)}{s}ds= \omega_{N-1}\xi(0) \int^1_0 \int_{B_{t}} \frac{|z|^{\tau_-(\s,\mu)}-|z|^{\tau_+(\s,\mu)}}{|e_1-z|^{N+2\s}}  dzdt,
$$
so (\ref{1.2}) holds true for $\mu>\mu_0$.\\[0.2cm]

When $\mu=\mu_0$, we write 
\begin{align*}
&\int_{\R^N\setminus B_\epsilon} \int_{B_\epsilon}\frac{u(x)-u(y)}{|x-y|^{N+2\s}}\psi(x) dy dx\\ 
=&\int_{\R^N\setminus B_\epsilon} \int_{B_\epsilon}\frac{-|x|^{\tau_-(\s,\mu_0)}\ln|x|+|y|^{\tau_-(\s,\mu_0)}\ln|y|}{|x-y|^{N+2\s}} |x|^{\tau_+(\s,\mu_0)}\xi(x) dy dx \\
   =& \int_{\R^N\setminus B_\epsilon} \int_{B_{\frac\epsilon{|x|}}}\left(\frac{1-|z|^{\tau_-(\s,\mu_0)}}{|e_x-z|^{N+2\s}} |x|^{-N}\ln|x| + \frac{ |z|^{\tau_-(\s,\mu_0)}\ln|z|}{|e_x-z|^{N+2\s}} |x|^{-N} \right) \xi(x) dz dx
   \\  =& \left(\int_{\R^N\setminus B_{\sqrt{\epsilon}} } \int_{B_{\frac\epsilon{|x|}} }+ \int_{B_{\sqrt{\epsilon}} \setminus B_\epsilon } \int_{B_{\frac\epsilon{|x|}}(0)}\right)
    \frac{1-|z|^{\tau_-(\s,\mu_0)}}{|e_x-z|^{N+2\s}} \frac{(-\ln|x|)\xi(x)}{|x|^{N}}dz dx \\
&-
     \left(\int_{\R^N\setminus B_{\sqrt{\epsilon}} } \int_{B_{\frac\epsilon{|x|}} }+ \int_{B_{\sqrt{\epsilon}} \setminus B_\epsilon } \int_{B_{\frac\epsilon{|x|}} }\right)\frac{ |z|^{\tau_-(\s,\mu_0)}(-\ln|z|)}{|e_x-z|^{N+2\s}} \frac{\xi(x)}{|x|^{N}}  dz dx.
\end{align*}
As $ \epsilon\to0$, we have, by (\ref{eq:h_1-pointwise-est}),
\begin{align*}
&\Big| \int_{\R^N\setminus B_{\sqrt\epsilon} } \int_{B_{\frac\epsilon{|x|}} }\frac{1-|z|^{\tau_-(\s,\mu_0)}}{|e_x-z|^{N+2\s}}\frac{(-\ln|x|)\xi(x)}{|x|^{N}} dz dx\Big| \le  \int_{\R^N\setminus B_{\sqrt\epsilon} } \bigl|h_1(\frac\epsilon{|x|})\bigr| \frac{|\xi(x)|}{|x|^{N}} dx \\
    &\le c \|\xi\|_{L^\infty(\R^N)} 
\epsilon^{ N+\tau_-(\s,\mu_0)}    \int_{\sqrt{\epsilon}}^{R}t^{-N-1-\tau_-(\s,\mu_0)}| \ln t| dt = O\Big(\epsilon^{\frac{N+\tau_-(\s,\mu_0)}2}|\ln \epsilon|\Big)
\end{align*}
and
\begin{align*}
&  \int_{ B_{\sqrt\epsilon}\setminus B_{\epsilon}} \int_{B_{\frac\epsilon{|x|}}}\frac{1-|z|^{\tau_-(\s,\mu_0)}}{|e_x-z|^{N+2\s}} |x|^{-N}(-\ln|x|)\xi(x) dz dx
= [\xi(0) +O(\sqrt\epsilon)]\int_{B_{\sqrt\epsilon}\setminus B_{ \epsilon}} h_1(\frac\epsilon{|x|})|x|^{-N}(-\ln|x|)dx\\ 
&= [\xi(0) +O(\sqrt\epsilon)] \int_\epsilon^{\sqrt\epsilon} h_1(\frac{\epsilon}{t}) t^{-1}(-\ln t) dt = [\xi(0) +O(\sqrt\epsilon)]  \int^1_{\sqrt{\eps}} \frac{h_1(s)}{s} \bigl(-\ln \frac{\eps}{s}\bigr) ds
\end{align*}
Moreover, with $h_3(t):= \int_{B_{t}} \frac{ |z|^{\tau_-(\s,\mu_0)} \ln|z|}{|e_1-z|^{N+2\s}}  dz $, 
\begin{align*}
&\Big| \int_{\R^N\setminus B_{\sqrt\epsilon}} \int_{B_{\frac\epsilon{|x|}}}\frac{ |z|^{\tau_-(\s,\mu_0)}\ln|z|}{|e_x-z|^{N+2\s}} |x|^{-N} \xi(x) dz dx\Big|
\le   \int_{\R^N\setminus B_{\sqrt\epsilon}} h_3(\frac\epsilon{|x|})|x|^{-N}|\xi(x)| dx \\
    &\le c\,\omega_{N-1}\, \|\xi\|_{L^\infty(\R^N)}\,  \eps^{N+\tau_-(\s,\mu_0)} 
|\ln \eps| \int_{\sqrt\epsilon}^R t^{-N-1-\tau_-(\s,\mu_0)}|\ln t| \,dt
=O(\epsilon^{\frac{N+\tau_-(\s,\mu_0)}2}|\ln \epsilon|^2)
\end{align*}
as $\eps \to 0$. Here we used the fact that 
$$
|h_3(t)|\le c\, t^{N+\tau_-(\s,\mu_0)}|\ln t| \qquad \text{for any $t\le \sqrt{\epsilon}$ with a constant $c>0$ independent of $\eps \in (0,\frac{1}{4})$.}
$$
Moreover,
\begin{align*}
 \int_{ B_{\sqrt\epsilon}\setminus B_{\epsilon}} \int_{B_{\frac\epsilon{|x|}}}\frac{|z|^{\tau_-(\s,\mu)}\ln|z|}{|e_x-z|^{N+2\s}} |x|^{-N} \xi(x) dz dx
 &=  [\xi(0) +O(\sqrt\epsilon)]\int_{B_{\sqrt\epsilon}(0)\setminus B_{ \epsilon}(0)} h_3(\frac\epsilon{|x|})|x|^{-N}dx \\
    &= \omega_{N-1}[\xi(0) +O(\sqrt\epsilon)] \int_1^{\epsilon^{-\frac12}} h_3(\frac1t) t^{-1}  dt
    \quad {\rm as}\ \, \epsilon\to0^+.
\end{align*}
Then as $\epsilon\to0^+$ we have that 
\begin{align}
 \int_{\R^N\setminus B_\epsilon} \int_{B_\epsilon}\frac{u(x)-u(y)}{|x-y|^{N+2\s}}\psi(x) dy dx\nonumber&= O\Big(\epsilon^{\frac{N+\tau_-(\s,\mu_0)}2}(-\ln \epsilon)\Big)+\omega_{N-1}\Big(\xi(0) +O(\sqrt\epsilon)\Big) \int_1^{\epsilon^{-\frac12}} h_3(\frac1t) t^{-1}  dt 
\\& +  \omega_{N-1}\Big(\xi(0) +O(\sqrt\epsilon)\Big)  \int_1^{\epsilon^{-\frac12}} h_1(\frac1t) t^{-1} (- \ln (\epsilon t))\,  dt.\label{2.1.0}
\end{align}

Now we deal with the last term in  (\ref{2.0}). Observe that
\begin{align*}
\int_{\R^N\setminus B_\epsilon} \int_{B_\epsilon}\frac{\psi(x)-\psi(y)}{|x-y|^{N+2\s}}u(x) dy dx  &= \int_{\R^N\setminus B_\epsilon} \int_{B_\epsilon}\frac{|x|^{\tau_+(\s,\mu)}-|y|^{\tau_+(\s,\mu)}}{|x-y|^{N+2\s}}|x|^{\tau_-(\s,\mu)}(-\ln|x|)\xi(x) dy dx
\\&+\int_{\R^N\setminus B_\epsilon} \int_{B_\epsilon}\frac{\xi(x)-\xi(y)}{|x-y|^{N+2\s}}|x|^{\tau_-(\s,\mu)}(-\ln|x|)|y|^{\tau_+(\s,\mu)} dy dx\\
  &=\int_{\R^N\setminus B_\epsilon} \int_{B_\frac{\epsilon}{|x|}}\frac{1-|z|^{\tau_+(\s,\mu)}}{|e_x-z|^{N+2\s}}|x|^{-N}(-\ln|x|)\xi(x) dy dx\\
  &+\int_{\R^N\setminus B_\epsilon} \int_{B_\epsilon}\frac{\xi(x)-\xi(y)}{|x-y|^{N+2\s}}|x|^{\tau_-(\s,\mu)}(-\ln|x|)|y|^{\tau_+(\s,\mu)} dy dx ,
\end{align*}
where
\begin{align*}
& \Big|\int_{\R^N\setminus B_\epsilon} \int_{B_\epsilon}\frac{\xi(x)-\xi(y)}{|x-y|^{N+2\s}}|x|^{\tau_-(\s,\mu)}(-\ln|x|)|y|^{\tau_+(\s,\mu)} dy dx\Big|
 \\ & \le\norm{\xi}_{C^2}\int_{\R^N\setminus B_\epsilon} \int_{B_\epsilon} |x-y|^{-N-2\s+2}|x|^{\tau_-(\s,\mu)}(-\ln|x|)|y|^{\tau_+(\s,\mu)} dydx
 \\&=\norm{\xi}_{C^2}(\int_{\R^N\setminus B_{2\epsilon}} \int_{B_\epsilon}+\int_{  B_{2\epsilon}\setminus  B_{\epsilon}} \int_{B_\epsilon} ) |x-y|^{-N-2\s+2}|x|^{\tau_-(\s,\mu)}|y|^{\tau_+(\s,\mu)} dydx,
\end{align*}
since
\begin{align*}
&\int_{\R^N\setminus B_{2\epsilon}} \int_{B_\epsilon} |x-y|^{-N-2\s+2}|x|^{\tau_-(\s,\mu)}(-\ln|x|)|y|^{\tau_+(\s,\mu)} dydx
 \\&= \int_{\R^N\setminus B_{2\epsilon}} \int_{B_{\frac{\epsilon}{|x|}}}|e_x-z|^{-N-2\s+2}|z|^{\tau_+(\s,\mu)} |x|^{2-N}(-\ln|x|)dzdx  \\
   &\le  2^{N+2\s-2} \int_{\R^N\setminus B_{2\epsilon}}(\frac{\epsilon}{|x|})^{N+\tau_+(\s,\mu)} |x|^{2-N}(-\ln|x|) dx \,\le \, c\epsilon^2
\end{align*}
and
\begin{align*}
&\int_{ B_{2\epsilon} \setminus B_\epsilon} \int_{B_\epsilon} |x-y|^{-N-2\s+2}|x|^{\tau_-(\s,\mu)}(-\ln|x|)|y|^{\tau_+(\s,\mu)} dydx
 \\&\le  \int_{ B_{2\epsilon} \setminus B_\epsilon}\left (\int_{B_1}|z|^{\tau_+(\s,\mu)}dz+\int_{B_{1/2}}|e_x-z|^{-N-2\s+2}dz\right) |x|^{2-N}(-\ln|x|)dx  \\
   &\le  c\int_{ B_{2\epsilon} \setminus B_\epsilon}  |x|^{2-N}(-\ln|x|) dx \,\le\,  c\epsilon^2.
\end{align*}
Therefore we have that
$$\Big|\int_{\R^N\setminus B_\epsilon} \int_{B_\epsilon}\frac{\xi(x)-\xi(y)}{|x-y|^{N+2\s}}|x|^{\tau_-(\s,\mu)}(-\ln|x|)|y|^{\tau_+(\s,\mu)} dy dx\Big|\le c\epsilon^2.$$

Moreover,
\begin{align*}
 \Big| \int_{\R^N\setminus B_{\sqrt\epsilon}} \int_{B_{\frac\epsilon{|x|}}(0)}\frac{1-|z|^{\tau_+(\s,\mu)}}{|e_x-z|^{N+2\s}} |x|^{-N}(-\ln|x|)\xi(x) dz dx\Big|
&\le   \int_{\R^N\setminus B_{\sqrt\epsilon}} h_2(\frac\epsilon{|x|})|x|^{-N}|(\ln|x|)|\, |\xi(x)| dx
    \\&= O\Big(\epsilon^{\frac{N+\tau_+(\s,\mu_0)}2}(-\ln \epsilon)\Big)\quad {\rm as}\ \, \epsilon\to0^+
\end{align*}
and
\begin{align*}
  \int_{ B_{\sqrt\epsilon}\setminus B_{\epsilon}} \int_{B_{\frac\epsilon{|x|}}}\frac{1-|z|^{\tau_+(\s,\mu)}}{|e_x-z|^{N+2\s}} &|x|^{-N}(-\ln|x|)\xi(x) dz dx =  [\xi(0) +O(\sqrt\epsilon)]\int_{B_{\sqrt\epsilon}\setminus B_{ \epsilon}} h_2(\frac\epsilon{|x|})|x|^{-N}(-\ln|x|) dx
  \\ &\longrightarrow \,\omega_{N-1}\Big(\xi(0) +O(\sqrt\epsilon)\Big)\int_1^{\epsilon^{-\frac12}} h_2(\frac1t) t^{-1} (- \ln (\epsilon t))\,  dt 
    \quad {\rm as}\ \, \epsilon\to0^+.
\end{align*}
Hence as $\epsilon\to0^+$ we have that
\begin{align}
\int_{\R^N\setminus B_\epsilon} \int_{B_\epsilon}\frac{\psi(x)-\psi(y)}{|x-y|^{N+2\s}}u(x) dy dx&=O(\epsilon^2)+O\Big(\epsilon^{\frac{N+\tau_+(\s,\mu_0)}2}(-\ln \epsilon)\Big)  \nonumber 
\\&+  \omega_{N-1}\Big(\xi(0) +O(\sqrt\epsilon)\Big)  \int_1^{\epsilon^{-\frac12}} h_2(\frac1t) t^{-1} (- \ln (\epsilon t))\,  dt.\label{2.2.0}
\end{align}
We note that $h_1=h_2$ for $\mu=\mu_0$, and by  (\ref{2.3.1}), (\ref{2.1.0}) and (\ref{2.2.0}), we have that
\begin{align*}
\int_{\R^N\setminus B_\epsilon}\psi_\mu   (-\Delta)^\s_{\Gamma_{\mu_0}}\xi \, dx &=O(\epsilon^2)+O\Big(\epsilon^{\frac{N+\tau_+(\s,\mu_0)}2}(-\ln \epsilon)\Big) +O(\sqrt\epsilon \ln \epsilon)\\&\ \ +C_{N,\s}\omega_{N-1}\Big(\xi(0) 
+O(\sqrt\epsilon)\Big) \int_1^{\epsilon^{-\frac12}} h_3(\frac1t) t^{-1} dt,\end{align*}
which, passing to the limit as $\epsilon\to0^+$,  implies that \begin{align*}
\int_{\R^N}\psi_\mu   (-\Delta)^\s_{\Gamma_{\mu_0}}\xi \, dx =-C_{N,\s}\omega_{N-1}   \int^1_0 h_3(s) s^{-1} dt  \,\xi(0) =  C_{N,\s} \omega_{N-1} \int^1_0 \int_{B_{t}} \frac{ |z|^{\tau_-(\s,\mu)}(-\ln|z|) }{|e_x-z|^{N+2\s}}  dzdt\, \xi(0),
\end{align*}
then (\ref{1.2}) holds true for $\mu=\mu_0$.
\end{proof}

\setcounter{equation}{0}
\section{The Dirichlet problem in a bounded domain}
\label{sec:dirichl-probl-bound}

This section is devoted to the analysis of the Dirichlet problem~(\ref{eq 2.1fk}). In particular, we shall prove Parts (i) and (ii) of Theorem~\ref{theorem-C}. We first need to set up a functional analytic framework for singular weak solutions.

\subsection{Function spaces, weak comparison principles, and singular weak solutions}

\label{sec:preliminary-notation}

We denote the quadratic form
$$
(u,v) \mapsto \cE^\s(u,v) = \frac{C_{N,\s}}{2}\int_{\R^N \times \R^N} \frac{(u(x)-u(y))(v(x)-v(y))}{|x-y|^{N+2\s}}dx dy,
$$
associated to $(-\Delta)^\s$ and the space
$$
\cH^\s(\R^N):= \{u \in L^2_{loc}(\R^N)\::\: \cE^{\s}(u,u)<\infty\}.  
$$
By the fractional Sobolev inequality, we have that
\begin{equation}
  \label{eq:Sobolev-embeddings-0}
\cH^\s(\R^N) \subset L^{\frac{2N}{N-2\s}}(\R^N).
\end{equation}
By density, it follows that the fractional Hardy inequality (\ref{eq:fractional-hardy-C-infty}) extends to functions $u,v \in \cH^\s(\R^N)$, which implies that the quadratic form
$$
\cE^\s_\mu(u,v) := \cE^\s(u,v) + \mu \int_{\R^N}\frac{uv}{|x|^{2\s}}\,dx,
$$
is well-defined for $u,v \in \cH^\s(\R^N)$ and $\mu>\mu_0$. In the following,
for measurable sets $A,B \subset \R^N$ and measurable functions $u,v: A \cup B \to \R$, we put
$$
\cE^{\s}_{A,B}(u,v) = \frac{C_{N,\s}}{2} \int_{A \times B} \frac{(u(x)-u(y))(v(x)-v(y))}{|x-y|^{N+2\s}}dx dy
$$
and $\cE^{\s}_A(u,v):= \cE^\s_{A,A}(u,v)$ whenever the integral are well defined in Lebesgue sense.

Next, let $\Omega \subset \R^N$ denote a bounded domain which may or may not contain the origin,  the space
$$
\cH^\s_0(\Omega):= \{u \in \cH^\s(\R^N)\::\: u \equiv 0 \text{ on $\R^N
\setminus \Omega$}\}.
$$
By (\ref{eq:Sobolev-embeddings-0}), we then have that
\begin{equation}
  \label{eq:Sobolev-embeddings}
\cH^\s_0(\Omega) \subset L^{s}(\Omega) \quad \text{for \,$1\leq s \le \frac{2N}{N-2\s}$.}
\end{equation}
We also recall the following fractional Hardy inequality with remainder term, see \cite[Theorem 2.3]{FLS-2008}: there exists a constant $C$ such that
\begin{equation}
  \label{eq:fractional-hardy-remainder}
\cE^\s_{\mu_0}(u,u) \ge C \|u\|_{L^2(\Omega)}^2\quad \text{for \,$u \in \cH^\s_0(\Omega)$.}
\end{equation}

In the following, it will be convenient to set
$$
\Omega_\eps:= \Omega \setminus \overline{B_\eps} \qquad \text{for $\eps>0$.}
$$

\begin{definition}
\label{def-W-alpha}
We define
$W^\s(\Omega)$ as the space of functions
$u \in L^1(\R^N, \frac{dx}{1+|x|^{N+2\s}})$ with $u\big|_{\Omega} \in L^2(\Omega)$ and the property that
$$
\cE^\s_{\Omega,\Omega'}(u,u)<\infty\quad \text{for some domain $\Omega' \subset \R^N$ with $\Omega \subset \subset \Omega'$.}
$$
Moreover, we let $W^\s_*(\Omega)$ denote the space of functions
$u \in L^1(\R^N, \frac{dx}{1+|x|^{N+2\s}})$ such that $u|_{\Omega_\eps} \in W^\s(\Omega_\eps)$ for every $\eps>0$.
\end{definition}

\begin{remark}
\label{gamma-mu-W-alpha-remark}
Let, as before, $\Omega \subset \R^N$ be a bounded domain. 
\begin{enumerate}
\item[(i)] Since $\cH^\s_0(\Omega) \hookrightarrow L^{2}(\Omega)$ by (\ref{eq:Sobolev-embeddings}) and the fact that $\Omega$ is bounded, we have $\cH^\s_0(\Omega) \subset W^\s(\Omega)$. Moreover, every function $u \in L^1(\R^N, \frac{dx}{1+|x|^{N+2\s}})$ which is of class $C^{\s+\eps}$ in a neighborhood of $\bar \Omega$ for some $\eps>0$ is contained in $W^\s(\Omega)$.
\item[(ii)] If $u \in L^1(\R^N, \frac{dx}{1+|x|^{N+2\s}}) \cap C^{\s+\eps}_{loc}(\R^N \setminus \{0\})$ for some $\eps>0$, then 
$u \in W^\s(\Omega_\eps)$ for every $\eps>0$ by $(i)$ and therefore $u \in W^\s_*(\Omega)$. In particular, this implies that  
$$
\Phi_\mu, \Gamma_\mu  \in W^\s_*(\Omega) \quad\ \text{for \, $\mu_0 \le \mu < \infty.$}
$$
On the other hand, if $0  \in \Omega$, it is obvious that $\Phi_\mu \not \in W^\s(\Omega)$ for any $\mu \in [\mu_0,\infty)$.
With regard to $\Gamma_\mu$, we have the following key lemma.
\end{enumerate}
\end{remark}

\begin{lemma}
\label{gamma-mu-W-alpha}
\begin{itemize}
\item[(i)] $\Gamma_\mu \in W^\s(\Omega)$ for $\mu>\mu_0$.
\item[(ii)] If $0 \in \Omega$, then $\Gamma_{\mu_0}\not \in W^\s(\Omega)$.
\end{itemize}
\end{lemma}

\begin{proof}  
Since $\frac{2\s-N}{2} \le \tau_+(\s,\mu)<2\s$ for $\mu \in [\mu_0,\infty)$ and $\Omega$ is bounded, we clearly have$$
\Gamma_\mu \in L^2(\Omega) \cap L^1\bigl(\R^N, \frac{dx}{1+|x|^{N+2\s}}\bigr) \qquad \text{for $\mu \in [\mu_0,\infty)$.}
$$
For $x \in \R^N \setminus \{0\}$ and $R>0$, then we have that
\begin{equation}
  \label{eq:gamma-mu-W-alpha-1}
\int_{B_R}\frac{(\Gamma_\mu(x)-\Gamma_\mu(y))^2}{|x-y|^{N+2\s}}dy =
 |x|^{2\tau_+(\s,\mu)-2\s} \int_{B_{\frac{R}{|x|}}} \frac{(1-|z|^{\tau_+(\s,\mu)})^2}{|e_x-z|^{N+2\s}}dz
\end{equation}
and
\begin{align}
&\int_{B_{\frac{R}{|x|}}} \frac{(1-|z|^{\tau_+(\s,\mu)})^2}{|e_x-z|^{N+2\s}}dz \le c \int_{B_{\frac{R}{|x|}}} \frac{\min \{(1-|z|)^2, (1-|z|)^{2\tau_+(\s,\mu)}\}}{|e_x -z|^{N+2\s}}dz  \nonumber\\
&\le c \int_{B_{\frac{R}{|x|}}} \frac{\min \{|e_x-z|^2,|e_x-z|^{2\tau_+(\s,\mu)}\}}{|e_x -z|^{N+2\s}}dz \nonumber \le c \int_{B_{\frac{R}{|x|}+1}} \min \bigl\{ |z|^{2-2\s-N},|z|^{2\tau_+(\s,\mu)-2\s-N} \bigr\} d z \nonumber\\
&\le c
\left\{\arraycolsep=1pt
\begin{array}{lll}
  1 \quad
   &{\rm if}\ \, \tau_+(\s,\mu) <\s ,\\[1mm]
 \phantom{   }
 1+ \ln\frac{R}{|x|} \quad &{\rm   if } \ \,  \tau_+(\s,\mu) =\s,\\[1mm]
 \phantom{   }
 1+ \Bigl(\frac{|x|}{R}\Bigr)^{2\s-2\tau_+(\s,\mu)} \quad &{\rm   if } \ \,  \tau_+(\s,\mu) >\s.
 \end{array}
 \right.   \label{eq:gamma-mu-W-alpha-2}
\end{align}
If $\mu> \mu_0$, we have $2\tau_+(\s,\mu)-2\s>-N$ and therefore, by combining (\ref{eq:gamma-mu-W-alpha-1}) and (\ref{eq:gamma-mu-W-alpha-2}),
$$
\int_{B_R} \int_{B_R}\frac{(\Gamma_\mu(x)-\Gamma_\mu(y))^2}{|x-y|^{N+2\s}}dydx < \infty
$$
for every $R>0$. Taking $R>0$ sufficiently large such that $\Omega \subset B_R$, we conclude that $\Gamma_\mu \in W^\s(\Omega)$. On the other hand, we have $2\tau_+(\s,\mu_0)-2\s=-N$ and $\tau_+(\s,\mu_0) <\s$, so
(\ref{eq:gamma-mu-W-alpha-1}) yields that
$$
\int_{B_R} \int_{B_R}\frac{(\Gamma_{\mu_0}(x)-\Gamma_{\mu_0}(y))^2}{|x-y|^{N+2\s}}dydx \ge \int_{B_R}
 |x|^{-N} \int_{B_{1}} \frac{(1-|z|^{\tau_+(\s,\mu_0)})^2}{|e_x-z|^{N+2\s}}dz = \infty
$$
for every $R>0$. Hence $\Gamma_{\mu_0} \not \in W^\s(\Omega)$.
\end{proof}

\begin{remark}\label{eq:Delta-alpha-homogeneous}
For $\tau>-\frac{N-2\s}{2}$, we have that $\phi_{\tau}=|\cdot|^\tau\in W^\s(\Omega)$, where $\Omega$ is a  bounded smooth domain.
\end{remark}

\begin{lemma}
\label{u-minus-lemma}
  \begin{enumerate}
\item[(i)] $\cE^\s(u,v)$ is well defined for $u \in W^\s(\Omega)$, $v \in \cH^\s_0(\Omega)$.
\item[(ii)] If $u \in W^{\s}(\Omega)$, then $u^\pm \in W^{\s}(\Omega)$.
\item[(iii)] If $u \in W^{\s}(\Omega)$ satisfies $u \equiv 0$ in $\Omega^c$, then $u \in \cH^\s_0(\Omega)$.
\item[(iv)] If $u \in W^\s(\Omega)$ is such that $u \ge 0$ in $\Omega^c$, then $u^- \in \cH^\s_0(\Omega)$ and
$$
\cE^\s(u^-,u^-) \le -\cE^\s(u,u^-).
$$
\item[(v)]  If $u \in W^\s(\Omega)$ and $\phi \in L^\infty(\R^N)$ is Lipschitz in some neighborhood of $\overline \Omega$,  then $\phi u \in W^\s(\Omega)$.
\item[(vi)] If $0 \in \Omega$ or $0 \not \in \overline \Omega$, we have $W^\s(\Omega) \subset L^2(\Omega,|x|^{-2\s})$.
\end{enumerate}
\end{lemma}

\begin{proof}
Throughout the proof, we use the notation $j_\s(z)=C_{N,\s} |z|^{-N-2\s}$.\smallskip

(i) Let $u\in W^\s(\Omega)$, $v\in \cH^\s_0(\Omega)$, and let $\Omega' \subset \R^N$ be such that $\Omega \subset \subset \Omega'$ and $\cE^\s{\Omega,\Omega'}(u)<\infty$. Then, by the symmetry of the integrand, Fubini's theorem and the Cauchy-Schwarz inequality, we have that
	\begin{align*}
	&\frac{1}{2}\int_{\R^N} \int_{\R^N} |u(x)-u(y)|\cdot|v(x)-v(y)|j_\s(x-y)\ dxdy
\\
&= \frac{1}{2} \Bigl(\int_{\Omega} \int_{\Omega'} \dots + \int_{\Omega} \int_{R^N \setminus \Omega'} + \int_{R^N \setminus
\Omega} \int_{\R^N}\dots\Bigr)= \frac{1}{2} \Bigl(\int_{\Omega} \int_{\Omega'} \dots + \int_{\Omega} \int_{R^N \setminus \Omega'} + \int_{R^N \setminus
\Omega} \int_{\Omega}\dots\Bigr)\\
&= \frac{1}{2} \Bigl(\int_{\Omega} \int_{\Omega'} \dots + \int_{\Omega} \int_{R^N \setminus \Omega'} + \int_{R^N \setminus
\Omega'} \int_{\Omega}\dots+ \int_{\Omega' \setminus
\Omega} \int_{\Omega}\dots \Bigr) \le \int_{\Omega} \int_{\Omega'} \dots + \int_{\Omega} \int_{R^N \setminus \Omega'} \dots\\
&\leq 2 \Bigl(\cE^\s_{\Omega,\Omega'}(u,u) \cE^\s_{\Omega,\Omega'}(v,v)\Bigr)^{1/2}
 +\int_{\Omega} |v(x)| \int_{\R^N\setminus \Omega'} |u(x)-u(y)|j_\s(x-y)\ dydx\\
&\leq 2\Bigl(\cE^\s_{\Omega,\Omega'}(u,u) \cE^{\s}(v,v)\Bigr)^{1/2}
 +c_1 \int_{\Omega} |v(x)||u(x)|\,dx + c_2 \int_{\Omega}|v(x)|\,dx \|u\|_{L^1(\R^N,\frac{dx}{1+|y|^{N+2\s}})}\\
&\le 2\Bigl(\cE^\s_{\Omega,\Omega'}(u,u) \cE^{\s}(v,v)\Bigr)^{1/2}
 +c_1 \|v\|_{L^2(\Omega)} \|u\|_{L^2(\Omega)}+ c_2 |\Omega|^{1/2} \|v\|_{L^2(\Omega)} \|u\|_{L^1(\R^N,\frac{dx}{1+|y|^{N+2\s}})}
\end{align*}
with
$$
c_1 = \sup_{x \in \Omega} \int_{\R^N\setminus \Omega'} j_\s(x-y)\ dy < \infty
$$
and
$$
c_2 = \sup_{x \in \Omega, y \in \R^N \setminus \Omega'} j_\s(x-y)(1+|y|)^{-N-2\s}< \infty.
$$
Here we used the fact that $\Omega$ and $\R^N \setminus \Omega'$ have positive distance. \smallskip

(ii) Let $u \in W^\s(\Omega)$. For $x,y \in \R^N$,  we have that
\begin{equation}
  \label{eq:pointwise}
(u^+(x)-u^+(y))(u^-(x)-u^-(y))=-2\bigl(u^+(x)u^-(y)+u^-(x)u^+(y)\bigr) \le 0.
\end{equation}
Therefore,
\begin{align*}
\cE^\s_{\Omega,\Omega'}(u,u)&=\cE^\s_{\Omega,\Omega'}(u^+-u^-,u^+-u^-)
=\cE^\s_{\Omega,\Omega'}(u^+,u^+)+\cE^\s_{\Omega,\Omega'}(u^-,u^-)\\
&-2 \int_{\Omega' \times \Omega'}(u^+(x)-u^+(y))(u^-(x)-u^-(y))j_\s(x-y)\,dx dy \\
&\ge \cE^\s_{\Omega,\Omega'}(u^+,u^+)+\cE^\s_{\Omega,\Omega'}(u^-,u^-).
\end{align*}
Consequently, $u^\pm \in W^\s(\Omega)$. \smallskip

(iii) Since $u \equiv 0$ in $\Omega^c$, similarly as in the proof of (i), we have that
\begin{align*}
\cE^\s(u,u)&=  \frac{1}{2}\int_{\R^N} \int_{\R^N} (u(x)-u(y))^2|j_\s(x-y)\ dxdy \le \int_{\Omega} \int_{\Omega'} \dots + \int_{\Omega} \int_{R^N \setminus \Omega'} \dots\\
&=2 \cE^\s_{\Omega,\Omega'}(u,u)+ \int_{\Omega}u(x) \int_{\R^N \setminus \Omega'}(u(x)-u(y))j_\s(x-y)dydx \\
&\le 2 \cE^\s_{\Omega,\Omega'}(u,u)+\|u\|_{L^2(\Omega)}^2+ c_2 |\Omega|^{1/2} \|u\|_{L^2(\Omega)} \|u\|_{L^1(\R^N,(1+|x|)^{-N-2\s})}< \infty
\end{align*}
and therefore $u \in \cH^\s_0(\Omega)$.\smallskip

(iv) Since $u^- \in W^{\s}(\Omega)$ by (ii) and $u^- \equiv 0$ on $\Omega^c$ by assumption, we have
$u^- \in \cH^\s_0(\Omega)$ by (iii). Moreover, we have
$$
\cE^\s(u,u^-)= \cE^\s(u^+-u^-,u^-) = \cE^\s(u^+,u-)-\cE^\s(u^-,u^-),
$$
where all terms are well-defined by (i), (ii) and the fact that 
$u^- \in \cH^\s_0(\Omega)$. Furthermore, $\cE^\s(u^+,u^-) \le 0$ by \eqref{eq:pointwise}. Hence the claim follows.

(v) By assumption, there exists a domain $\Omega' \subset \R^N$ with $\Omega \subset \subset \Omega'$ and 
$\cE^\s_{\Omega,\Omega'}(u,u)<\infty$ and a constant $c(\phi)>0$ with
$$
|\phi(x)-\phi(y)|\le c  (1 \wedge |x-y|)
\quad \text{for $x,y \in \Omega'$}.
$$
We may then estimate as follows (with $c= c(\phi,\s)$ changing from line to line):
	\begin{align*}
	&\cE^\s_{\Omega,\Omega'}(\phi u,\phi u) = \frac{1}{2}\int_{\Omega}\int_{\Omega'} [u(x)\phi(x)-u(y)\phi(y)]^2j(x-y)\ dydx\\
	&\leq  \int_{\Omega}u(x)^2 \int_{\Omega'}[\phi(x)-\phi(y)]^2j(x-y)\ dy dx+ \int_{\Omega}\int_{\Omega'} \phi(y)^2[(u(x)-u(y)]^2j(x-y)\ dydx \\
	&\leq  c \int_{\Omega}u(x)^2 \int_{\Omega'} (1 \wedge |x-y|^2) j(x-y)\ dy dx + \|\phi\|_{L^\infty(\Omega')}^2
\int_{\Omega}\int_{\Omega'} [(u(x)-u(y)]^2j(x-y)\ dy dx \\
	&\leq c \|u\|_{L^2(\Omega)}^2 + 2 \|\phi\|_{L^\infty(\Omega')}^2 \cE^\s_{\Omega,\Omega'}(u) < \infty.
	\end{align*}
Hence $u \in W^\s(\Omega)$.

(vi) The result is true if $0 \not \in \bar{\Omega}$, since $W^\s(\Omega) \subset L^2(\Omega)$ by the definition.
Now we consider the case of $0\in\Omega$.
Let $u \in W^\s(\Omega)$, and let $\phi \in C^\infty_c(\Omega)$ be a function with $\phi \equiv 1$ in a neighborhood of the origin. By (iii) and (v), 
then we have that $\phi u \in \cH^\s_0(\Omega)$, and therefore $\phi u \in L^2(\Omega,|x|^{-2\s})$ by the fractional Hardy inequality.
 Moreover, since $u \in L^2(\Omega)$ and $(1-\phi)u \equiv 0$ in a neighborhood of the origin, we also have $(1-\phi)u \in L^2(\Omega,|x|^{-2\s})$. 
 Consequently, $u = \phi u + (1-\phi) u \in L^2(\Omega,|x|^{-2\s})$.
\end{proof}

We now provide a weak version of Lemma~\ref{relationship-operators-0}.
\begin{lemma}
\label{relationship-operators}
Let $\mu > \mu_0$, $v \in C^2_c(\R^N)$ and $u \in \cH^\s_0(\R^N)$. Then we have
$$
\cE^\s_\mu(u,\Gamma_\mu v) = \int_{\R^N} u (-\Delta)^\s_{\Gamma_\mu} v\,dx.
$$
\end{lemma}

\begin{proof} Since $\mu > \mu_0$,  we have that $\psi: = \Gamma_\mu v \in H^\s_0(\R^N)$ by
 Lemma \ref{gamma-mu-W-alpha} and Lemma \ref{u-minus-lemma}(v). Using (\ref{product-prelim}), we then find that
\begin{align}
\cE^\s_\mu (u,\Gamma_\mu v) &=\cE^\s_\mu(u,\psi)= \int_{\R^N}  \frac{\mu}{|x|^{2\s}} u \psi\,dx+ \frac12 \int_{\R^N} \int_{\R^N}\frac{[u(x)-u(y)][\psi(x)-\psi(y)]}{|x-y|^{N+2\s}}dy dx\nonumber\\
&=  \int_{\R^N} u \cL^\s_\mu \psi\,dx = \int_{\R^N} u [(-\Delta)^\s_{\Gamma_\mu} v]\,dx,\label{relationship-operators-intermediate}
\end{align}
as claimed. Here, we used the facts that
$$
u \in \cH^\s_0(\R^N) \subset L^{\frac{2N}{N-2\s}}(\R^N) \quad \text{and}\quad
(-\Delta)^\s_{\Gamma_\mu} v= \cL^\s_\mu \psi \in L^{\frac{2N}{N+2\s}}(\R^N)
$$
by Lemma~\ref{lambda-estimate} to see that all integrals in (\ref{relationship-operators-intermediate}) are well defined in Lebesgue sense.
\end{proof}

\begin{definition}
\label{weak-solution}
Let $f \in L^{\frac{2N}{N+2\s}}(\Omega)$ and $u \in W^\s_*(\Omega)$.
\begin{itemize}
\item[(i)] We say that $\cL^\s_\mu u \ge f$ in $\Omega \setminus \{0\}$ if
$$
\cE^\s_\mu(u,v) \ge \int_{\Omega} f v \,dx \quad \text{for $v \in \cH^\s_0(\Omega)$, $v\ge 0$ with $v \equiv 0$ in a neighborhood of $0$.}
$$
\item[(ii)] We say that $u$ satisfies $\cL^\s_\mu u \le f$ in $\Omega \setminus \{0\}$ if $\cL^\s_\mu (-u) \ge -f$ in the sense of (i).
\item[(iii)] We say that $u$ satisfies $\cL^\s_\mu u = f$ in $\Omega \setminus \{0\}$ if $\cL^\s_\mu u \ge f$ and $\cL^\s_\mu u \le f$.
\end{itemize}
\end{definition}

\begin{remark}
\label{weak-solution-more-regular}
\begin{enumerate}
\item[(i)] If $f \in L^{\frac{2N}{N+2\s}}(\Omega)$ and $u \in W^\s(\Omega)$ satisfies $\cL^\s_\mu u \ge f$ in $\Omega \setminus \{0\}$ in the sense of Definition \ref{weak-solution}, then
$$
\cE^\s_\mu(u,v)  \ge \int_{\Omega} f v \,dx \qquad \text{for all
$v \in \cH^\s_0(\Omega)$, $v\ge 0$.}
$$
This follows by approximation using Lemma \ref{u-minus-lemma}(i) and (vi). Indeed, since $N  \ge 2$, for every $v \in \cH^\s_0(\Omega)$, $v \ge 0$, there exists a sequence of functions $v_n \in \cH^\s_0(\Omega)$, $n \in \N$ with $v_n \ge 0$, $v_n \equiv 0$ in a neighborhood of $0$ and $v_n \to v$ in $\cH^\s_0(\Omega)$, see Appendix B.\\
In this case, we also say that $u$ satisfies $\cL^\s_\mu u \ge f$ in $\Omega$.\\
The same remark applies to the properties $\cL^\s_\mu u \le f$ and $\cL^\s_\mu u = f$.
\item[(ii)] If $f \in L^\infty_{loc}(\overline{\Omega} \setminus \{0\})$ and $u \in L^\infty_{loc}(\R^N \setminus \{0\}) \cap L^1(\R^N,\frac{dx}{1+|x|^{N+2\s}})$ satisfies
 \begin{equation*}
\left \{
\begin{aligned}
\cL^\s_\mu u  = f  \quad  {\rm in}\ \  \Omega \setminus \{0\}, \\
 u  =0    \quad  {\rm in}\ \   \Omega^c\qquad\ 
\end{aligned}
\right.
\end{equation*}
in distributional sense, then also $u \in W^\s_*(\Omega)$, and $u$ satisfies $\cL^\s_\mu u  = f$ in $\Omega$ in the sense of 
Definition \ref{weak-solution}. We postpone the proof of this fact to Appendix C.
\end{enumerate}
\end{remark}

\begin{lemma}
\label{Phi-mu-Gamma-mu-sol-property}
For $\mu \ge \mu_0$, we have, in the sense of Definition~\ref{weak-solution} and Remark~\ref{weak-solution-more-regular},
\begin{itemize}
\item[(i)] $\cL^\s_\mu \Phi_\mu = 0$ in $\Omega \setminus \{0\}$;
\item[(ii)] $\cL^\s_{\mu_0} \Gamma_{\mu_0}=0$ in $\Omega \setminus \{0\}$;
\item[(iii)] $\cL^\s_\mu \Gamma_{\mu}=0$ in $\Omega$ for $\mu >\mu_0$.
\end{itemize}
\end{lemma}

\begin{proof}
(i) and (ii) follow directly from {\it Theorem A} and Remark~\ref{gamma-mu-W-alpha-remark}, using that every function
$v \in \cH^\s_0(\Omega)$ with $v \equiv 0$ in a neighborhood of $0$ can be approximated by functions in $C^\infty_c(\Omega \setminus \{0\})$.\\
(iii) follows by Remark~\ref{weak-solution-more-regular} and Lemma~\ref{gamma-mu-W-alpha}.
\end{proof}

\begin{lemma} (Comparison principle, Version 1)\\
\label{sec:dirichl-probl-bound-comp-1}
Let $u \in W^\s(\Omega)$ satisfy $\cL^\s_\mu u \ge 0$ in $\Omega$ and $u \ge 0$ in $\Omega^c$. Then 
$u \ge 0$ in $\R^N$.
\end{lemma}

\begin{proof}
By Lemma~\ref{u-minus-lemma}(iii) and the assumption, it follows that
$u^- \in \cH_0^\s(\Omega)$ satisfies
\begin{align*}
-\cE^\s(u^-,u^-) &\ge \cE^\s(u,u^-) \ge -\mu \int_{\Omega}\frac{u(x)u^-(x)}{|x|^{2\s}}\,dx = \mu \int_{\Omega}\frac{|u^-(x)|^2}{|x|^{2\s}}\,dx.
\end{align*}
Since $\mu \ge \mu_0$, it thus follows from (\ref{eq:fractional-hardy-remainder}) that $u^- \equiv 0$. Hence the claim follows.
\end{proof}

\begin{lemma} (Comparison principle, Version 2)\\
\label{sec:dirichl-probl-bound-comp}
Let $u \in W^\s_*(\Omega)$ satisfy $\cL^\s_\mu u \ge 0$ in $\Omega \setminus \{0\}$ and
\begin{equation}
  \label{eq:positivity-condition-boundary}
u \ge 0 \quad \text{in $\Omega^c$}, \qquad \liminf_{x \to 0} \frac{u(x)}{\Phi_\mu(x)} \ge 0.
\end{equation}
Then  $u \ge 0$ in $\R^N$.
\end{lemma}

\begin{proof}
Let $\delta>0$. By (\ref{eq:positivity-condition-boundary}), we may choose
$\eps >0$ such that $u \ge -\delta \Phi_\mu$ on $\R^N \setminus \Omega_\eps$. By assumption, $w:= u + \delta \Phi_\mu \in W^\s(\Omega_\eps)$ satisfies
$\cL^\s_\mu w \ge 0$ in $\Omega_\eps$, and we also have that $w \ge 0$ on $\R^N \setminus \Omega_\eps$. By Lemma~\ref{sec:dirichl-probl-bound-comp-1}, we thus conclude that $w \ge 0$, so that $u \ge - \delta \Phi_\mu$. Since $\delta>0$ was chosen arbitrarily, we conclude that $u \ge 0$.
\end{proof}

The following is the main result of this section. 

\begin{theorem}
\label{sec:dirichl-probl-bound-comp-corol}
Let $\mu\ge\mu_0$ and $f \in L^\infty(\Omega, |x|^{\rho}dx) $ for some $\rho < 2\s - \tau_+(\s,\mu)$.
Then there exists a unique solution $u \in W^\s_{*}(\Omega)$ for the problem
\begin{equation}\label{eq 2.1f}
\cL^\s_\mu u = f \quad  {\rm in}\ \, \Omega \setminus \{0\}, \qquad \ 
u =0 \quad   {\rm in}\ \, \Omega^c,\qquad \  
\lim_{x \to 0}\frac{u(x)} {\Phi_\mu(x)}= 0.
\end{equation}
Moreover: 
\begin{itemize}
\item[(i)] $u \ge 0$  if   $f \ge 0$.
\item[(ii)] There exists a constant $c=c(\s,\mu,\rho,\Omega)>0$ such that
\begin{equation}
  \label{eq:Gamma-mu-comp}
\frac{|u(x)|}{\Gamma_\mu(x)} \le c \|f\|_{L^\infty(\Omega,\,|x|^\varrho dx)} \qquad \text{for all $x \in \Omega \setminus \{0\}.$}
\end{equation}
Moreover, if $\mu_1>\mu_0$ is fixed with $\varrho < 2\s - \tau_+(\s,\mu_1)$ and $\mu \in [\mu_0,\mu_1]$, then $c$ can be chosen such that it only depends on $\s,\mu_1,\varrho$ and $\Omega$.
\item[(iii)] If $g \in L^\infty(\Omega, |x|^{\rho}dx)$, then 
\begin{equation}\label{g-estimate}
\int_{\Omega} u g \,dx \le c\, \|g\|_{L^\infty(\Omega, |x|^{\rho}dx)} \,\norm{f}_{L^1(\Omega,d\gamma_\mu)} \qquad \text{with $c>0$ as in (ii).}
\end{equation}
Here $d\gamma_\mu = \Gamma_\mu dx$ as before.
\item[(iv)] If $\mu >\mu_0$, we have $u \in \cH^\s_0(\Omega)$ and
  \begin{equation}
    \label{eq:weak-identity}
\cE^\s_\mu(u,v)= \int_{\Omega}f v\,dx\qquad \text{for all $v \in \cH^\s_0(\Omega)$.}
  \end{equation}
\item[(v)] $u$ verifies the distributional identity
\begin{equation}
  \label{eq:distributional-0}
\int_{\Omega}u_\mu   (-\Delta)^\s_{\Gamma_\mu} \xi \,dx = \int_{\Omega}f \xi \,d\gamma_\mu \qquad \text{for all $\xi \in \cC^2_c(\Omega)$,}
\end{equation}
and satisfies the estimate
\begin{equation}\label{l1}
\norm{u}_{L^1(\Omega,\Lambda _\mu dx)}\le d\norm{f}_{L^1(\Omega,d\gamma_\mu)} \qquad \text{with some constant $d= d(\s,\mu,\rho,\Omega)>0$.}
\end{equation}
\end{itemize}
Here we recall that we have set $d\gamma_\mu = \Gamma_\mu dx$.
\end{theorem}

\begin{proof}
To see the uniqueness of solutions of (\ref{eq 2.1f}), let $u_1,u_2 \in W^\s_{*}(\Omega)$ be solutions. Applying Lemma~\ref{sec:dirichl-probl-bound-comp} to $u_1-u_2$ and $u_2-u_1$, we find that $u_1 \equiv u_2$.

To prove the existence and asserted properties, we may from now on assume that $f \ge 0$, since in the general case we can decompose $f= f^+- f^-$. In this case we deduce that $u \ge 0$ once we have proved the existence of the (unique) solution $u$ of (\ref{eq 2.1f}), so (i) holds. Moreover, since $\tau_+(\s,\mu)<2\s$ for all $\mu \in [\mu_0,\infty)$, we may assume that $\varrho>0$ in the following. Finally, by homogeneity, we may assume that $\|f\|_{L^\infty(\Omega,|x|^\varrho dx)}=1$.
To prove the existence, we first consider the case where $\mu>\mu_0$. In this case, the Riesz representation theorem -- together with the fractional Hardy inequality -- immediately yields the existence of $u  \in \cH^\s_0(\Omega)$ such that
\begin{equation}
  \label{eq:weak-sol-prop}
\cE^\s_\mu(u,v) = \int_{\Omega} f v \,dx\qquad \text{for all $v \in \cH^\s_0(\Omega)$,}
\end{equation}
as claimed in (iii). In particular, we have $\cL^\s_{\mu}u = f$ in $\Omega$ and $u \equiv 0$ in $\Omega^c$. We also note that the condition $\lim_{|x| \to 0} \frac{u(x)}{\Phi_\mu(x)} = 0$ follows once we have proved (\ref{eq:Gamma-mu-comp}). To prove (\ref{eq:Gamma-mu-comp}), we fix $\mu_1>\mu_0$ and $\rho < 2\s - \tau_+(\s,\mu_1)$.
For $\mu \in (\mu_0,\mu_1]$, then we have that
\begin{equation}
  \label{eq:mu-c-estimate}
\mu +c_\s(2\s-\varrho)\le \mu_1 +c_\s(2\s-\varrho) < \mu_1 + c_\s(\tau_+(\s,\mu_1))=0,
\end{equation}
since $2\s- \varrho > \tau_+(\s,\mu_1)$ and the function $c_\s$ is strictly decreasing on $[\frac{2\s-N}{2},2\s)$ by Lemma~\ref{c-function}. Moreover, we recall from Remark \ref{eq:Delta-alpha-homogeneous} that the function $\phi_{2\s-\varrho} \in W^\s(\Omega)$ satisfies
$$
\cL^\s_\mu  \phi_{2\s-\varrho}  (x) = [\mu + c_\s (2\s-\varrho)] |x|^{-\varrho}
\le [\mu_1 + c_\s (2\s-\varrho)] |x|^{-\varrho} \quad \text{in $\Omega$.}
$$
By (\ref{eq:mu-c-estimate}), we may thus fix $\kappa_1>0$ sufficiently large depending only on $\mu_1$ so that
$$
\cL^\s_\mu \bigl(-\kappa_1 \phi_{2\s-\varrho} \bigr)(x) \ge |x|^{-\varrho} +1 \qquad \text{for $x \in \Omega$, $\mu \in (\mu_0,\mu_1]$.}
$$
We now let $\eta \in C^\infty_c(\R^N)$ be a radial function with $0 \le \eta \le 1$, $\eta \equiv 1$ in $B_1$ and $\eta \equiv 0$ in $\R^N \setminus B_2$. We also let $\eta_\delta \in C^\infty_c(\R^N)$ be defined by
$\eta_\delta(x)= \eta(\delta x)$, and we define
$$
v_\delta \in W^\s(\Omega), \qquad v_\delta(x)= - \kappa_1  \eta(\delta x) |x|^{2\s-\varrho}.
$$
If $\sigma >0$ is chosen small enough such that $\Omega \subset B_{\frac{1}{2\sigma}}$, we have, for $x \in \Omega$,
\begin{align*}
\cL^\s_\mu v_\sigma(x) &=
 \eta(\sigma x) \cL^\s_\mu \bigl(-\kappa_1 |\cdot|^{2\s-\varrho}\bigr)(x)
- \kappa_1 \int_{\R^N} (\eta(\sigma x)-\eta(\sigma y))|y|^{2\s-\varrho} j_\s(x-y)dy\biggr)\\
&\ge |x|^{-\rho} +1 - 2\kappa_1 \int_{\R^N} (1-\eta(\sigma y))|y|^{2\s-\varrho}j_\s(x-y)dy\\
&\ge  f(x) + 1 -  \kappa_1 \sigma^{\varrho} \int_{\R^N} (1-\eta(y))|y|^{2\s-\varrho} j_\s(\sigma x - y)dy\\
&=  f(x) + 1 -   \kappa_1 \sigma^{\varrho} \int_{\R^N \setminus B_1} (1-\eta(y))|y|^{\varrho}j_\s(\sigma x - y)dy\\
&\ge  f(x) + 1 - \kappa_2 \sigma^{\varrho},
\end{align*}
where
$\kappa_2:= \kappa_1 \int_{\R^N \setminus B_1} (1-\eta(y))
(|y|-\frac{1}{2})^{-N-2\s}dy$. 
Consequently, we may fix $\sigma>0$ sufficiently small (independently of $\mu$ and $f$) to guarantee that
$$
\cL^\s_\mu v_\sigma \ge f \qquad \text{in $\Omega\quad $ for $\mu \in (\mu_0,\mu_1]$}.
$$
We now define the function
$$
w \in W^\s(\Omega), \quad w(x)= c \Gamma_\mu(x)+v_\sigma(x) \qquad \text{with $c:=\kappa_1 \bigl(\frac{2}{\sigma}\bigr)^{2\s-\varrho-\tau_+(\s,\mu_0)}.$}
$$
Since $2\s-\varrho > \tau_+(\s,\mu)$, we have that
\begin{align*}
w_c(x) &\ge |x|^{\tau_+(\s,\mu)}\Bigl(c-\kappa_1 1_{B_{\frac{2}{\sigma}}}(x)|x|^{2\s-\varrho-\tau_+(\s,\mu)}\Bigr)
\ge |x|^{\tau_+(\s,\mu)}\Bigl(c- \kappa_1 \bigl(\frac{2}{\sigma}\bigr)^{2\s-\varrho-\tau_+(\s,\mu)}\Bigr)\\
&= |x|^{\tau_+(\s,\mu)}\kappa_1 \bigl(\frac{2}{\sigma}\bigr)^{2\s-\varrho -\tau_+(\s,\mu_0)} \Bigl(1 -\bigl(\frac{2}{\sigma}\bigr)^{\tau_+(\s,\mu_0)-\tau_+(\s,\mu)}\Bigr)\ge 0
\end{align*}
for $x\in\R^N \setminus \{0\}$, since $\tau_+(\s,\mu_0)-\tau_+(\s,\mu) \le 0$. Since also
$$
\cL^\s_\mu w \ge f \quad \text{in $\Omega$},
$$
we apply Lemma~\ref{sec:dirichl-probl-bound-comp-1} to $w-u$ and find that $u \le w \le c \Gamma_\mu$ a.e. in $\Omega$. Hence we have shown (\ref{eq:Gamma-mu-comp}) for $\mu \in (\mu_0,\mu_1]$ with a constant $c=c(\s,\mu_1,\rho,\Omega)$.

We now prove the distributional identity (\ref{eq:distributional-0}) for $\mu>\mu_0$. For $\xi \in C^2_c(\Omega)$,  $\phi= \xi \Gamma_\mu$ and $u \in H^1_0(\Omega)$, Lemma~\ref{relationship-operators} and (\ref{eq:weak-sol-prop}) imply that
$$
\int_{\R^N} u (-\Delta)^\s_{\Gamma_\mu}\xi \,dx = \cE^\s_\mu(u, \phi)= \int_{\R^N}f \phi\,dx = \int_{\R^N}f \xi \,d\gamma_\mu,
$$
as claimed.
Next, we let $g \in L^\infty(\Omega, |x|^{\rho}dx)$, and we prove (\ref{g-estimate}) for $\mu>\mu_0$. By what we have proved so far, there exists a function $v \in \cH^\s_0(\Omega)$ satisfying
$$
\cE^\s_\mu(v,\xi)= \int_{\Omega}g \xi\,dx \qquad \text{for all $\xi \in \cH^\s_0(\Omega)$}
$$
and $\frac{|v|}{\Gamma_\mu} \le c \|g\|_{L^\infty(\Omega,|x|^{\varrho} dx)}$ in \ $\Omega$. Then, choosing $\xi=u$, we obtain that
$$
\int_{\Omega}g u \,dx = \cE^\s_\mu(v,u)= \int_{\Omega}fv \,dx \le
c \|g\|_{L^\infty(\Omega,|x|^{\varrho} dx)} \int_{\Omega}|f| \Gamma_\mu \,dx
= c \|g\|_{L^\infty(\Omega,|x|^{\varrho} dx)} \int_{\Omega} |f|d\gamma_\mu(x),
$$
as claimed in (\ref{g-estimate}).

To treat the case $\mu=\mu_0$, we consider a decreasing sequence of numbers $\mu_n$, $n \in \N$ with $\mu_n >\mu_0$ for all $n$ and $\mu_n \to \mu_0$ as $n \to \infty$. Since $\varrho< 2\s-\tau_+(\s,\mu_0)$ by assumption, we may assume, passing to a subsequence, that
\begin{equation}
  \label{eq:uniform-condition-rho}
\varrho < 2\s-\tau_+(\s,\mu_1).
\end{equation}
Using what we have proved so far for $\mu=\mu_n$, we have
\begin{equation}
  \label{eq:u-n-estimate+1}
|u_n| \le c \|f\|_{L^\infty(\Omega,|x|^\varrho dx)} \Gamma_{\mu_n} \qquad \text{in $\Omega\quad$ for every $n$ with $c = c(\s,\varrho,\Omega,\mu_1)$}.
\end{equation}
Adjusting the value of $c = c(\s,\varrho,\Omega,\mu_1)$, we can thus infer that
\begin{equation}
  \label{eq:u-n-estimate-mu-0}
|u_n| \le c \|f\|_{L^\infty(\Omega,|x|^\varrho dx)} \Gamma_{\mu_0} \qquad \text{in $\Omega\quad$ for every $n$ with $c = c(\s,\varrho,\Omega,\mu_1)$}
\end{equation}
Since $\Gamma_{\mu_0}(x)= |x|^{\frac{2\s-N}{2}}$, we have $\Gamma_{\mu_0} \in L^{s}(\Omega)$ for every $1<s < \frac{2N}{N-2\s}$, 
and hence the sequence $\{u_n\}_n$ is bounded in $L^{s}(\Omega)$ for every $1<s < \frac{2N}{N-2\s}$. Thus, there exists
\begin{equation}
  \label{eq:weak-convergence-u-L-s}
u \in \bigcap_{1<s < \frac{2N}{N-2\s}}L^s(\Omega) \quad \text{with}\quad u_n \weak u\qquad \text{in $L^{s}(\Omega)$ for $s < \frac{2N}{N-2\s}$.}
\end{equation}
We claim that $u \in W^\s_*(\Omega)$, and that $u$ solves (\ref{eq 2.1f}). To show that $u \in W^\s_*(\Omega)$, it suffices to show that
$$
\cE^\s(\psi u,\psi u)<\infty \qquad \text{for every $\psi \in C^\infty_c (\R^N \setminus \{0\})$.}
$$
To this end, we note that
$$
\cE^\s_{\mu_n} (u_n,\psi^2 u_n) = \int_{\Omega}f \psi^2 u_n\,dx\le \tilde c \int_{\Omega} f \Gamma_{\mu_0}\,dx = \tilde c \int_{\Omega} f |x|^{\frac{2\s-N}{2}}\,dx \le \tilde c \int_{\Omega} |x|^{
\frac{2\s-N}{2}-\varrho}\,dx \le \tilde c.
$$
Here and in the following, the symbol $\tilde c$ stands for a constant, depending on $\psi$ but not on $n$,  which may change its value in every step. Moreover, we have that
\begin{align*}
\cE^\s(\psi u_n, \psi u_n) &- \cE^\s_{\mu_n} (u_n,\psi^2 u_n)\\
&= \int_{\R^N \times \R^N}u_n(x)u_n(y)(\psi(x)-\psi(y)^2j(x-y)\,dxdy  - \mu_n \int_{\Omega}\frac{\psi^2u_n^2}{|x|^{2\s}}\,dx\\
&\le 2  \int_{\Omega} u_n(x) \int_{K} u_n(y) (\psi(x)-\psi(y)^2j(x-y)\,dydx - \tilde c \mu_n \|u_n\|_{L^2(\Omega)}^2 \\
&\le \tilde c \int_{\Omega}u_n(x)\,dx + \tilde c  \le \tilde c.
\end{align*}
Combining these inequalities, we find that
$$
\cE^\s(\psi u_n, \psi u_n) \le \tilde c \qquad \text{for all $n$.}
$$
Consequently, the sequence $\{\psi u_n\}_n$ is bounded in $\cH^\s_0(\Omega)$, which implies that, after passing to a subsequence,
\begin{equation}
  \label{eq:label-weak-local-convergence}
\psi u_n  \weak \psi u  \qquad \text{in $\cH^\s_0(\Omega)$.}
\end{equation}
Since also $\psi u_n \weak \psi u$ in $L^2(\Omega)$, it follows that
$\psi u_n \weak \psi u$ in $\cH^\s_0(\Omega)$, and therefore
$$
\cE^\s(\psi u,\psi u)<\infty.
$$
This implies that $u \in W^\s_*(\Omega)$. Next, we show that
\begin{equation}
  \label{eq:very-weak-sol-prop}
\cE^\s_{\mu_0} (u,v)= \int_{\Omega} f v \,dx\quad \text{for every $v \in \cH^\s_0(\Omega)$ with $v \equiv 0$ in a neighborhood of $0$.}
\end{equation}
Fixed $v \in \cH^\s_0(\Omega)$, and let $\eps>0$ be such that $v \equiv 0$ in $B_{2\eps} $. Moreover, let $\psi \in C^\infty_c (\R^N \setminus \{0\})$ satisfy $0 \le \psi \le 1$ and $\psi \equiv 1$ in $\Omega \setminus B_\eps $. Since $\psi u_n \weak \psi u$ in $\cH^\s_0(\Omega)$, we have that
\begin{equation}
  \label{eq:very-weak-sol-comb-1}
\cE^\s(\psi u_n,v) \to \cE^\s(\psi u,v) \quad \text{as\ \, $n \to \infty$.}
\end{equation}
Moreover, we have
$$
\cE^\s(\psi u,v)-\cE^\s(u,v)= \cE^\s(\tilde \psi u,v)= -\int_{\R^N \times \R^N}\tilde\psi(x)u(x)v(y)j(x-y)\,dydx
$$
with $\tilde\psi= 1-\psi$ and
$$
\cE^\s(\psi u_n,v)-\cE^\s(u_n,v)= \cE^\s(\tilde\psi u_n,v)= -\int_{\R^N \times \R^N}\tilde\psi(x)u_n(x)v(y)j(x-y)\,dydx,
$$
since $v \equiv 0$ in $B_{2\eps} $ and $\psi \equiv 0$ in $\R^N \setminus B_\eps $. Consequently,
$$
\cE^\s(\psi u_n,v)-\cE^\s(u_n,v)-\bigl[\cE^\s(\psi u,v)-\cE^\s(u,v)\bigr]= \int_{\R^N} (u_n(x)-u(x))h(x) dx
$$
with
$$
h \in L^\infty(\R^N), \qquad h(x):= \tilde\psi(x)\int_{\R^N \setminus B_{2\eps}(0)}v(y)j(x-y)\,dy.
$$
Since $u_n \weak u$ in $L^1(\Omega)$, we thus conclude that
\begin{equation}
  \label{eq:very-weak-sol-comb-2}
\cE^\s(\psi u_n,v)-\cE^\s(u_n,v)-\bigl[\cE^\s(\psi u,v)-\cE^\s(u,v)\bigr] \to 0 \qquad \text{as $n \to \infty$.}
\end{equation}
Combining (\ref{eq:very-weak-sol-comb-1}) and (\ref{eq:very-weak-sol-comb-2}), we find that
$$
\cE^\s(u,v)= \lim_{n \to \infty}  \cE^\s(u_n,v) = - \lim_{n \to \infty}
\mu_n \int_{\Omega}\frac{u_n v}{|x|^{2\s}}\,dx + \int_{\Omega} f v\,dx
= \mu_0 \int_{\Omega}\frac{u v}{|x|^{2\s}}\,dx + \int_{\Omega} fv\,dx,
$$
as claimed in (\ref{eq:very-weak-sol-prop}). Hence $u$ solves (\ref{eq 2.1f}) with $\mu=\mu_0$.

Next we note that, along a subsequence, we may pass from weak to strong convergence in (\ref{eq:weak-convergence-u-L-s}), i.e., we have
\begin{equation}
  \label{eq:label-strong-convergence}
u_n \to u\quad \text{in $L^{s}(\Omega)$\ for $s < \frac{2N}{N-2\s}$.}
\end{equation}
Indeed, for every $\psi \in C_c^\infty(\R^N \setminus \{0\})$, after passing to a subsequence, we have that
$$
\psi u_n \to \psi u\quad \text{in $L^{s}(\Omega)$\ for $s < \frac{2N}{N-2\s}$}
$$
by (\ref{eq:label-weak-local-convergence}) and the compact Sobolev embeddings $\cH^\s_0(\Omega) \hookrightarrow L^{s}(\Omega)$ for
$s < \frac{2N}{N-2\s}$. Then the set $A_\eps:= \{u_n 1_{\Omega\setminus B_{\eps}}\::\:\: n \in \N\}$
are precompact in $L^{s}(\Omega)$ for $s < \frac{2N}{N-2\s}$ and $\eps>0$. Moreover, recalling (\ref{eq:u-n-estimate-mu-0}) and the fact
that $\Gamma_0 \in L^{s}(\Omega)$ for $s < \frac{2N}{N-2\s}$, we then conclude that the set $\{u_n\::\:n \in \N\}$ is 
precompact in $L^{s}(\Omega)$ for $s < \frac{2N}{N-2\s}$. Thus we may pass to a subsequence such that (\ref{eq:label-strong-convergence}) holds.

From \eqref{eq:u-n-estimate-mu-0} and (\ref{eq:label-strong-convergence}), it  follows that
\begin{equation}
  \label{eq:u-n-estimate}
|u| \le c \|f\|_{L^\infty(\Omega),|x|^\rho dx)} \Gamma_{\mu_0} \quad \text{on $\Omega$ with $c = c(\s,\rho,\Omega,\mu_1)$.}
\end{equation}
Hence (\ref{eq:Gamma-mu-comp}) is  true for $\mu=\mu_0$.\\

Next we prove the distributional identity (\ref{eq:distributional-0}) in the case $\mu= \mu_0$.
Since $\tau_+(\s,\mu_0)= \frac{2\s-N}{2}<2\s-1$, we may assume  that $\tau_+(\s,\mu_n) < 2\s-1$ for all $n \in \N$. We then note that
\begin{equation}
  \label{eq:bounded-L-t}
|\Lambda_{\mu_n} (x)| = |x|^{1-2\s+\tau_+(\s,\mu_n)} \le C |x|^{1-2\s+ \tau_+(\s,\mu_0)}= C |x|^{\frac{2-2\s-N}{2}} \qquad \text{for $x \in \Omega$, $n \in \N$}
\end{equation}
with a constant $C>0$. Fixed $t \in \bigl(\frac{2N}{N+2\s-2}, \frac{2N}{N+2\s}\bigr)$ and
$\xi \in C^2_c(\Omega)$, we claim that
\begin{equation}
  \label{eq:weak-modified-fractional}
(-\Delta)^\s_{\Gamma_{\mu_n}} \xi \weak  (-\Delta)^\s_{\Gamma_{\mu_0}} \xi \quad \text{in $L^t(\Omega)$ as $n \to \infty$.}
\end{equation}
Indeed, from Lemma~\ref{lambda-estimate} and (\ref{eq:bounded-L-t}), we deduce that
the sequence $\{(-\Delta)^\s_{\mu_n} \xi\}_n$ is bounded in $L^t(\Omega)$. Moreover, for $\psi \in C^2_c(\Omega \setminus \{0\})$, we have  that
\begin{align*}
\int_{\Omega}\psi(x) [(-\Delta)^\s_{\Gamma_{\mu_n}} \xi](x)\,dx &= \int_{\Omega}\psi(x) [(-\Delta)^\s (\xi \Gamma_{\mu_n})](x)\,dx
+ \int_{\Omega}\psi(x)\xi(x) [(-\Delta)^\s \Gamma_{\mu_n}](x)\,dx\\
\to &\int_{\Omega}\psi(x) [(-\Delta)^\s (\xi \Gamma_{\mu_0})](x)\,dx
+ \int_{\Omega}\psi(x)\xi(x) [(-\Delta)^\s \Gamma_{\mu_0}](x)\,dx
\end{align*}
as $n \to \infty$, since $\Gamma_{\mu_n} \to \Gamma_{\mu_0}$ and $\xi \Gamma_{\mu_n} \to \xi \Gamma_{\mu_n}$ in $C^2_{loc}(\Omega \setminus \{0\})$ and in $L^1(\R^N,\frac{dx}{1+|x|^{N+2\s}})$. Since $\psi \in C^2_c(\Omega \setminus \{0\})$ is dense in $L^t(\Omega)$, we thus conclude that (\ref{eq:weak-modified-fractional}) holds.\smallskip

For $\mu=\mu_n$, we now write the identity (\ref{eq:distributional-0}) as
 \begin{equation}\label{eq:distributional-2}
\int_{\Omega}u_n   (-\Delta)^\s_{\Gamma_{\mu_0}} \xi \,dx+\int_{\Omega}u_n  [(-\Delta)^\s_{\Gamma_{\mu_n}} \xi- (-\Delta)^\s_{\Gamma_{\mu_0}}\xi] \,dx = \int_{\Omega}f \xi\,d\mu_n.
\end{equation}
Since $t'= \frac{t}{t-1} < \frac{2N}{N-2\s}$ and therefore,
\begin{equation}
  \label{eq:t-prime-convergence}
u_n \to u\quad \text{in $L^{t'}(\Omega)$}\ \ {\rm as} \ \, n \to \infty,
\end{equation}
 we now deduce from (\ref{eq:weak-modified-fractional}) that
$$\int_{\Omega}u_n   (-\Delta)^\s_{\Gamma_{\mu_0}} \xi \,dx\to \int_{\Omega}u   (-\Delta)^\s_{\Gamma_{\mu_0}} \xi \ \ {\rm as} \ \, n \to \infty$$
and
$$\int_{\Omega}u_n  \left((-\Delta)^\s_{\Gamma_{\mu_n}}- (-\Delta)^\s_{\Gamma_{\mu_0}}\right) \xi \,dx \to 0\ \ {\rm as} \ \, n \to \infty.
$$
 Moreover,
$$
\int_{\Omega}f \xi\, d\mu_n = \int_{\Omega}f \xi \Gamma_{\mu_n}\,dx \to \int_{\Omega}f \xi \Gamma_{\mu_0}\,dx = \int_{\Omega}f \xi\, d\gamma_{\mu_0} \quad \text{as $n \to \infty$.}
$$
We thus conclude from (\ref{eq:distributional-2}) that
$$
\int_{\Omega}u   (-\Delta)^\s_{\Gamma_{\mu_0}} \xi \,dx = \int_{\Omega}f \xi\, d\gamma_{\mu_0},
$$
as desired.\smallskip

Next, we let $g \in L^\infty(\Omega, |x|^{\rho}dx)$, and we prove (\ref{g-estimate}) for $\mu=\mu_0$. For this we use (\ref{g-estimate}) for $\mu=\mu_n$, obtaining that
\begin{equation}
  \label{eq:lambda-more-detailed-est-1}
\int_{\Omega}g u_n \,dx  \le c \|g\|_{L^\infty(\Omega,|x|^{\rho}dx)} 
\int_{\Omega} |f|\,d\gamma_{\mu_n}.
\end{equation}
with $c= c(\s,\varrho,\Omega,\mu_1)$, where, by Lebesgue's Theorem,  
$$
\int_{\Omega} |f(x)|\,d\gamma_{\mu_n}(x)= \int_{\Omega}|f(x)||x|^{\tau_+(\s,\mu_n)}\,dx \to \int_{\Omega}|f(x)||x|^{\tau_+(\s,\mu)}\,dx = \int_{\Omega} |f(x)|\,d\gamma_{\mu_0}(x)\quad \text{as $n \to \infty$.}
$$
Moreover, we note that $L^\infty(\Omega, |x|^{\rho}dx) \subset L^t(\Omega)$ for $1 \le t< \frac{N}{\rho}$. Since $\rho <2\s-\tau_+(\s,\mu_0)= \frac{N+2\s}{2}$, there exists $t \in \bigl(\frac{2N}{N+2\s-2}, \frac{2N}{N+2\s}\bigr)$ with $g \in L^\infty(\Omega, |x|^{\rho}dx) \subset L^t(\Omega)$. Thus (\ref{eq:t-prime-convergence}) implies that 
$$
\int_{\Omega}gu_n  \,dx \to \int_{\Omega}g u \,dx \quad \text{as $n \to \infty$.}
$$
Consequently, we may pass to the limit in (\ref{eq:lambda-more-detailed-est-1}) and obtain (\ref{g-estimate}) for $\mu=\mu_0$.

Finally, we prove (\ref{l1}) for $\mu\ge \mu_0$. By Lemma~\ref{lambda-estimate}, we have
$$
\Lambda _\mu \in L^\infty(\Omega,|x|^{\varrho(\s,\mu)}dx)\qquad \text{with $\varrho(\s,\mu):= \max \bigl\{ 2\s-\tau_+-1,
\frac{2\s-\tau_+(\s,\mu)}{2} \bigr\}< 2\s-\tau_+(\s,\mu).$}
$$
Hence we may choose $g= \Lambda_\mu$ in (\ref{g-estimate}), which yields (\ref{l1}) with $d:=c \|\Lambda _\mu\|_{L^\infty(\Omega,|x|^{\varrho(\s,\mu)}dx)}$.
\end{proof}

We also need the following uniform estimate away from the origin.  

\begin{proposition}
\label{sec:dirichl-probl-bound-comp-corol-corol}
Let $\mu\ge\mu_0$, let $f \in L^\infty(\Omega, |x|^{\rho}dx) $ for some $\rho < 2\s - \tau_+(\s,\mu)$, and let $u \in W^\s_{*}(\Omega)$ be the unique solution of \eqref{eq 2.1f}. Then for any numbers $0<r_1<r_2<\infty$, there exists a constant $c=c(\s,\mu,\rho,\Omega,r_1,r_2)>0$ such that
\begin{equation*}
|u(x)| \le c \Bigl(\|f\|_{L^\infty(\Omega \setminus B_{r_1})}+ \|f\|_{L^1(\Omega,d\gamma_\mu)}\Bigr) \qquad  \text{for all $x \in \Omega \setminus B_{r_2}$.}
\end{equation*}
\end{proposition}

\begin{proof}
We fix numbers $s_i,\sigma_i, t_i, \theta_i$, $i=1,2$ with $r_1<s_1<\sigma_1<t_1<\theta_1<\theta_2<t_2<\sigma_2<s_2<r_2$. 
We also choose a function $\eta \in C^\infty(\R^N)$ with $\eta \equiv 1$ on $\R^N \setminus B_{\theta_2}$ and $\eta \equiv 0$ on $B_{\theta_1}$. 
We then consider the function 
$$
g: \Omega \to \R, \qquad g(x)= p.v.\int_{\R^N} \frac{\eta(x)-\eta(y)}{|x-y|^{N+2\s}}u(y)\,dy. 
$$
We claim that $g \in L^\infty(\Omega)$ with 
\begin{equation}
  \label{eq:g-bound}
\|g\|_{L^\infty(\Omega)} \le C \Bigl(\|f\|_{L^\infty(\Omega \setminus B_{r_1}}+ \|f\|_{L^1(\Omega,d\gamma_\mu)}\Bigr). 
\end{equation}
Here and in the following, the letter $C$ denotes positive constants depending only on $\s,\mu,\rho,\Omega,r_1$ and $r_2$. We suppose for the moment that this is true. Then we see that the function $\tilde u:= u \eta \in W^{\s}(\Omega)$ solves \eqref{eq 2.1f}
with $f$ replaced by the function 
$$
\tilde f:= f \eta + C_{N,\s} g \in L^\infty(\Omega).
$$
Consequently, Theorem~\ref{sec:dirichl-probl-bound-comp-corol} implies that 
$$
\frac{|\tilde u(x)|}{\Gamma_\mu(x)} \le C \|\tilde f\|_{L^\infty(\Omega, |x|^{\rho}dx)} \le C \|\tilde f\|_{L^\infty(\Omega,dx)}
\quad\ \text{for $x \in \R^N \setminus \{0\}$}
$$
and therefore 
$$
|u(x)|= |\tilde u(x)| \le C\, \Gamma_\mu(x)\|\tilde f\|_{L^\infty(\Omega,dx)} \le C \bigl(\|f\|_{L^\infty(\Omega \setminus B_{r_1}(0)}+ \|f\|_{L^1(\Omega,d\gamma_\mu)}\Bigr)\quad \text{for $x \in \R^N \setminus B_{r_2}$,}
$$
as claimed. Hence it remains to show (\ref{eq:g-bound}). For $x \in \R^N \setminus B_{t_2}$, we have  
\begin{align*}
g(x)&= \int_{\R^N} \frac{\eta(x)-\eta(y)}{|x-y|^{N+2\s}}u(y)\,dy = \int_{B_{\theta_2}}\frac{1-\eta(y)}{|x-y|^{N+2\s}}u(y)\,dy\\ 
&\le \int_{B_{\theta_2}}\frac{|u(y)|}{|x-y|^{N+2\s}}\,dy \le (t_2-\theta_2)^{-N-2\s}\|u\|_{L^1(B_{t_2})} \le C \|u\|_{L^1(\Omega,\Lambda_\mu dx)}
 \le C \|f\|_{L^1(\Omega,d\gamma_\mu)},   
\end{align*}
where in the last step we used Theorem~\ref{sec:dirichl-probl-bound-comp-corol} again. Similarly, for $x \in B_{t_1}$ we see that 
\begin{align*}
g(x) &\le \int_{\R^N \setminus B_{\theta_1}}\frac{|u(y)|}{|x-y|^{N+2\s}}\,dy \le C \int_{\Omega \setminus B_{\theta_1}}\frac{|u(y)|}{1+|y|^{N+2\s}}\,dy\\
&\le C \int_{\Omega \setminus B_{\theta_1}}\frac{|u(y)|}{1+|y|^{N+2\s}}\,dy \le C \|u\|_{L^1(\Omega,\Lambda_\mu dx)} \le C \|f\|_{L^1(\Omega,d\gamma_\mu)}.\end{align*}  
For estimate $g(x)$ for $x \in B_{t_2} \setminus B_{t_1}$, we write 
$$
g(x)= g_1(x)+ g(x)\qquad \text{with} \quad\ g_1(x)=\frac{u(x)}{C_{N,\s}}[(-\Delta)^\s \eta](x),\quad g_2(x)= \int_{\R^N} \frac{(\eta(x)-\eta(y))(u(y)-u(x))}{|x-y|^{N+2\s}}u(x)\,dy.
$$
Now it follows from Proposition~\ref{appendix-main-proposition} and Theorem~\ref{sec:dirichl-probl-bound-comp-corol} that $u \in L^\infty(B_{s_2} \setminus B_{s_1})$ with 
\begin{equation}
  \label{eq:l-infty-local-proof}
\|u\|_{L^\infty(B_{s_2} \setminus B_{s_1})} \le C \Bigl(\|f\|_{L^\infty(\Omega \setminus B_{r_1})}+ \|u\|_{L^1(\Omega)}\Bigr) \le  C \Bigl(\|f\|_{L^\infty(\Omega \setminus B_{r_1})}+\|f\|_{L^1(\Omega,d\gamma_\mu)}\Bigr).
\end{equation}
Consequently, 
$$
\|g_1\|_{L^\infty(B_{t_2} \setminus B_{t_1})} \le \|g_1\|_{L^\infty(B_{s_2} \setminus B_{s_1})} \le C \Bigl(\|f\|_{L^\infty(\Omega \setminus B_{r_1}}+\|f\|_{L^1(\Omega,d\gamma_\mu)}\Bigr).
$$
Since 
$$
(-\Delta)^\s u =  -\frac{\mu}{|\cdot|^2}u + f \qquad \text{in $B_{s_2} \setminus B_{s_1}$,}
$$
it follows from (\ref{eq:l-infty-local-proof}) and regularity estimates in \cite{RS} that $u \in C^{\s}(\overline{B_{\sigma_2} \setminus B_{\sigma_1}})$ with 
\begin{align*}
\|u\|_{C^{\s}(\overline{B_{\sigma_2} \setminus B_{\sigma_1}})} &\le C \Bigl( \|u\|_{L^\infty(B_{s_2} \setminus B_{s_1})}+ \|f\|_{L^\infty(\Omega \setminus B_{r_1})}+ \|u\|_{L^1(\Omega)}\Bigr)\\
&\le C \Bigl(\|f\|_{L^\infty(\Omega \setminus B_{r_1})}+\|f\|_{L^1(\Omega,d\gamma_\mu)}\Bigr).
\end{align*}
Hence for $x \in B_{t_2} \setminus B_{t_1}$, we can write 
\begin{align*}
g_2(x)&= \int_{\overline{B_{\sigma_2} \setminus B_{\sigma_1}}} \frac{(\eta(x)-\eta(y))(u(y)-u(x))}{|x-y|^{N+2\s}}u(x)\,dy +
\int_{\R^N \setminus (\overline{B_{\sigma_2} \setminus B_{\sigma_1}})} \frac{(\eta(x)-\eta(y))(u(y)-u(x))}{|x-y|^{N+2\s}}u(x)\,dy\\ 
&\le \|u\|_{C^{\s}(\overline{B_{\sigma_2} \setminus B_{\sigma_1}})} \|\eta\|_{C^{1}(\overline{B_{\sigma_2} \setminus B_{\sigma_1}})}
\int_{\overline{B_{\sigma_2} \setminus B_{\sigma_1}}}|x-y|^{1-N-\s}\,dy\\&\quad\ + 4\|\eta\|_{L^\infty(\R^N)} \|u\|_{L^1(\Omega)} \max \bigl\{(\sigma_2-t_2)^{-N-2\s}, (t_1-\sigma_1)^{-N-2\s} \bigr\} \\
&\le C \Bigl(\|f\|_{L^\infty(\Omega \setminus B_{r_1})}+\|f\|_{L^1(\Omega,d\gamma_\mu)}\Bigr).
\end{align*}
Combining these estimates, we deduce~(\ref{eq:g-bound}), as required.
\end{proof}

\subsection{Fundamental solution in a bounded domain}

The goal in this subsection is to obtain the unique solution $ \Phi_\mu^\sOmega \in W^\s_*(\Omega)$ of
the problem
\begin{equation}\label{eq 1.1b}
\mathcal{L}_\mu^\s  u= 0\quad
   {\rm in}\ \,  {\Omega}\setminus \{0\},\qquad \  
 u= 0\quad  {\rm   in}\ \,  \Omega^c,\qquad \ 
 \lim_{x\to0}\frac{u(x)}{\Phi_\mu(x)}=1,
\end{equation}
and to derive estimates for it. By definition, $\Phi_\mu^\sOmega$ has an isolated singularity at zero. We have the following result:

\begin{theorem}\label{teo 2}
There exists a unique solution $\Phi_\mu^\sOmega \in W^\s_*(\Omega)$ of the problem (\ref{eq 1.1b}). Moreover, $\Phi_\mu^\sOmega$ is nonnegative in $\Omega$ and satisfies the distributional identity
\begin{equation}\label{1.2b}
 \int_{\Omega}\Phi_\mu^\sOmega  (-\Delta)^\s_{\Gamma_\mu} \xi\, dx =c_{\s,\mu} \xi(0),\qquad \text{for all $\xi\in  C^{1.1}_0(\Omega)$.}
\end{equation}
\end{theorem}

\begin{proof}
The uniqueness follows readily from Lemma~\ref{sec:dirichl-probl-bound-comp}. To obtain the existence and the distributional identity, let $w_1$ be a $C^2$-function on $\R^N$ such that $w_1=\Phi_\mu$ in $\R^N\setminus \Omega$. Moreover, let $f_1=\mathcal{L}_\mu^\s w_1$, let $w_2$ be the solution of (\ref{eq 2.1f}) with $f=f_1$
and let
$$
\Phi_\mu^\sOmega= \Phi_\mu-w_1+w_2,
$$
then $\Phi_\mu^\sOmega$ is a solution of (\ref{eq 1.1b}). Let $\xi\in C_0^2(\Omega)$.
Since $\Phi_\mu$ is in $C^2(\R^N\setminus \{0\})$,
$w_1$ is a bounded $C^2$-function on $\R^N$ and $\Phi_\mu-w_1=0$ in $\R^N\setminus \Omega$, we have, by Theorem~\ref{Theorem B} and (\ref{eq:adjoint-op}),
\begin{align*}
\int_{\Omega}(\Phi_\mu-w_1) (-\Delta)^\s_{\Gamma_\mu}  \xi \,dx &=
\int_{\R^N}(\Phi_\mu-w_1) (-\Delta)^\s_{\Gamma_\mu} \xi \,dx =
c_{\s,\mu}\xi(0)- \int_{\Omega} \xi \cL_\mu^\s w_1 d\gamma_\mu\\
&=c_{\s,\mu}\xi(0)- \int_{\Omega} f_1 \xi d\gamma_\mu.
\end{align*}
Moreover, by Theorem~\ref{sec:dirichl-probl-bound-comp-corol}, 
$$\int_{\Omega}w_2  (-\Delta)^\s_{\Gamma_\mu} \xi \,dx = \int_{\Omega}f_1 \xi d\gamma_\mu.$$
Combining these identities gives (\ref{1.2b}).   This completes the proof.
\end{proof}

We may now complete the proofs of parts (i) and (ii) of Theorem~\ref{theorem-C}.

\begin{proof}[Proof of Theorem~\ref{theorem-C}(i).]
We make use of the ansatz $u_k=k\Phi_\mu^\sOmega+ u_{f_+}- u_{f_-},$ where $f_\pm=\max\{0,\pm f\}$ and $u_{f_{+}} \in L^1(\Omega,\Lambda _\mu dx)$ resp. $u_{f_{-}} \in L^1(\Omega,\Lambda _\mu dx)$ satisfy the distributional identity (\ref{eq:distributional-k}) with $k=0$ and $f=f_{\pm}$, respectively.
To construct $u_{f_+}$, we define $f_n:= \min \{f_+,n\}$ for $n \in \N$. Since $f_+ \in C^\theta_{loc}(\bar\Omega\setminus \{0\})\cap L^1(\Omega,d\gamma_\mu)$, we then have $f_n \in C^\theta_{loc}(\bar\Omega\setminus \{0\}) \cap L^\infty(\bar \Omega)$, and $0 \le f_n\le f_{n+1}\le f$ for all $n$.
Let, for $n \in \N$, $u_n \in W^\s_{*}(\Omega)$ be the solutions corresponding to $f_n$ as given by Theorem~\ref{sec:dirichl-probl-bound-comp-corol}. We then have $0 \le u_n\le u_{n+1}$ and, by \eqref{l1}, 
\begin{equation*}
 \norm{u_n}_{L^1(\Omega,\Lambda _\mu dx)}\le d\norm{f_n}_{L^1(\Omega,d\gamma_\mu)} \le d\norm{f}_{L^1(\Omega,d\gamma_\mu)}\qquad \text{for all $n$ with some constant $d= d(\s,\mu,\rho,\Omega)>0$.}
  \end{equation*}
Consequently, by monotone convergence, there exists $u_{f_{+}}$ with $u_n \to u_{f_{+}}$ in $L^1(\Omega,\Lambda _\mu dx)$. Since also 
$f_n \to f_+ \in L^1(\Omega,d \gamma_\mu)$, we find that, for $\xi \in C^2_c(\Omega)$
\begin{align*}
\int_{\Omega}u_{f_{+}} (-\Delta)^{\s}_{\Gamma_\mu} \xi \,dx &= \lim_{n \to \infty} \int_{\Omega}u_n (-\Delta)^{\s}_{\Gamma_\mu} \xi \,dx
= \lim_{n \to \infty} \int_{\Omega}f_n \xi d \gamma_\mu\\
& = \int_{\Omega}f_+ \xi d \gamma_\mu.   
\end{align*}
Here we used Proposition~\ref{lambda-estimate} for the first equality. Consequently, $u_{f_+}$ satisfy the distributional identity (\ref{eq:distributional-k}) with $k=0$ and $f=f_{+}$. From this it follows that $u_{f_+}$ satisfies 
$$
\cL^\s_\mu u_{f_+} = f_+ \quad\ \text{in $\Omega \setminus \{0\}$,}
$$
whereas $u_{f_+} \equiv 0$ in $\Omega^c$ by construction. So in order show that $u_{f_+}$ satisfies (\ref{eq 2.1fk}) with $f=f_+$ in the sense defined in the introduction, we only need to show that $u_{f_+} \in L^\infty(\R^N \setminus \{0\})$. For this we note that 
Proposition~\ref{sec:dirichl-probl-bound-comp-corol-corol} implies that 
\begin{equation}
  \label{eq:Gamma-mu-comp-away-origin}
\|u_n\|_{L^\infty(\R^N \setminus B_{r_2})} \le c \Bigl(\|f_n\|_{L^\infty(\Omega \setminus B_{r_1})}+ \|f_n\|_{L^1(\Omega,d\gamma_\mu)}\Bigr)
\le c \Bigl(\|f_+\|_{L^\infty(\Omega \setminus B_{r_1})}+ \|f_+\|_{L^1(\Omega,d\gamma_\mu)}\Bigr)
\end{equation}
for fixed numbers $0<r_1<r_2$ and every $n \in \N$. Since $u_n \to u_{f_{+}}$ in $L^1(\Omega,\Lambda _\mu dx)$ and therefore a.e. in $\R^N$ after passing to a subsequence, we conclude that 
$$
\|u_{f_+}\|_{L^\infty(\R^N \setminus B_{r_2})} \le c \Bigl(\|f_+\|_{L^\infty(\Omega \setminus B_{r_1})}+ \|f_+\|_{L^1(\Omega,d\gamma_\mu)}\Bigr).
$$
Hence $u_{f_+} \in L^\infty(\R^N \setminus \{0\})$, as required.
 
By the same arguments, we see that $u_{f_-} L^\infty(\R^N \setminus \{0\})\cap L^1(\R^N)$ solves (\ref{eq 2.1fk}) with $f=f_-$ in the sense defined in the introduction, and it satisfy the distributional identity (\ref{eq:distributional-k}) with $k=0$ and $f=f_{-}$.
We thus conclude that $u_k=k\Phi_\mu^\sOmega+ u_{f_+}- u_{f_-}$ has the required properties. 
\end{proof}

\begin{proof}[Proof of Theorem~\ref{theorem-C}(ii).]
{\em Existence:} Let $f \in L^\infty(\Omega, |x|^{\rho}dx)$ for some $\rho < 2\s - \tau_+(\s,\mu)$, and $u_f \in W^\s_*(\Omega)$ be the unique solution of (\ref{eq 2.1f}). Then, by Theorem~\ref{sec:dirichl-probl-bound-comp-corol} and Proposition~\ref{teo 2}, the function $u_k=k\Phi_\mu^\sOmega+ u_{f}$ has the desired properties.\\
{\em Uniqueness:} Suppose that $u_k^1,u_k^2 \in L^1(\Omega,\Lambda_\mu dx)$ are solutions of problem (\ref{eq 2.1fk}) in the sense defined in the introduction which satisfy 
$$
 \lim_{x \to 0}\:\frac{u_k^1(x)}{\Phi_\mu(x)} = k =  \lim_{x \to 0}\:\frac{u_k^2(x)}{\Phi_\mu(x)}.
$$
Then $u:= u_k^1-u_k^2 \in L^1(\Omega)\cap L^\infty(\R^N \setminus \{0\})$ satisfies 
$$
\cL^\s_\mu u = 0 \quad \text{in $\Omega \setminus \{0\}$},\quad  u=0 \quad \text{in $\Omega^c$},\quad  \lim_{x \to 0}\:\frac{u(x)}{\Phi_\mu(x)}= 0.
$$
By Remark \ref{weak-solution-more-regular}(ii) and Lemma~\ref{sec:dirichl-probl-bound-comp}, it follows that $u =0$, hence $u_k^1 = u_k^2$.
\end{proof}

\setcounter{equation}{0}
\section{Nonexistence}
\label{sec:nonexistence}
This section is devoted to the proof of our nonexistence results. In the following subsection, we give the proof of Theorem~\ref{theorem-C}(iii), which is devoted to the case $\mu \ge \mu_0$ and to measurable functions $f \ge 0$ with $\int_{\Omega} f d\gamma_\mu = \infty$ in (\ref{eq 2.1fk}). Here, as before, we put $d\gamma_\mu = \Gamma_\mu dx$.

\subsection{Nonnegative  source functions $f\not\in L^1(\Omega,d\gamma_\mu)$ }
We start with the following lemma.

\begin{lemma}\label{lm 2.3}
Let $(r_n)_n$ be a decreasing sequence of positive numbers, and let $\delta_n$, $n \in \N$ be nonnegative  $L^\infty$-functions with supp $\,\delta_n\subset B_{r_n}(0)$ and 
$$
\int_\Omega \delta_n dx=\int_{B_{r_n}} \delta_n dx =1 \quad\ \text{for all $n \in \N$.}
$$
For any $n$, let $w_n \in W_*^\s(\Omega)$  be the unique solution of (\ref{eq 2.1f}) with $f:=f_n = \frac{\delta_n}{\Gamma_\mu}$. 
Then, after passing to a subsequence, we have 
$$
w_n \to w_\infty \quad\ \text{in $L^1(\Omega,\Lambda_\mu dx)$ and a.e. in $\Omega$,}
$$
where $w_\infty \in L^1(\Omega,\Lambda_\mu dx)$ satisfies the distributional identity 
\begin{equation}\label{3.0-infty}
  \int_{\Omega} w_\infty  (-\Delta)_{\Gamma_\mu}^\s \xi \, dx = \xi(0) \quad\ 
\text{for all $\xi\in C^{2}_0(\Omega)$.}
\end{equation}
\end{lemma}

\begin{proof}
We first note that $w_n \ge 0$ for all $n$ by Theorem~\ref{sec:dirichl-probl-bound-comp-corol}(i), and 
\begin{equation}
  \label{eq:uniform-bound-wn}
\|w_n\|_{L^1(\Omega,\Lambda_\mu dx)} \le d \int_{\Omega} \frac{\delta_n}{\Gamma_\mu}d\gamma_\mu = d \int_{\Omega}\delta_n dx = d \quad\ \text{for all $n \in \N$}
\end{equation}
by (\ref{l1}). Moreover, it follows from Proposition~\ref{sec:dirichl-probl-bound-comp-corol-corol} that 
\begin{equation}
  \label{eq:uniformly-bounded-l-infty}
\text{the sequence $w_n$ remains uniformly bounded in $L^\infty_{loc}(\R^N \setminus \{0\})$.}  
\end{equation}
Moreover, since for every $r>0$ there exists $n_r \in \N$ such that $f_n \equiv 0$ in $\Omega \setminus B_{r}(0)$ for $n \ge n_r$ and therefore 
$$
(-\Delta)^\s w_n = \frac{\mu}{|\cdot|^{2\s}}w_n \qquad \text{in $\Omega \setminus B_{r}(0)$ for $n \ge n_r$,}
$$
it follows from \cite{RS} that $(w_n)_n$ remains bounded in $C^\s_{loc}(\Omega \setminus \{0\})$. In particular, we may pass to a subsequence such 
\begin{equation}
  \label{eq:convergence-loc-l-infty}
\text{$w_n \to w_\infty$ in  $L^\infty_{loc}(\Omega \setminus \{0\})$ for a function $w_\infty \in L^\infty_{loc}(\Omega \setminus \{0\})$.}
\end{equation}
In particular, $w_n \to w_\infty$ a.e. in $\Omega$. Moreover, $w_\infty \in L^1(\Omega,\Lambda_\mu dx)$ as a consequence of (\ref{eq:uniform-bound-wn}) and Fatou's Lemma. 

Next, we claim that $w_n \to w_\infty$ in $L^1(\Omega,\Lambda_\mu dx)$. For this we define $\varrho_0>0$ by 
\begin{equation}\label{de rho0}
\varrho_0=\left\{\arraycolsep=1pt\begin{array}{lll}
   \s-\frac12\tau_+(\s,\mu) \quad\
   &{\rm if}\ \, \tau_+(\s,\mu) >2\s-1 ,\\[1mm]
 \phantom{   }
\frac12 \quad\ &{\rm   if } \ \,  \tau_+(\s,\mu) \leq 2\s-1
 \end{array}
 \right. 
 \end{equation}
and 
$$
\rho = \left \{
  \begin{aligned}
   & \s-\frac12\tau_+(\s,\mu) &&\quad\
   \text{if \,$\tau_+(\s,\mu) >2\s-1$,}\\
   &2\s -\tau_+(\s,\mu)- \frac{1}{4}&&\quad\
   \text{if \,$\tau_+(\s,\mu)  \le 2\s-1$.}     
  \end{aligned}
\right.
$$
Moreover, we define the functions $g_r= \Lambda_\mu \chi_r: \Omega \to \R$ for $r \in (0,1)$, where $\chi_r$ denotes the characteristic function of the ball $B_{r}(0)$. By (\ref{def-Lambda}) and the choices of $\rho_0$ and $\rho$, we have 
\begin{align*}
\|g_r\|_{L^\infty(\Omega,|x|^\rho dx)}&\leq
\left\{
\begin{aligned}
  &r^\rho &&\quad\ {\rm if}\ \, \tau_+(\s,\mu) >2\s-1,\\
  &r^\rho (1- \ln r) &&\quad\ {\rm   if } \ \,  \tau_+(\s,\mu) =2\s-1,\\
&r^{\rho+1-2\s+\tau_+(\s,\mu)} &&\quad\ {\rm   if } \ \,  \tau_+(\s,\mu) <2\s-1
\end{aligned}
 \right.\\
&\le c\, r^{\varrho_0} 
\end{align*}
with a constant $c>0$ independent of $r \in (0,1)$. Since $\rho < 2\s -\tau_+(\s,\mu)$, we may apply (\ref{g-estimate}) with $g=g_r$ and deduce that 
$$
\|w_n\|_{L^1(B_{r}(0)),\Lambda_\mu dx)} = \int_{\Omega}g_r w_n\,dx 
\le c \|g_r \|_{L^\infty(\Omega,|x|^{\rho} dx)} \int_{\Omega} \frac{\delta_n}{\Gamma_\mu}d\gamma_\mu = 
c \|g_r\|_{L^\infty(\Omega,|x|^{\rho})}\le c r^{\varrho_0}
$$
with a (possibly different) constant $c>0$ independent of $r \in (0,1)$ and $n \in \N$. The same estimate then follows for $w_\infty$ in place of $w_n$, which implies that 
$$
\int_{B_r} |w_n-w_\infty|\,\Lambda_\mu dx  \leq   \int_{B_r} |w_n|\,\Lambda_\mu dx + \int_{B_r} |w_\infty|\,\Lambda_\mu dx \leq 2 c  r^{\varrho_0}\quad \text{for $n \in \N$.}
$$
Combining this estimate with the fact with (\ref{eq:uniformly-bounded-l-infty}), (\ref{eq:convergence-loc-l-infty}) and the fact that $\Omega$ is bounded, we see that 
$$
\limsup_{n \to \infty}\int_{\Omega} |w_n-w_\infty|\,\Lambda_\mu dx \le 2 c  r^{\varrho_0}\quad\ \text{for every $r \in (0,1)$.}
$$
Since $\varrho_0>0$, we thus conclude that $w_n \to w_\infty$ in $L^1(\Omega,\Lambda_\mu dx)$, as claimed. 

It thus remains to prove the distributional identity \eqref{3.0-infty}. For this we let $\xi\in C^{2}_0(\Omega)$. It then follows from Proposition~\ref{lambda-estimate} and Theorem~\ref{sec:dirichl-probl-bound-comp-corol} that 
\begin{align*}
\int_{\Omega} w_\infty  (-\Delta)_{\Gamma_\mu}^\s \xi \, dx &= \lim_{n \to \infty} \int_{\Omega} w_n  (-\Delta)_{\Gamma_\mu}^\s \xi \, dx   \\
&=
\lim_{n \to \infty} \int_{\Omega} f_n \xi \, d\gamma_\mu =\lim_{n \to \infty} \int_{\Omega} \delta_n \xi \, dx = \xi(0).
\end{align*}
We complete the proof. \end{proof}

We may now complete the 

\begin{proof}[Proof of Theorem \ref{theorem-C}(iii)]
By contradiction, we assume that problem (\ref{eq 1.1f}) has a nonnegative solution of $u_f \in L^1(\R^N,\frac{dx}{1+|x|^{N+2s}}) \cap L^\infty_{loc}(\R^N \setminus \{0\})$. By Remark\eqref{weak-solution-more-regular}, we have $u_f \in W^\s_*(\Omega)$, and $u_f$ satisfies $\cL^\s_\mu u_f  = f$ in $\Omega$ in the sense of Definition \ref{weak-solution}. Let $\{r_n\}_n$ be a sequence of strictly decreasing positive numbers converging to zero.
From (\ref{f2}) and since $f\in L^\infty_{loc}(\overline{\Omega}\setminus \{0\})$, we have that,  for any $r_n$, 
$$
 \lim_{r\to0^+} \int_{B_{r_n}\setminus B_r}f(x) d\gamma_\mu  =+\infty,
$$
then there exists $s_n\in (0,r_n)$ such that
$$
  \int_{B_{r_n}\setminus B_{s_n}}f(x)  d\gamma_\mu =n,
$$
In the following, let $\chi_n$ denote the characteristic function of $B_{r_n}\setminus B_{s_n}$, 
and define $\delta_n=\frac1n \Gamma_\mu f\chi_{n} \in L^\infty(\Omega)$, $n \in \N$. Then the sequence $\delta_n$ satisfies the assumptions of Lemma~\ref{lm 2.3}. Let, for $n \in \N$, $w_n \in W_*^\s(\Omega)$ be the unique solution of (\ref{eq 2.1f}) with $f:=f_n = \frac{\delta_n}{\Gamma_\mu}= \frac1n f\chi_{n}$. The comparison principle given in Lemma~\ref{sec:dirichl-probl-bound-comp} then shows that
$\frac{u_f}{n} \ge w_n$ in $\Omega$ for every $n \in \N$. In particular, this implies that $w_n \to 0$ a.e. in $\Omega$. 
On the other hand, Lemma~\ref{lm 2.3} implies that $w_n \to w_\infty$ a.e. in $\Omega$, where $w_\infty \in L^1(\Omega,\Lambda_\mu dx)$ satisfies the distributional identity (\ref{3.0-infty}), so it cannot hold $w_\infty =0$ a.e. in $\Omega$. This contradiction shows 
$u_f$ cannot exist, and the proof is finished.
\end{proof}

\subsection{Nonexistence in the case $\mu< \mu_0$}
\label{sec:nonexistence-case-mu}

In this subsection we give the proof of Theorem~\ref{theorem-D}. Thus, we consider $\mu<\mu_0$ here in contrast to the remainder of the paper, and we let $f \in L^\infty_{loc}(\overline{\Omega} \setminus \{0\})$ be an arbitrary nonnegative function.  

 \begin{proof}[Proof of Theorem \ref{theorem-D}]
By contradiction, we assume that $u_0$ is a nontrivial nonnegative solution of (\ref{eq 1.1f}). By Remark~\ref{weak-solution-more-regular}(ii), $u_0 \in W^*_\s(\Omega)$ is a solution in the sense of Definition~\ref{weak-solution}. 

Fix $t>0$ be such that $\tilde u_0 \min \{t,- \frac{\mu}{|\cdot|^{2\s}}u_0\} \in L^\infty(\Omega)$ is a nontrivial function, and let 
$u_1 \in \cH^\s_0(\Omega)$ be the unique solution of
\begin{equation*}
\left\{
  \begin{aligned}
 (-\Delta)^\s u&=\tilde u_0 &&\quad \text{in \, $\Omega$},\\
u&=0 &&\quad  \text{in \, $\Omega^c.$}
  \end{aligned}
\right.
\end{equation*}
Then $u_1$ is continuous, and by the strong maximum principle for $(-\Delta)^\s$, there exists $r_0>0$ such that $B_{r_0}(0) \subset \Omega$ and $t_1>0$ such that
$$u_1\ge t_1 \quad{\rm in}\ \, B_{r_0}.$$
Now $w:= u_0- u_1 \in \cW^\s_*(\Omega)$ satisfies 
$$
(-\Delta)^\s w = f- \frac{\mu}{|\cdot|^{2\s}}u_0- \min \{t,- \frac{\mu}{|\cdot|^{2\s}}u_0\}  \ge 0 \quad \text{in $\Omega \setminus \{0\}$},\qquad \text{$u \equiv 0$ in $\Omega^c$}
$$
and $\liminf \limits_{x \to 0} \frac{w(x)}{\Phi_0(x)}\ge 0$. From Lemma~\ref{sec:dirichl-probl-bound-comp}, we infer that $w \ge 0$ and hence $u_0 \ge u_1$ in $\Omega$. In particular, 
$$
u_0\ge u_1\ge t_1\quad{\rm in}\ \, B_{r_0}.
$$
Next, define 
$$
u_2 \in W^\s_*(\Omega),\qquad  u_2(x)=\frac{\mu_0-\mu}{-\mu_0}\frac{t_1}{r_0^{\tau_+(\s,\mu_0)}}(\Gamma_{\mu_0}(x)-r_0^{\tau_+(\s,\mu_0)}),
$$
and observe that
   \begin{equation*}
\mathcal{L}^\s_{\mu_0} u =\frac{(\mu_0-\mu)t_1}{|x|^{2\s}} \quad\ \text{in $B_{r_0}\setminus\{0\}$},\qquad \   u \le 0\quad   \text{in $B_{r_0}^c,$}
\end{equation*}
whereas $u_0$ satisfies
 \begin{equation*}
\mathcal{L}^\s_{\mu_0} u_0= \frac{\mu_0-\mu }{|x|^{2\s}}u_0+f\ge \frac{(\mu_0-\mu)t_1}{|x|^{2\s}}\quad \text{in $B_{r_0}\setminus\{0\}$},\qquad   u_0\ge 0 \quad \text{in $B_{r_0}^c$.}
\end{equation*}
Since also $\liminf \limits_{x \to 0} \frac{u_0(x)-u_2(x)}{\Phi_{\mu_0}(x)}\ge 0$ 
Lemma~\ref{sec:dirichl-probl-bound-comp} implies that 
$u_0 \ge u_2$ in $B_{r_0}$ and therefore 
$$
u_0(x)\ge c |x|^{-\frac{N-2\s }{2}} \quad{\rm in}\ \, B_{\frac{r_0}2}
$$
with a constant $c>0$. Hence $u_0$ is a solution of
\begin{equation}\label{eq 5.1}
\arraycolsep=1pt\left\{
 \begin{array}{lll}
 \displaystyle  \mathcal{L}^\s_{\mu_0} u= \tilde f \ \ &{\rm in}\  \, \Omega\setminus\{0\} ,\\[2mm]
 \phantom{  \mathcal{L}^\s_{\mu_0} }
 \displaystyle  u= 0 \ \  &{\rm   in}\   \, \Omega^c,
 \end{array}\right.
\end{equation}
where $\tilde f=\frac{\mu_0-\mu }{|\cdot|^{2\s}}u_0+f\ge 0$ and
$$
\int_{\Omega} \tilde f d\gamma_{\mu_0} \ge \int_{B_r } (\mu_0-\mu)\frac{u_0(x)}{|x|^{2\s}}  d\gamma_{\mu_0} 
\ge (\mu_0-\mu) t_1\int_{B_{r_0} }|x|^{-N}dx = \infty.
$$
Applying Theorem~\ref{theorem-C}(iii) to $\tilde f$ and $\mu=\mu_0$, we obtain a contradiction.
Therefore we conclude that  problem (\ref{eq 1.1f}) has no nontrivial nonnegative solutions.
\end{proof}

 
\section{Appendix}
\label{sec:appendix}

\subsection{Appendix A. An $L^\infty$-estimate for solutions of linear inhomogeneous fractional equations}
\label{sec:appendix-a.-an}

In this section we prove the following.

\begin{proposition}
\label{appendix-main-proposition}
 Let $x_0 \in \R^N$, $r>0$, $\rho \in (0,1)$ and $V,f \in L^\infty(B_r(x_0))$. Moreover, let $u \in L^1(\R^N,\frac{dx}{1+|x|^{N+2\s}})$ be a distributional solution of 
$$
(-\Delta)^s u + V(x) u = f \qquad \text{in $B_r(x_0)$.}
$$
Then 
$$
\|u\|_{L^\infty(B_{\rho r}(x_0))} \le C \Bigl(\|f\|_{L^\infty(B_{r}(x_0))} + \|u\|_{L^1(\R^N,\frac{dx}{1+|x|^{N+2\s}})}\Bigr)
$$
with a constant $C=C(N,V,r,\rho)>0$.  
\end{proposition}

We start with the following lemma.

\begin{lemma}
\label{appendix-harmonic}
 Let $x_0 \in \R^N$, $r>0$, $\rho \in (0,1)$ and let $u \in L^1(\R^N,\frac{dx}{1+|x|^{N+2\s}})$ satisfy
$$
(-\Delta)^s u = 0 \qquad \text{in $B_r(x_0)$}
$$
in distributional sense. Then $u \in C^\infty(B_{r}(x_0))$, and there exists a constant $C=C(N,r,\rho)$ with  
$$
\|u\|_{L^\infty(B_{\rho r}(x_0))} \le C \|u\|_{L^1(\R^N,\frac{dx}{1+|x|^{N+2\s}})}.
$$
\end{lemma}

\begin{proof}
Without loss of generality, we may assume that $x_0=0$. It has been proved in \cite[Theorem 3.12]{BB} (see also \cite[Theorem 2.6]{FW}) that $u \in C^\infty(B_{r})$, and that $u= \tilde \Gamma * u$ in $B_{\rho r}$
with a function $\tilde \Gamma \in C^\infty(\R^N) \cap L^\infty(\R^N,(1+|y|^{N+2\s}) dy)$ depending on $r$ and $\rho$ but not on $u$.  Consequently,  
$$
\|u\|_{L^\infty(B_{\rho r})} \le C \|u\|_{L^1(\R^N,\frac{dx}{1+|x|^{N+2\s}})}
\quad \text{with}\quad C:= \sup_{\stackrel{x \in B_{\rho r}}{y \in \R^N}}|\tilde \Gamma(x-y)|(1+|y|^{N+2\s})<\infty.
$$We complete the proof. 
\end{proof}

\begin{lemma}
\label{appendix-poisson-simple}
 Let $x_0 \in \R^N$, $r>0$, $\rho \in (0,1)$ and let $u \in L^1(\R^N,\frac{dx}{1+|x|^{N+2\s}})$, $f \in L^\infty(B_r(x_0))$ satisfy
$$
(-\Delta)^s u = f \quad\ \text{in $B_r(x_0)$}
$$
in distributional sense. Then there exists a constant $C=C(N,r,\rho)$ with  
$$
\|u\|_{L^\infty(B_{\rho r}(x_0))} \le C \Bigl(\|u\|_{L^1(\R^N,\frac{dx}{1+|x|^{N+2\s}})}+\|f\|_{L^\infty(B_r(x_0))}\Bigr).
$$
\end{lemma}

\begin{proof}
Again we assume that $x_0=0$. We write $u=u_1+u_2$ with 
$$
u_1: \R^N \to \R, \qquad u_1(x)= \kappa_{N,s} \int_{B_1}|x-y|^{2s-N}f(y)\,dy,\qquad \kappa_{N,s}=\frac{\Gamma(\frac{N}{2}-s)}{4^s\pi^{N/2}\Gamma(s)}.
$$
In distributional sense, we then have 
$$
(-\Delta)^s u_1 = f \qquad \text{and}\qquad (-\Delta)^s u_2 = 0 \quad \text{in $B_1$.}
$$
Moreover,  
$$
\|u_1\|_{L^\infty(\R^N)}\le \kappa_{N,s} \|f\|_{L^\infty(B_1)} \int_{B_1}|x-y|^{2s-N}\,dy \le C_1 \|f\|_{L^\infty(B_1)}\quad\ \text{with}\quad C_1:= \frac{\kappa_{N,s}\omega_{N-1}}{2s}, 
$$
and by Lemma~\ref{appendix-harmonic} we have
\begin{align*}
\|u_2\|_{L^\infty(B_{\rho r}(x_0))} &\le C \|u_2\|_{L^1(\R^N,\frac{dx}{1+|x|^{N+2\s}})} \le 
C\bigl(\|u\|_{L^1(\R^N,\frac{dx}{1+|x|^{N+2\s}})}+\|u_1\|_{L^1(\R^N,\frac{dx}{1+|x|^{N+2\s}})}\bigr)\\[1.5mm]
&\le C\bigl(\|u\|_{L^1(\R^N,\frac{dx}{1+|x|^{N+2\s}})}+\|u_1\|_{L^\infty(\R^N)}\bigr)
\le C\bigl(\|u\|_{L^1(\R^N,\frac{dx}{1+|x|^{N+2\s}})}+\|f\|_{L^\infty(B_1)}\bigr)
\end{align*}
with possibly changing constants $C>0$ depending only on $N,r$, and $\rho$.
\end{proof}

We also need the following lemma

\begin{lemma}\label{embedding}\cite[Proposition 1.4]{RS0}
Let $r>0$. Then, for every $h \in L^1(B_r)$, there exists a distributional solution 
$u:= \mathbb{G}_r[h]  \in L^1(B_r)$ of $(-\Delta)^\s u = h$ in $B_r$ satisfying $u = 0$ on $\R^N \setminus B_r$ and  
given as $u(x)= \int_{B_r} G_r(x,y)h(y)dy$ for $x \in B_r$, where $G_r$ denotes the Green function for $(-\Delta)^\s$ on $B_r$. Moreover, if $h \in L^s(B_r)$ for some $s \ge 1$, we have the following estimates.
\begin{enumerate}
\item[(i)] $\|\mathbb{G}_r[h] \|_{L^\infty (B_r)}\le c\|h\|_{L^s(B_r)}$ if $\frac1s<\frac{2\s }N$;
\item[(ii)] $\|\mathbb{G}_r[h]\|_{L^\tau(B_r)}\le c\|h\|_{L^s(B_r)}$ if $\frac1s\le \frac1\tau+\frac{2\s}N$ and $s>1$;
\item[(iii)] $\|\mathbb{G}_r[h]\|_{L^\tau(B_r)}\le c\|h\|_{L^1(B_r)}$ if $1<\frac1\tau+\frac{2\s}N.$
\end{enumerate}
 \end{lemma}

\begin{proof}[Proof of Proposition~\ref{appendix-main-proposition}]
Again we may assume that $x_0=0$. The proof follows a bootstrap procedure. We choose $\tau_0>1$ with $1<\frac{1}{\tau_0}+\frac{2\s}N$ and define inductively, for every positive integer $i$,
$$
\tau_{i+1}:= \left \{
  \begin{aligned}
  &\frac{\tau_i N}{N-2\s \tau_i}&&\quad\ \text{if $\tau_i < \frac{N}{2\s}$},\;\\
  &\tau_i &&\quad\ \text{if $\tau_i \ge \frac{N}{2\s}$.}    
  \end{aligned}
\right.
$$
Adjusting the value of $\tau_0 = \tau$ if necessary, we may assume that $\tau_i \not = \frac{N}{2\s}$ for all $i \in \N$. 
Since $\tau_{i+1}-\tau_i=\frac{2\s \tau_i}{N-2\s \tau_i} \ge \frac{2\s   }{N-2\s }$ if 
$\tau_i<\frac{N}{2\s}$, it follows that there is a minimal $i_0 \ge 1$ such that $\tau_{i_0}>\frac{N}{2\s}$.
We then fix $\rho \in (0,1)$ sufficiently close to $1$ so that $\rho^{\tau_{i_0}}> \frac{1}{2}$.
Recalling the definition of $\cG_r$ from Lemma~\ref{embedding}, we now write $u= w + \cG_r (Vu)$, and we note that 
\begin{equation}
  \label{eq:tau-bootstrap-1}
\|\mathbb{G}_r[Vu]\|_{L^1(B_{r})} \le \|\mathbb{G}_r[Vu]\|_{L^{\tau_0}(B_{r})} \le  c\|Vu\|_{L^1(B_{r})} \le c \|u\|_{L^1(B_{r})}
\end{equation}
by Lemma~\ref{embedding}(iii). Hence $w = u-\mathbb{G}_r[Vu] \in L^1(\R^N,\frac{dx}{1+|x|^{N+2\s}})$, and it is a solution of 
$(-\Delta)^\s w = f$ in $B_{r}$. By Lemma~\ref{appendix-poisson-simple} and (\ref{eq:tau-bootstrap-1}), it follows that 
$$
\|w\|_{L^\infty(B_{\rho r})} \le c \Bigl(\|w\|_{L^1(\R^N,\frac{dx}{1+|x|^{N+2\s}})}+\|f\|_{L^\infty(B_r)}\Bigr) \le c \Bigl(\|u\|_{L^1(\R^N,\frac{dx}{1+|x|^{N+2\s}})}+\|f\|_{L^\infty(B_r)}\Bigr).
$$
Therefore, using (\ref{eq:tau-bootstrap-1}) again, we infer that 
$$
\|u\|_{L^{\tau_0}(B_{\rho r})} \le \|w\|_{L^{\tau_0}(B_{\rho r})}  + \|\mathbb{G}_r [Vu]\|_{L^{\tau_0}(B_{\rho r})} \le c \Bigl(\|u\|_{L^1(\R^N,\frac{dx}{1+|x|^{N+2\s}})}+\|f\|_{L^\infty(B_r)}\Bigr).
$$
We shall now prove inductively
\begin{equation}
  \label{eq:inductive-inequality}
\|u\|_{L^{\tau_i}(B_{\rho^{i} r})} \le c \Bigl(\|u\|_{L^1(\R^N,\frac{dx}{1+|x|^{N+2\s}})}+\|f\|_{L^\infty(B_r)}\Bigr) \qquad \text{for $i=0,\dots,i_0-1$.}
\end{equation}
The case $i=0$ has already been considered. Now suppose that (\ref{eq:inductive-inequality}) holds for some $i \in \{0,\dots,i_0-2\}$. We then write $u= \tilde w + \mathbb{G}_{\rho^{i}r}[Vu]$. By Lemma~\ref{embedding}(ii),
\begin{equation}
  \label{eq:tau-bootstrap-induction}
\|\mathbb{G}_{\rho^i r}[Vu]\|_{L^{1}(B_{\rho^i r})}\le c \|\mathbb{G}_{\rho^i r}(Vu)\|_{L^{\tau_{i+1}}(B_{\rho^i r})} \le  c\|Vu\|_{L^{\tau_i}(B_{\rho^i r})} \le c \|u\|_{L^{\tau_i}(B_{\rho^i r})}.
\end{equation}
Moreover $\tilde w = u-\mathbb{G}_{\rho^i r}[Vu] \in L^1(\R^N,\frac{dx}{1+|x|^{N+2\s}})$ is a solution of 
$(-\Delta)^\s \tilde w = f$ in $B_{\rho^i r}$, and therefore, by Lemma~\ref{appendix-poisson-simple} and \eqref{eq:tau-bootstrap-induction}, 
\begin{align*}
 \|\tilde w\|_{L^\infty(B_{\rho^{i+1} r})} &\le C \Bigl(\|\tilde w\|_{L^1(\R^N,\frac{dx}{1+|x|^{N+2\s}})}+\|f\|_{L^\infty(B_{\rho^i r})}\Bigr)\\
& \le c \Bigl(\|u\|_{L^1(\R^N,\frac{dx}{1+|x|^{N+2\s}})} + \|u\|_{L^{\tau_i}(B_{\rho^i r})} +  \|f\|_{L^\infty(B_r)}\Bigr) \le c \Bigl(\|u\|_{L^1(\R^N,\frac{dx}{1+|x|^{N+2\s}})} +  \|f\|_{L^\infty(B_r)}\Bigr).
\end{align*}
Here we used the induction hypothesis in the last step. Using this estimate, (\ref{eq:tau-bootstrap-induction}) and again the induction hypothesis, we find that  
\begin{align*}
\|u\|_{L^{\tau_{i+1}} (B_{\rho^{i+1} r})} &\le \|\tilde w\|_{L^{\tau_{i+1}}(B_{\rho^{i+1} r})}  + \|\mathbb{G}_r [Vu]\|_{L^{\tau_{i+1}}(B_{\rho^{i+1} r})}
\\&\le c \Bigl(\|\tilde w\|_{L^{\infty}(B_{\rho^{i+1} r})}  + \|u\|_{L^{\tau_{i}}(B_{\rho^{i} r})}\Bigr)
\le c \Bigl(\|u\|_{L^1(\R^N,\frac{dx}{1+|x|^{N+2\s}})} +  \|f\|_{L^\infty(B_r)}\Bigr).
\end{align*}
Inductively, we deduce that (\ref{eq:inductive-inequality}) holds for $i=1,\dots,i_0-1$. By Proposition \ref{embedding} $(i)$,
it then follows that
\begin{equation}
  \label{eq:tau-bootstrap-final}
\|\mathbb{G}_{\rho^{i_0} r}(Vu)\|_{L^\infty(B_{\rho^{i_0-1}r})} \le c \|u\|_{L^{\tau_{i_0-1}}(B_{\rho^{i_0-1} r})}\le 
c \Bigl(\|u\|_{L^1(\R^N,\frac{dx}{1+|x|^{N+2\s}})}+\|f\|_{L^\infty(B_r)}\Bigr).
\end{equation}
Writing, similarly as before, $u = \tilde {\tilde w}+\mathbb{G}_{\rho^{i_0-1} r}[Vu]$, we see that $\tilde {\tilde w} \in L^1(\R^N,\frac{dx}{1+|x|^{N+2\s}})$ is a solution of 
$(-\Delta)^\s \tilde w = f$ in $B_{\rho^{i_0-1} r}$, and therefore, by Lemma~\ref{appendix-poisson-simple} and (\ref{eq:tau-bootstrap-final}), 
\begin{equation}
  \label{eq:tau-bootstrap-final-1}
\|\tilde {\tilde w}\|_{L^\infty(B_{\rho^{i_0} r})} \le C \Bigl(\|\tilde{\tilde w}\|_{L^1(\R^N,\frac{dx}{1+|x|^{N+2\s}})}+\|f\|_{L^\infty(B_{\rho^{i_0-1} r})}\Bigr) \le c \Bigl(\|u\|_{L^1(\R^N,\frac{dx}{1+|x|^{N+2\s}})} +  \|f\|_{L^\infty(B_r)}\Bigr). 
\end{equation}
Combining (\ref{eq:tau-bootstrap-final}) and (\ref{eq:tau-bootstrap-final-1}) yields 
 \begin{align*}
\|u\|_{L^{\infty}(B_{\frac{r}{2}})}\le  \|u\|_{L^{\infty}(B_{\rho^{i_0} r})} &\le c \Bigl(\|w\|_{L^{\infty}(B_{\rho r})}  + \|\mathbb{G}_r [Vu]\|_{L^{\infty}(B_{\rho^{i_0} r})}\Bigr)
\\&\le c \Bigl(\|u\|_{L^1(\R^N,\frac{dx}{1+|x|^{N+2\s}})}+\|f\|_{L^\infty(B_r)}\Bigr),
\end{align*}
as claimed.
\end{proof}

\subsection{Appendix B. A density property}
\label{sec:appendix-b.-density}

In this appendix, we give the proof of a fact used in Remark \ref{weak-solution-more-regular}(i). More precisely, we shall prove the following.

\begin{proposition}
\label{appendix-b-prop}
Let $\Omega \subset \R^N$ be a bounded domain with $0 \in \Omega$. For every nonnegative function $v \in \cH^\s_0(\Omega)$, there exists a sequence of functions $v_n \in \cH^\s_0(\Omega)$, $n \in \N$ with $v_n \ge 0$, 
$v_n \equiv 0$ in a neighborhood of $0$ and $v_n \to v$ in $\cH^\s_0(\Omega)$ as $n\to+\infty$.
\end{proposition}

\begin{proof}
We define the functions
$$
x \mapsto v_n(x)=v(x)\eta_n(|x|),
$$ 
where
$\eta_n(r)=\eta_0(nr)$ for $r\ge0$ and $\eta_0:[0,+\infty)\to[0,1]$ is a non-decreasing $C^\infty$ function such that $\eta_0(t)=0$ for $t\in[0,\frac12]$ and
$\eta_0(t)=1$ for $t\ge 1$.
We   show that $v_n\to v$ in  $\cH^\s_0(\Omega)$ as $n\to+\infty$, which is equivalent to that
$$
 v\tilde \eta_n\to 0\quad{\rm in}\quad  \cH^\s_0(\Omega)\ \ {\rm as}\ \, n\to+\infty,
$$
 where $\tilde \eta_n=1-\eta_n$.
 Note that
 \begin{equation}\label{Apped 2}
\norm{ v\tilde \eta_n}_{\cH^\s_0(\Omega)}^2\le \int_{\R^N\times \R^N} \frac{(v(x)-v(y))^2}{|x-y|^{N+2\s}}\tilde \eta_n^2(x)  \,dxdy +
\int_{\R^N\times \R^N} \frac{(\tilde \eta_m(x)-\tilde \eta_m(y))^2}{|x-y|^{N+2\s}}v^2(y)\,dx dy.
 \end{equation}
By direct computation, the first term of the right hand side of (\ref{Apped 2}) satisfies the estimate
\begin{equation*}
\int_{\R^N\times \R^N} \frac{(v(x)-v(y))^2}{|x-y|^{N+2\s}}\tilde \eta_m^2(x)  \,dxdy  \,\le \, \int_{B_{\frac1n} }\int_{\R^N}\frac{(v(x)-v(y))^2}{|x-y|^{N+2\s}}dy dx \,\to \, 0\quad {\rm as}\ \, n\to+\infty.
\end{equation*}
Since $\tilde \eta\in C^\infty([0,\infty))$ has compact support in $[0,1]$, then for the second term of the right hand side of (\ref{Apped 2}), we have that  $r=\frac{2}n$,
\begin{equation*}
  \int_{\R^N\times \R^N} \frac{(\tilde \eta_m(x)-\tilde \eta_m(y))^2}{|x-y|^{N+2\s}}v^2(y)\,dx dy   = \int_{B_r  } K_{n}(y) v^2(y)\,  dy+\int_{\R^N\setminus B_r  } K_{n}(y) v^2(y)\,  dy,
\end{equation*}
where
\begin{equation*}
K_n(x)=\int_{\R^N } \frac{(\tilde \eta_n(x)-\tilde \eta_n(z))^2}{|x-z|^{N+2\s}}dz
= \int_{B_r  } \frac{(\tilde \eta_n(x)-\tilde \eta_n(z))^2}{|x-z|^{N+2\s}}dz+\int_{\R^N\setminus B_r  } \frac{\tilde \eta_n(x)^2}{|x-z|^{N+2\s}}dz.
\end{equation*}
For $y,\, z\in B_r$, we have that $ny,\, nz\in B_{m_0}$ and
$$
\frac{(\tilde \eta_n(y)-\tilde \eta_n(z))^2}{|y-z|^{N+2\s}}= \frac{(\tilde \eta_0(ny)-\tilde \eta_0(nz))^2}{|y-z|^{N+2\s}}\le \norm{\tilde \eta_0}_{C^1(\R_+)}  n^2 |y-z|^{2-N-2\s}.
$$
Consequently,
\begin{align*}
&\int_{B_r } \int_{B_r } \frac{(\tilde \eta_n(y)-\tilde \eta_n(z))^2}{|y-z|^{N+2\s}}dz  v^2(y)\,  dy  \le  \norm{\tilde \eta_0}_{C^1(\R_+)}  n^2 \int_{B_r } \int_{B_r  }|y-z|^{2-N-2\s}v^2(y)\, dz dy  \\
  &= c\norm{\tilde \eta_0}_{C^1(\R_+)} n^{2\s}\int_{B_r}v^2(y)\, dy
 \le c\norm{\tilde \eta_0}_{C^1(\R_+)} \left(\int_{B_r}|v|^{\frac{2N}{N-2\s}}(y)\, dy\right)^{\frac{N-2\s}{N}} \, \to \, 0\quad{\rm as}\ \, n\to+\infty.
\end{align*}
For $y \in B_r(0)$ and $z\in \R^N\setminus B_r(0)$,
 \begin{align*}
\int_{B_r} \int_{\R^N\setminus B_r} \frac{(\tilde \eta_n(y)-\tilde \eta_n(z))^2}{|y-z|^{N+2\s}}dz  v^2(y)\,  dy  &\le 
   \int_{B_r } \int_{\R^N\setminus  B_r }  \frac{ 1}{|y-z|^{N+2\s}}dz v^2(y)\, dz     dy  \\
  &= c n^{2\s}\int_{B_r}v^2(y)\, dy \,\to \, 0\qquad{\rm as}\ \, n\to+\infty.
\end{align*}

For $y \in\R^N\setminus B_r(0)$ and $z\in B_r(0)$,
 \begin{align*}
\int_{\R^N\setminus B_r } \int_{B_r } \frac{(\tilde \eta_n(y)-\tilde \eta_n(z))^2}{|y-z|^{N+2\s}}dz  v^2(y)\,  dy  &\le 
   \int_{B_r } \int_{\R^N\setminus  B_r }  \frac{ 1}{|y-z|^{N+2\s}}dz v^2(y)\, dz v^2(y)\,   dy  \\
  &= c n^{2\s}\int_{B_r(0)}v^2(y)\, dy \,\to \,0\quad{\rm as}\ \, n\to+\infty.
\end{align*}
For $y,\, z \in\R^N\setminus B_r(0)$, we have
$$
\int_{\R^N\setminus B_r(0) } \int_{\R^N\setminus B_r(0) } \frac{(\tilde \eta_n(y)-\tilde \eta_n(z))^2}{|y-z|^{N+2\s}}dz  v^2(y)\,  dy =0.
$$\smallskip
As a consequence, we have that
$$\int_{\R^N\times \R^N} \frac{(\tilde \eta_m(x)-\tilde \eta_m(y))^2}{|x-y|^{N+2\s}}v^2(y)\,dx dy\to  0\quad{\rm as}\ \, n\to+\infty,$$
which ends the proof.
\end{proof}

\subsection{Appendix C. Proof of Remark \ref{weak-solution-more-regular}(ii)}
\label{sec:appendix-c.-proof}

The aim of this appendix is to give the proof of Remark~\ref{weak-solution-more-regular}(ii), which we restate in the next proposition for the reader's convenience. 

\begin{proposition}
  \label{sec:appendix-c.-proof-1}
Let $\Omega \subset \R^N$ be a bounded domain of class $C^2$ with $0 \in \Omega$, let $f \in L^\infty_{loc}(\overline{\Omega} \setminus \{0\})$, and let $u \in L^\infty_{loc}(\R^N \setminus \{0\}) \cap L^1(\R^N,\frac{dx}{1+|x|^{N+2\s}})$ be a distributional solution of 
\begin{equation*}
  \left\{
    \begin{aligned}
\cL^\s_\mu u  = f  \quad  {\rm in}\ \  \Omega \setminus \{0\}, \\
 u  =0    \quad  {\rm in}\ \   \Omega^c.\qquad
\end{aligned}
\right.
\end{equation*}
Then $u \in W^\s_*(\Omega)$, and $u$ satisfies $\cL^\s_\mu u  = f$ in $\Omega$ in the sense of 
Definition \ref{weak-solution}. 
\end{proposition}

\begin{proof}
As before, we put $\Omega_\eps:= \Omega \setminus B_\eps(0)$ for $\eps>0$ small. For given $\eps>0$, we need to prove the following: 
\begin{enumerate}
\item[(i)] $u \in W^\s(\Omega_{\eps})$;
\item[(ii)] For every $v \in \cH^\s_0(\Omega_{\eps})$, we have 
  \begin{equation}
    \label{eq:claim-ii}
\cE_\mu^\s(u,v) = \int_{\Omega} f v\,dx.  
  \end{equation}
\end{enumerate}
Suppose for the moment that $(i)$ is already proved. Then for every $\xi \in C^\infty_c(\Omega_{\eps})$ we have 
$$
\cE_\mu^\s(u,\xi) = \int_{\Omega}u \cL^\s_\mu \xi \,dx = \int_{\Omega} f v\,dx. 
$$
Moreover, the map $\cH^\s_0(\Omega_{\eps}) \to \R$, $v \mapsto \cE_\mu^\s(u,v)$ is continuous, and $C^\infty_c(\Omega_{\eps})$ is dense in $\cH^\s_0(\Omega_{\eps})$ by \cite[Theorem 2.4]{DPV}. Hence (\ref{eq:claim-ii}) follows by approximation. 

It thus remains to prove $(i)$. For this we choose $\eps_0 \in (0,\eps)$, and we let $\eta \in C^\infty(\R^N)$ with $\eta \equiv 1$ in a neighborhood of $\Omega_{\eps}$ and $\eta \equiv 0$ on $B_{\eps_0}(0)$. It then suffices to show that $\tilde u:= \eta u \in \cH^\s_0(\Omega_{\eps_0})$. By essentially the same arguments as in the proof of Proposition~\ref{sec:dirichl-probl-bound-comp-corol-corol}, we see that $\tilde u \in L^\infty(\R^N)$ is a distributional solution of 
 \begin{equation}
\label{appendix-C-equation}
\left \{
\begin{aligned}
\cL^\s_\mu \tilde u  = \tilde f  \quad  {\rm in}\ \  \Omega_{\eps_0}, \\
  u  =0    \quad  {\rm in}\ \   \Omega_{\eps_0}^c \,
\end{aligned}
\right.
\end{equation}
with $\tilde f = \eta f + C_{N,\s}g$, where 
$$
g: \Omega \to \R, \qquad g(x)= p.v.\int_{\R^N} \frac{\eta(x)-\eta(y)}{|x-y|^{N+2\s}}u(y)\,dy. 
$$
By the same arguments as the proof of Proposition~\ref{sec:dirichl-probl-bound-comp-corol-corol}, using the a priori assumption that $u \in L^\infty_{loc}(\R^N \setminus \{0\})$, we find that $g \in L^\infty(\Omega)$ and therefore $\tilde f \in L^\infty(\Omega)$. 
Hence $\tilde u \in C^\s(\R^N)$ by the main result in \cite[Prop. 1.1]{RS}. On the other hand, (\ref{appendix-C-equation}) also admits a unique weak solution $u_* \in \cH^\s_0(\Omega_\eps)$ which also satisfies $u_* \in C^\s(\R^N)$ by \cite[Prop. 1.1]{RS}.
It now follows from \cite[Proposition 2.17]{Ls07} that $\tilde u = u_*$, and hence $\tilde u \in \cH^\s_0(\Omega_{\eps_0})$, as required.
\end{proof}

 \bigskip\bigskip

\noindent{\bf Acknowledgments:}    H. Chen  is supported by NNSF of China, No: 11726614, 11661045  and by the Alexander von Humboldt Foundation.    T. Weth is supported by DAAD and BMBF (Germany) within the project 57385104.

\end{document}